\documentclass[11pt,leqno]{amsart}
 \usepackage{ucs}
\usepackage{upgreek}
\usepackage{amsthm,amsfonts,amssymb,epsfig,graphics,cancel,esint,upgreek,
amsmath,eufrak,latexsym,enumerate,mathrsfs}
\usepackage[T1]{fontenc}
\usepackage{lmodern}
\usepackage{oldgerm}
\usepackage[latin1]{inputenc}\relax
\usepackage[active]{srcltx}
\usepackage[final]{pdfpages}
\usepackage{color}
\usepackage{setspace}
\definecolor{darkred}{rgb}{0.5,0,0}
\definecolor{darkgreen}{rgb}{0,0.5,0}
\definecolor{darkblue}{rgb}{0,0,0.5}
\usepackage{graphicx,subfigure}
\usepackage[colorlinks = true, linkcolor=darkblue,
filecolor=darkgreen,
urlcolor=darkred,
citecolor=darkgreen,pdftex]{hyperref}
\usepackage{url}
\usepackage[small]{caption}
\usepackage{wrapfig}
\usepackage{dsfont} 

\newcommand{\br}{\mathbb{R}}

\newcommand{\bz}{\mathbb{Z}}

\newcommand{\bn}{\mathbb{N}}

\newcommand{\bt}{\mathbb{T}}

\newcommand{\cc}{\mathcal{C}}
\newcommand{\ch}{\mathcal{H}}
\newcommand{\ci}{\mathcal{I}}
\newcommand{\ce}{\mathcal{E}}
\newcommand{\cj}{\mathcal{J}}
\newcommand{\dd}{\mathrm{d}}
\newcommand{\dL}{\mathrm{L}}

\newtheorem{theo}{Theorem}
\newtheorem{lem}{Lemma} [section]

\newtheorem{defi}{Definition} [section]

\numberwithin{equation}{section}
\newtheorem{defi-pro}{Definition-Proposition} [section]

\newcounter{numero}
\newcounter{numerob}
\newcounter{numerobb}
\usepackage{geometry}
\usepackage{amsmath}
\usepackage{amsfonts}
\usepackage{amsthm,amssymb}
\usepackage[latin1]{inputenc}  
\usepackage[T1]{fontenc}       
\usepackage{amssymb}

\title{Stable ground states for the HMF Poisson model}
\date{\today}
\author[M. Fontaine]{Marine Fontaine}
\address[Marine Fontaine]{IRMAR, ENS Rennes, CNRS, Université de Rennes 1, 
Campus de Beaulieu, 263 avenue du G\'en\'eral Leclerc CS 74205 
35042 Rennes Cedex\\FRANCE}
\email{marine.fontaine@ens-rennes.fr}
\author[M. Lemou]{Mohammed Lemou}
\address[Mohammed Lemou]{CNRS, IRMAR, Université de Rennes 1,
Campus de Beaulieu, 263 avenue du G\'en\'eral Leclerc CS 74205 
35042 Rennes Cedex\\FRANCE}
\email{mohammed.lemou@univ-rennes1.fr}
\author[F. M\'ehats]{Florian M\'ehats}
\address[Florian M\'ehats]{IRMAR, Université de Rennes 1, CNRS,
Campus de Beaulieu, 263 avenue du G\'en\'eral Leclerc CS 74205 
35042 Rennes Cedex\\FRANCE}
\email{florian.mehats@univ-rennes1.fr}

\begin{document}


\begin{abstract} 
In this paper we prove the nonlinear orbital stability of a large class of steady states solutions to the Hamiltonian Mean Field (HMF) system
with a Poisson interaction potential. These steady states are obtained as minimizers of an energy functional under {\em one, two or infinitely many constraints.}
The singularity of the Poisson potential  prevents from a direct run of the general strategy in \cite{Stability_VP,Lemou_seul} which was based on  generalized rearrangement techniques, and which has been recently extended to the case of  the usual (smooth) cosine potential \cite{Stability_HMFcos}.
Our strategy is rather based on variational techniques. However, due to the boundedness of the space domain, our variational problems do not enjoy  the usual scaling invariances which are, in general, very important in the analysis of variational problems. To replace these scaling arguments, we introduce new transformations which, although specific to our context, remain somehow  in the same spirit of rearrangements tools introduced in the references above. In particular, these transformations allow for the incorporation of an arbitrary number of constraints, and yield a stability result for a large class of steady states.
\end{abstract}

\maketitle


\section{Introduction and main results}
\subsection{The HMF Poisson model}
The Hamiltonian mean-field (HMF) model \cite{Messer1982, Antoni_Ruffo} describes the evolution of particles moving on a circle under the action of a given potential. The most popular model is the HMF system with an infinite range attractive cosine potential. Although this model has no direct physical relevance, it is commonly used in the physics literature as a toy model to describe some gravitational systems. In particular, it is involved in the study of non equilibrium phase transitions \cite{chavanis1, staniscia, antoniazzi, yama2}, of travelling clusters \cite{PhysRevE.79.036208, yama1} or of relaxation processes \cite{yamaguchi:hal-00008414, barre:hal-00018773, Chavanis-Vatteville-Bouchet}.  Many results exist concerning the stability of steady states solutions to the HMF system with a cosine potential. Some are about the dynamics of perturbations of inhomogeneous steady states \cite{BarreOlivettiYamaguchiDynamics,BarreOlivettiYamaguchidamping} and others deal with the linear stability of steady states \cite{chavanis1, ogawa, BarreYamaguchi}. In \cite{Stability_HMFcos}, the nonlinear stability of inhomogeneous steady states that satisfy an explicit criterion is proved. 
In the case of homogeneous (i.e. with dependence in velocity only) steady states and a cosine interaction potential , a nonlinear Landau damping analysis has been investigated for the HMF model in Sobolev spaces  \cite{faou_rousset}. 

There exist other kinds of potentials for the HMF model like the Poisson potential or the screened Poisson potential \cite{Chavanis-Delfini, Nagai}. In this paper, we study the orbital stability of ground states of a HMF model with a Poisson potential. This model is closer to the Vlasov-Poisson system than the HMF model with a cosine potential. The Poisson interaction potential is however more singular, which induces serious technical difficulties and prevent from a complete application of the strategy introduced in \cite{Stability_VP} for the Vlasov-Poisson system or in  \cite{Stability_HMFcos} for the HMF model with a cosine potential. For this reason, our analysis is based on variational methods. A general approach is introduced allowing to prove the nonlinear stability of a large class of steady states thanks to the study of variational problems with one, two or infinitely many constraints. Notice that, in our case, since the domain of the position is bounded and since the number of constraints may be infinite,  scaling arguments like in \cite{Manev, lemou-mehats-raphael} cannot be used. New transformations will be introduced to bypass these technical difficulties.

The HMF Poisson system reads
\begin{align}\label{HMF}
\begin{cases}
&\partial_t f+v\partial_{\theta}f-\partial_{\theta}\phi_{f}\partial_v f=0, \quad (t, \theta, v) \in \br_+\times\bt\times \br,\\
&f(t=0,\theta,v)=f_{init}(\theta,v)\geq 0,
\end{cases}
\end{align} 
where $\bt$ is the flat torus $\br/2\pi\bz$ and $f=f(t,\theta,v)$ is the nonnegative distribution function. The self-consistent potential $\phi_{f}$ associated to a distribution function $f$ is defined for $\theta \in \bt$ by 
\begin{equation}\label{Poisson}
\partial_\theta^2 \phi_f=\rho_f-\frac{\Vert f\Vert_{\dL^1}}{2\pi}, \qquad \rho_f(\theta)=\int_{\br}f(\theta,v)\dd v
\end{equation}
or, equivalently,
\begin{equation}\label{phi}
\phi_f(\theta)=\int_0^{2\pi}W(\theta-\tilde{\theta})\rho_f(\tilde{\theta})\dd\tilde{\theta},
\end{equation}
where the function $W$ is defined on $\br$ by 
\[W \text{ is } 2\pi\text{-periodic }, \quad \forall \theta\in[-\pi,\pi], \quad W(\theta)= -\frac{\theta^2}{4\pi}+\frac{\vert\theta\vert}{2}-\frac{\pi}{6}.\]
Note that $W$ has a zero average, is continuous on $\br$ and that $\phi_f$ is  $2\pi$-periodic with zero average : $\int_0^{2\pi}\phi_f(\theta)\dd\theta=0$.

Some quantites are invariant during the evolution:
\begin{itemize}
\item the Casimir functions: $\iint j(f(\theta,v))\dd\theta\dd v$, for any function $j\in \cc^1(\br_+)$ such that $j(0)=0$;
\item the nonlinear energy: 
\begin{equation}\label{Hamiltonien}
\ch(f)=\iint \frac{v^2}{2}f(\theta,v)\dd\theta\dd v-\frac{1}{2}\int_0^{2\pi}\phi_{f}'(\theta)^2\dd \theta;
\end{equation}
\item the total momentum: $\iint vf(\theta,v)\dd \theta\dd v$.
\end{itemize}
Moreover, the HMF system satisfies the Galilean invariance, that is, if $f(t, \theta,v)$ is a solution, then so is $f(t, \theta+v_0 t,v+v_0)$, for all $v_0\in \br$.\\

In Section \ref{1contrainte}, we prove the orbital stability of stationary states which are minimizers of a one-constraint variational problem. It is obtained for two kinds of steady states: the compactly supported ones and the  Maxwell-Boltzmann (non compactly supported) distributions \cite{Chavanis-Vatteville-Bouchet}. In Section \ref{2contraintes}, we prove the orbital stability of compactly supported steady states which are minimizers of a two constraints problem. In particular, this covers the case of compactly supported steady states which are minimizers of a one constraint problem. Lastly, in Section \ref{infini_contraintes}, we prove the orbital stability of the set of all the minimizers of a problem with an infinite number of constraints. This set of minimizers contains the minimizers of one and two constraints problems. However, at this stage, our strategy only provides a collective stability result (stability of the set of minimizers) for the minimizers of this problem with infinite number of constraints, instead of the individual stability of each minimizer which is only obtained for the one and two constraints variational problems.

\subsection{Statement of the results}\label{result}
\subsubsection{One-constraint problem}
First, in Section \ref{1contrainte}, we will show the orbital stability of stationary states which are minimizers of the following variational problem
\vspace{-0.2cm}
\begin{equation} \label{I(M)}
\mathcal{I}(M)=\underset{f\in E_j,\Vert f\Vert_{\mathrm{L}^1}=M}{\inf}\mathcal{H}(f)+\iint j(f(\theta,v))\mathrm{d}\theta\mathrm{d}v.
\end{equation}
The constant $M>0$ is given and $E_j$ is the energy space:
\begin{equation}\label{Ej1}
E_j=\left\{f\geq 0, \Vert (1+v^2) f\Vert_{\mathrm{L}^1}<+\infty, \left\vert\iint j(f(\theta,v))\mathrm{d}\theta\mathrm{d}v\right\vert<+\infty\right\},
\end{equation}
where $j: \mathbb{R}_+ \to \mathbb{R}$ is either the function defined by $j(t)=t\ln(t)$ for $t>0$ and $j(0)=0$ or a function $j$ satisfying the following assumptions
\begin{enumerate}
\item[(H1)] $j\in\mathcal{C}^2(\br_+^*)$; $j(0)=j'(0)=0$ and $j''(t)>0$ for all $t>0$,
\item[(H2)] $\underset{t\to+\infty}{\lim}\frac{j(t)}{t}=+\infty$,
\end{enumerate}
Note that $j(t)=t\ln(t)$ satisfies (H2) but not (H1) since $j'(0)\neq 0$ in this case. 
 
\begin{defi}
We shall say that a sequence $f_n$ converges to $f$ in $E_j$ and we shall write $f_n \xrightarrow{E_j} f$ if $\Vert (1+v^2)(f_n - f) \Vert_{\mathrm{L}^1}\underset{n\to+\infty}{\longrightarrow}0$ and $\iint j( f_n(\theta,v))\mathrm{d}\theta\,\mathrm{d}v \underset{n\to+\infty}{\longrightarrow}\iint j( f(\theta,v))\,\mathrm{d}\theta\mathrm{d}v$.
\end{defi}
In our first result, we establish the existence of ground states for the HMF Poisson model (\ref{HMF}) which are minimizers of the variational problem (\ref{I(M)}). This theorem will be proved in Section \ref{Proof1}.
\begin{theo}[Existence of ground states]\label{Existence_minimizers1} Let $j$ be the function $j(t)=t\ln(t)$ or a function satisfying (H1) and (H2). We have
\begin{enumerate}
\item In both cases, the infimum (\ref{I(M)}) exists and is achieved at a minimizer $f_0$ which is a steady state of (\ref{HMF}).
\item If $j$ satisfies (H1) and (H2), any minimizer $f_0$ of (\ref{I(M)}) is continuous, compactly supported,  piecewise $\cc^1$ and takes the form
\[f_0(\theta,v)=(j')^{-1}\left(\lambda_0-\frac{v^2}{2}-\phi_{f_0}(\theta)\right)_+ \text{ for some } \lambda_0\in \br.\]
The function $(.)_+$ is defined by $(x)_+=x$ if $x\geq 0$, 0 else.
\item If $j(t)=t\ln(t)$, any minimizer $f_0$ of (\ref{I(M)}) is a $\cc^{\infty}$ function which takes the form
\vspace{-0.2cm}
\[f_0(\theta,v)=\exp{\left(\lambda_0-\frac{v^2}{2}-\phi_{f_0}(\theta)\right)} \text{ for some } \lambda_0\in \br.\]
\end{enumerate}
\end{theo}

Our second result concerns the orbital stability of the above constructed ground states under the action of the HMF Poisson flow. But first and foremost, we need to prove the uniqueness of the minimizers under equimeasurability condition. To do that, first recall the definition of the equimeasurability of two functions.

\begin{defi}
Let $f_1$ and $f_2$ be two nonnegative functions in $\dL^1([0,2\pi]\times\br)$. The functions $f_1$ and $f_2$ are said equimeasurable, if and only if, $\mu_{f_1}=\mu_{f_2}$ where $\mu_f$ denotes the distribution function of $f$, defined by
\begin{equation}\label{mu_f}
\mu_f(s)=\vert\{(\theta,v)\in[0,2\pi]\times\br: f(\theta,v)>s\}\vert,\, \text{ for all }s\geq0,
\end{equation}
and $\vert A\vert$ stands for the Lebesgue measure of a set $A$.
\end{defi}
\begin{lem}[Uniqueness of the minimizer under equimeasurability condition]\label{Unicite1}
Let $f_1$ and $f_2$ be two equimeasurable steady states of (\ref{HMF}) which minimize (\ref{I(M)}) with $j(t)=t\ln(t)$ or with $j$ given by a function satisfying (H1) and (H2). Then the steady states $f_1$ and $f_2$ are equal up to a shift in $\theta$.
\end{lem} 

This lemma will be proved in Section \ref{Unicite 1}. Now, using the compactness of all the minimizing sequences of (\ref{I(M)}) (which will be obtained along the proof of Theorem \ref{Orbital_stability1} in Section \ref{OS1}) and the uniqueness result given by Lemma \ref{Unicite1}, we can get the following stability result. It will be proved in Section \ref{OS1}.

\begin{theo}[Orbital stability of ground states]\label{Orbital_stability1} Consider the variational problem (\ref{I(M)}) with $j(t)=t\ln(t)$ or with $j$ given by a function satisfying (H1) and (H2). In both cases, we have the following result. For all $M>0$, any steady state $f_0$ of (\ref{HMF}) which minimizes (\ref{I(M)}) is orbitally stable under the flow (\ref{HMF}). $\!\!\!$ More precisely for all $\varepsilon\!>\!0$, there exists $\eta(\varepsilon)\!>\!0$ such that the following holds true. Consider $f_{init}\in E_j$  satisfying $\Vert (1+v^2)( f_{init}-f_0)\Vert_{\dL^1}<\eta(\varepsilon)$ and $\vert\iint j(f_{init})-\iint j(f_0)\vert < \eta(\varepsilon)$. Let $f(t)$ be a weak global solution to (\ref{HMF}) on $\br_+$ with initial data $f_{init}$ such that the Casimir functions are preserved during the evolution and that $\ch(f(t))\leq \ch(f_{init})$. Then there exists a translation shift $\theta(.)$ with values in $[0,2\pi]$ such that $\forall t\in\br_+^*$, we have
\[\Vert(1+v^2) (f(t,\theta+\theta(t),v)-f_0(\theta,v))\Vert_{\dL^1}<\varepsilon.\]
\end{theo}

\subsubsection{Two-constraints problem}
In Section \ref{2contraintes}, we will show the orbital stability of stationary states which are minimizers of the following variational problem 
\vspace{-0.1cm}
\begin{equation}\label{I(M1,Mj)}
\ci(M_1,M_j)=\underset{\underset{\Vert f\Vert_{\mathrm{L}^1}=M_1, \Vert j(f) \Vert_{\mathrm{L}^1}=M_j }{f\in E_j}}{\inf}\mathcal{H}(f)
\end{equation}
where $E_j$ is the same energy space as above and the function $j$ satisfies (H1) and (H2) together with the following additional assumption

\begin{itemize}
\item[(H3)] There exist  $p, q >1$ such that $p\leq \frac{tj'(t)}{j(t)}\leq q,$  for t>0,
\end{itemize} 
Note that $j$ is a nonnegative function. The first result of this part is the following theorem which will be proved in Section \ref{Proof2}.

\begin{theo}[Existence of ground states]\label{Existence2} Let $j$ be a function satisfying (H1), (H2) and (H3). We have
\begin{enumerate}
\item $\!\!\!\!$ The infimum (\ref{I(M1,Mj)}) exists and is achieved at a minimizer $f_0$ which is a steady state of (\ref{HMF});
\item $\!\!\!\!$ Any steady state $f_0$ obtained as a minimizer of (\ref{I(M1,Mj)}) is continuous, compactly supported, piecewise $\cc^1$ and takes the form
\begin{equation}\label{forme}
f_0(\theta,v)=(j')^{-1}\left(\frac{\frac{v^2}{2}+\phi_{f_0}(\theta)-\lambda_0}{\mu_0}\right)_+ \text{ where }(\lambda_0,\mu_0)\in \br\times\br_-^* ;
\end{equation}
\item $\!\!\!\!$ The associated density $\rho_{f_0}\!$ is $\!\!\!$ continuous and the associated potential $\phi_{f_0}$ is $\cc^2$ on $\bt$.
\end{enumerate}
\end{theo}
Since the existence of ground states is established, the natural second result is the uniqueness of these ground states. For the two constraints cases, we are only able to obtain a local uniqueness for the ground states under equimeasurability condition. A steady state $f$ will be said to be homogeneous if $\phi_f = 0$ and inhomogeneous is $\phi_f\neq 0$. We have the  following lemma which will be proved in Section \ref{Uniqueness2}.

\sloppy \begin{lem}[Local uniqueness of the minimizer under equimeasurability condition] \label{Unicite2}Let $f_0\in E_j$ be a steady state of (\ref{HMF}) and a minimizer of (\ref{I(M1,Mj)}). It can be written in the form (\ref{forme}) with  $(\lambda_0, \mu_0)\in\br\times\br_-^*$. We have the following cases:
\begin{itemize}
\item[$\bullet$] $f_0$ is a homogeneous steady state. Then it is the only steady state minimizer of (\ref{I(M1,Mj)}) under equimeasurability condition. 
\item[$\bullet$]$f_0$ is an inhomogeneous steady states.Then, there exists $\delta_0>0$ such that for all $f\in E_j$ inhomogeneous steady state of (\ref{HMF}) and minimizer of (\ref{I(M1,Mj)}) equimeasurable to $f_0$ which can be written as (\ref{forme}) with $(\lambda,\mu)\in\br\times\br_-^*$, we have 
\begin{itemize}
\item either $\mu_0\neq \mu$ and $\vert \vert\mu_0\vert-\vert\mu\vert\vert>\delta_0$,
\item either $\mu_0=\mu$ and $f_0=f$ up to a translation shift in $\theta$.
\end{itemize}
\end{itemize}
\end{lem}
Then, similarly to the one-constraint problem, we will show the following result concerning the orbital stability of the ground states under the action of the HMF Poisson flow. It will be proved in Section \ref{implication et compacite}.

\begin{theo}[Orbital stability of ground states]\label{Orbital_stability2} Let $M_1, M_j>0$. Then any steady state $f_0$ of (\ref{HMF}) which minimizes (\ref{I(M1,Mj)}) is orbitally stable under the flow (\ref{HMF}). It means that given $\varepsilon>0$, there exists $\eta(\varepsilon)>0$ such that the following holds true. Consider $f_{init}\in E_j$ with $\Vert (1+v^2) (f_{init}-f_0)\Vert_{\dL^1}<\eta(\varepsilon)$ and with $\left\vert \iint j(f_{init}) -\iint j(f_0)\right\vert<\eta(\varepsilon)$. Let $f(t)$ be a weak global solution to (\ref{HMF}) on $\br_+$ with initial data $f_{init}$ such that the Casimir functions are preserved during the evolution and that $\ch(f(t))\leq \ch(f_{init})$. Then there exists a translation shift $\theta(.)$ with values in $[0,2\pi]$ such that $\forall t\in\br_+^*$, we have
\[\Vert (1+v^2)(f(t,\theta+\theta(t),v)-f_0(\theta,v))\Vert_{\dL^1}<\varepsilon.\]
\end{theo}
\subsubsection{Infinite number of constraints problem}
Finally, in Section \ref{infini_contraintes}, we will show the orbital stability of stationary states which are minimizers of a problem with an infinite number of constraints. In this Section, the energy space is the following
\begin{equation}
\ce=\{f\geq 0, \Vert (1+v^2)f\Vert_{\dL^1}<+\infty, \Vert  f\Vert_{\dL^{\infty}}<+\infty\}.
\end{equation}
Let $f_0\in \ce\cap\mathcal{C}^0([0,2\pi]\times\mathbb{R})$. We will denote by  $Eq(f_0)$ the set of equimeasurable functions to $f_0$. The variational problem is
\begin{equation}\label{H0}
H_0= \underset{f\in Eq(f_0), f\in \ce}{\inf}\ch (f).
\end{equation}
This is a variational problem with infinitely  many constraints since the equimeasurability condition on $f$ is  
equivalent to say that $f$ has the same casimirs as $f_0$: 
$\|j(f)\|_{L^1} = \|j(f_0)\|_{L^1}$, $\forall j$.
\begin{defi}
We shall say that a sequence $f_n$ converges to $f$ in $\mathcal{E}$ and we shall write $ f_n \xrightarrow{\ce}f$ if $(f_n)_n$ is uniformly bounded and satisfies $\Vert (1+v^2)(f_n - f) \Vert_{\mathrm{L}^1}\underset{n\to +\infty}{\longrightarrow}0$.
\end{defi}
We start by showing in Section \ref{existence_limite_faible} the existence of ground states for the HMF Poisson model (\ref{HMF}) which are minimizers of the variational problem (\ref{H0}).
\begin{theo}[Existence of ground states]\label{Existence3}
The infimum (\ref{H0}) is finite and is achieved at a minimizer $\bar{f}\in\ce$ which is a steady state of (\ref{HMF}). 
\end{theo}

Our second result concerns the orbital stability of the above constructed ground states under the action of the HMF flow. As we do not have the uniqueness of the minimizers under constraint of equimeasurablility, we can just get the orbital stability of the set of minimizers and not the orbital stability of each minimizer. It will be proved in Section \ref{OS3}.

\begin{theo}[Orbital stability of ground states]\label{Stability3}
Let $f_0\in\ce\cap\cc^0([0,2\pi]\times\br)$.Then the set of steady states of (\ref{HMF}) which minimize (\ref{H0}) is orbitally stable under the flow (\ref{HMF}). More precisely given $f_{i_0}$ minimizer of (\ref{H0}), for all $\varepsilon>0$, there exists $\eta(\varepsilon)>0$ such that the following holds true. Consider $f_{init}\in \ce$ with $\Vert (1+v^2)(f_{init}-f_{i_0})\Vert_{\dL^1}<\eta(\varepsilon)$. Let $f(t)$ a weak global solution to (\ref{HMF}) on $\br^+$ with initial data $f_{init}$ such that the Casimir functions are preserved during the evolution and that $\ch(f(t))\leq \ch(f_{init})$. Then there exist $f_{i_1}$ minimizer of (\ref{H0}) and a translation shift $\theta(.)$ with values in $[0,2\pi]$ such that $\forall t\in\br_+^*$, we have
\[\Vert (1+v^2)( f(t,\theta+\theta(t),v)-f_{i_1}(\theta,v))\Vert_{\dL^1}<\varepsilon.\]
\end{theo}

\section{Minimization problem with one constraint} \label{1contrainte}
\subsection{Existence of ground states}\label{Partie_existence1}
This section is devoted to the proof of Theorem \ref{Existence_minimizers1}. 
\subsubsection{Properties of the infimum}
For convenience, we set for $f\in E_j$, the below functional
\begin{align}\label{J}
J(f)=\mathcal{H}(f)+\!\!\iint j(f)=\iint \frac{v^2}{2}f(\theta,v)\mathrm{d}\theta\mathrm{d}v-\frac{1}{2}\int_0^{2\pi}\phi'_f(\theta)^2\mathrm{d}\theta + \iint j(f(\theta,v))\mathrm{d}\theta\mathrm{d}v.
\end{align}

\begin{lem}\label{Inf_fini_compacite1}
The variational problem (\ref{I(M)}) satisfies the following statements.
\begin{enumerate}
\item Let $j$ be a function satisfying (H1) and (H2) or $j(t)=t\ln(t)$, in both cases, the infimum (\ref{I(M)}) exists i.e $\ci(M)>-\infty$ for all $M>0$.
\item For any minimizing sequence $(f_n)_n$ of the variational problem (\ref{I(M)}), we have the following properties:
\begin{enumerate}
\item The minimizing sequence $(f_n)_n$ is weakly compact in $\dL^1([0,2\pi]\times\br)$ i.e. there exists $\bar{f}\in\dL^1([0,2\pi]\times\br)$ such that $f_n\underset{n\to+\infty}{\rightharpoonup}\bar{f}$ weakly in $\dL^1$.
\item We have $\Vert\phi_{f_n}-\phi_{\bar{f}}\Vert_{H^1}\underset{n\to+\infty}{\longrightarrow}0.$
\end{enumerate}
\end{enumerate}
\end{lem}

\begin{proof}
Let us start with the proof of item (1). Let $f\in E_j$ such that $\Vert f\Vert_{\dL^1}=M$. If $j$ satisfies (H1) and (H2), then $j$ is nonnegative and we have 
\vspace{-0.1cm}
\[J(f)\geq -\frac{1}{2} \int_0^{2\pi}\phi_f'(\theta)^2\mathrm{d}\theta\geq -\pi\Vert W'\Vert_{\dL^{\infty}}^2 M^2 \]
and this term is finite for $f\in E_j$. Note that
\vspace{-0.1cm}
\begin{equation}\label{Norme_infini_phi'}
\Vert \phi'_f\Vert_{\dL^{\infty}}\leq \Vert W'\Vert_{\dL^{\infty}}\Vert f\Vert_{\dL^1}.
\end{equation} 
If $j(t)=t\ln(t)$, the sign of $j$ is not constant and we have to bound from below the term $\iint j(f(\theta,v))\mathrm{d}\theta\mathrm{d}v$. With Jensen's inequality and the convexity of $t\to -\ln(t)$, we get
\begin{equation}\label{jensen}
\iint f\ln\left(\frac{f}{f_1}\right)\geq \left(\iint f\right)\left[\ln\left(\iint f\right)-\ln\left(\iint f_1\right)\right].
\end{equation}
Taking $f_1(\theta,v)=e^{-\frac{v^2}{2}}$ and let $C_1=\ln\left(\iint f_1\right)$, we obtain
 \begin{equation}\label{minoration_J}
 J(f)\geq -\frac{1}{2} \int_0^{2\pi}\phi_f'^2(\theta)\mathrm{d}\theta +M[\ln(M)-C_1]\geq -\pi\Vert W'\Vert_{\dL^{\infty}}^2 M^2 +M[\ln(M)-C_1].
 \end{equation}
Each term is finite for $f\in E_j$. Thus $\mathcal{I}(M)$ exists for both functions $j$.

Then let us continue with the proof of item (2). Let $(f_n)_n$ be a minimizing sequence of (\ref{I(M)}). By the Dunford-Pettis theorem (see \cite{Dunford-Pettis}), if $\Vert f_n \Vert_{\dL^1}$, $\Vert v^2 f_n\Vert_{\dL^1}$ and $\iint j(f_n(\theta,v))\mathrm{d}\theta\mathrm{d}v$ are bounded from above, the sequence of functions $(f_n)_n$ is weakly compact in $\dL^1$. Notice that the domain in $\theta$ is bounded thus contrary to the Vlasov-Poisson system, there is no loss of mass at the inifinity in $\theta$ and $v$. Let us show that $\Vert v^2 f_n\Vert_{\dL^1}$ is bounded. We have from equality (\ref{J})
\vspace{-0.1cm}
\[\Vert v^2 f_n\Vert_{\dL^1} = 2J(f_n)+\int_0^{2\pi}\phi_{f_n}'(\theta)^2\mathrm{d}\theta-2\iint j(f_n(\theta,v)\mathrm{d}\theta\mathrm{d}v.\]
If j satisfies the hypotheses (H1) and (H2), this equality becomes
\vspace{-0.1cm}
\[\Vert v^ 2 _n\Vert_{\dL^1}\leq 2J(f_n)+2\pi\Vert W'\Vert_{\dL^{\infty}}^2 M^2.\]
Since $J(f_n)$ is bounded, we deduce in this case that $\Vert v^2 f_n\Vert_{\dL^1}$ is bounded. If $j(t)=t\ln(t)$, we have
\vspace{-0.1cm}
\begin{align*}
\Vert v^2f_n\Vert_{\dL^1}&\leq 2J(f_n)+2\pi\Vert W'\Vert_{\dL^{\infty}}^2 M^2 -2\iint f_n(\theta,v)\ln(f_n(\theta,v))\mathrm{d}\theta\mathrm{d}v,\\
&\leq 2J(f_n)+2\pi\Vert W'\Vert_{\dL^{\infty}}^2 M^2-2M[\ln(M)-C_1]+\frac{1}{2}\Vert v^2 f_n\Vert_{\dL^1}
\end{align*}
using Jensen's inequality (\ref{jensen}) with $ f_1(\theta,v)=e^{-\frac{v^2}{4}}$ and $C_1=\ln(\iint f_1)$. Thus 
\vspace{-0.1cm}
\[\Vert v^2f_n\Vert_{\dL^1}\leq 4J(f_n)+4\pi\Vert W'\Vert_{\dL^{\infty}}^2 M^2-4M[\ln(M)-C_1]\] 
and this quantity is bounded. Let us then show that $\iint j(f_n(\theta,v))\mathrm{d}\theta\mathrm{d}v$ is bounded from above. Let $j$ be a function satisfying (H1) and (H2) or $j(t)=t\ln(t)$, we have

\[\iint j(f_n(\theta,v))\mathrm{d}\theta\mathrm{d}v \leq J(f_n)+\pi\Vert W'\Vert_{\dL^{\infty}}^2 M^2.\]
Each term of this inequality is bounded, therefore this quantity is bounded. Hence by Dunford-Pettis theorem, there exists $\bar{f}\in \dL^1$ such that $f_n \underset{n\to +\infty}{\rightharpoonup}\bar{f}$ in $\dL^1_w$. This concludes the proof of item (1) of Lemma \ref{Inf_fini_compacite1}. Then, let us prove the last result.  Since 
\[\phi_{f_n}(\theta)-\phi_{\bar{f}}(\theta)=\int_{\mathbb{R}}\int_0^{2\pi}W(\theta-\tilde{\theta})[f_n(\tilde{\theta},v)-\bar{f}(\tilde{\theta},v)]\mathrm{d}\tilde{\theta}\mathrm{d}v,\]
and
\[\phi_{f_n}'(\theta)-\phi_{\bar{f}}'(\theta)=\int_{\mathbb{R}}\int_0^{2\pi}W'(\theta-\tilde{\theta})[f_n(\tilde{\theta},v)-\bar{f}(\tilde{\theta},v)]\mathrm{d}\tilde{\theta}\mathrm{d}v,\]
we immediately deduce applying dominated convergence and from the weak convergence of $f_n$ in $\dL^1([0,2\pi]\times\br)$ that $\Vert\phi_{f_n}-\phi_{\bar{f}}\Vert_{H^1}\underset{n\to+\infty}{\longrightarrow}0$. 
\end{proof}

\noindent The following lemma is the analogous for $j(t)=t\ln(t)$ of a well-known result about the lower semicontinuity properties of convex nonnegative functions see \cite{Kavian}. The proof is not a direct consequence of the lower semicontinuity properties of convex positive functions since $j(t)=t\ln(t)$ changes sign on $\br_+$. It will be detailed in the appendix.
\begin{lem}\label{sci_bis}
Let $(f_n)_n$ be a sequence of nonnegative functions converging weakly in $\dL^1$ to $\bar{f}$ such that $\Vert f_n\Vert_{\dL^1}=M$, $\Vert v^2 f_n\Vert_{\dL^1}\leq C_1$ and $\vert\iint f_n\ln(f_n)\vert\leq C_2$ where $M$, $C_1$ and $C_2$ do not depend on $n$, we have the following inequality
\[\iint \bar{f}\ln(\bar{f})\dd\theta\dd v\leq \underset{n\to +\infty}{\liminf}\iint f_n\ln(f_n)\dd\theta\dd v.\] 
\end{lem}


\subsubsection{Proof of Theorem \ref{Existence_minimizers1}}\label{Proof1}
We are now ready to prove Theorem \ref{Existence_minimizers1}.\\

\textbf{Step 1} Existence of a minimizer.\\

\noindent Let $M>0$. From item (1) of Lemma \ref{Inf_fini_compacite1}, we know that $\ci(M)$ is finite for functions $j$ satisfying (H1) and (H2) or $j(t)=t\ln(t)$.  Let us show that there exists a function $\bar{f}\in E_j$ which minimizes the variational problem (\ref{I(M)}). Let  $(f_n)_n\in E_j ^{\mathbb{N}}$ be a minimizing sequence of $\mathcal{I}(M)$. Thus $J(f_n)\underset{n\to+\infty}{\longrightarrow}\mathcal{I}(M)$ and $\Vert f_n\Vert_{\mathrm{L}^1}=M$ where $J$ is defined by (\ref{J}). From item (2) of Lemma \ref{Inf_fini_compacite1}, we know that there exists $\bar{f}\in \dL^1([0,2\pi]\times\br)$ such that $f_n \underset{n\to +\infty}{\rightharpoonup}\bar{f}$ weakly in $\dL^1([0,2\pi]\times\br)$. The $\dL^1$-weak convergence implies $\Vert \bar{f}\Vert_{\dL^1}=M$ and $\bar{f}\geq 0$ a.e. In the case where $j$ satisfies (H1) and (H2), from lower semicontinuity properties of nonnegative convex functions (see \cite{Kavian}) and from item (b) of Lemma \ref{Inf_fini_compacite1}, we get $\bar{f} \in E_j$. For $j(t)=t\ln(t)$, from lower semicontinuity properties of nonnegative convex functions and item (b) of Lemma \ref{Inf_fini_compacite1}, we get $\Vert v^2\bar{f}\Vert_{\dL^1}<+\infty$ and from Lemma \ref{sci_bis} and item (b) of Lemma \ref{Inf_fini_compacite1}, we get $\iint \bar{f}\ln(\bar{f})<+\infty$. Using Jensen's inequality (\ref{jensen}) with $f_1(\theta,v)=e^{-\frac{v^2}{2}}$, we get
\[M(\ln(M)-C_1)-\iint \frac{v^2}{2}\bar{f}\dd\theta\dd v\leq \iint \bar{f}\ln(\bar{f})\dd\theta\dd v,\]
and we conclude that $\vert \iint j(f(\theta,v))\dd\theta\dd v\vert<+\infty$ and that $\bar{f}\in E_j$. Therefore, in both cases, we have $\ci(M)\leq J(\bar{f})$. Moreover from item (2) of Lemma \ref{Inf_fini_compacite1} and classical inequalities about the lower semicontinuity properties of convex nonnegative functions see \cite{Kavian} for $j$ satisfying (H1) and (H2) and Lemma \ref{sci_bis} for $j(t)=t\ln(t)$, we have the followings inequalities:
\vspace{-0.1cm}
\[\mathcal{I}(M)=\underset{n\to+\infty}{\lim} J(f_n)\geq\iint \frac{v^2}{2}\bar{f}(\theta,v)\mathrm{d}\theta\mathrm{d}v-\frac{1}{2}\int_0^{2\pi}\phi'_{\bar{f}}(\theta)^2\mathrm{d}\theta+\iint j(\bar{f}(\theta,v))\mathrm{d}\theta\mathrm{d}v.\]
Thus $\mathcal{I}(M)\geq J(\bar{f})$.  To recap, we have proved that $\mathcal{I}(M)=J(\bar{f})$ with $\bar{f}\in E_j$ and $\Vert \bar{f} \Vert_{\dL^1}=M$ thus $\ci(M)$ is achieved.\\

\textbf{Step 2} Euler-Lagrange equation for the minimizers.\\

\noindent Let $M>0$ and $\bar{f}$ be a minimizer of $\mathcal{I}(M)$, let us write Euler-Lagrange equations satisfied by $\bar{f}$. For this purpose, for any given potential $\phi$, we introduce a new distribution function $F^{\phi}$ having mass $M$ and displaying nice monotonicity property for the energy-Casimir functional.

\begin{lem}\label{F^phi1}
Let $j$ be a function verifying (H1) and (H2) or $j(t)=t\ln(t)$ and let $M >0$. For all $\phi:[0,2\pi]\longrightarrow\br$  continuous function, there exists a unique $\lambda\in ]\min\phi, +\infty[$ for $j$ satisfying (H1), (H2) and $\lambda\in\br$ for $j(t)\!=\!t\ln(t)$ such that the function $F^{\phi}\!\!:\!\![0,2\pi]\!\times\br\longrightarrow\br_+$ defined by
\begin{equation}\label{Fphi_expression1}
\begin{cases}
&F^{\phi}(\theta,v)=(j')^{-1}\left(\lambda-\frac{v^2}{2}-\phi(\theta)\right)_+ \text{ for } j \text{ satisfying (H1), (H2)}\\ 
& F^{\phi}(\theta,v)=\exp\left(\lambda-\frac{v^2}{2}-\phi(\theta)\right) \text{ for } j(t)=t\ln(t),
\end{cases}
\end{equation}
satisfies $\Vert F^{\phi}\Vert_{\dL^1}=M$.
\end{lem}

\begin{proof}
Let $\lambda\in\br$, we define
\vspace{-0.1cm}
\begin{equation}
\begin{cases}
&\!\!\!\!\!K(\lambda)=\int_0^{2\pi}\int_{\br}(j')^{-1}\left(\lambda-\frac{v^2}{2}-\phi(\theta)\right)_+\dd\theta\dd v \text{ for } j \text{ satisfying (H1), (H2) }\\ 
&\!\!\!\!\!K(\lambda)=\int_0^{2\pi}\int_{\br}\exp\left(\lambda-\frac{v^2}{2}-\phi(\theta)\right)\dd\theta\dd v \text{ for }j(t)=t\ln(t).
\end{cases}
\end{equation}
Since in both cases, $j$ is strictly convex and $\left\vert\left\{\frac{v^2}{2}+\phi(\theta)<\lambda\right\}\right\vert$ is strictly increasing in $\lambda$, the map $K$ is strictly increasing on $[\min \phi,+\infty[$ for $j$ satisfying (H1), (H2) and on $\br$ for $j(t)=t\ln(t)$. Note that for $j$ satisfying (H1), (H2), $K(\lambda)=0$ for $\lambda \leq \min \phi$, then we have the following limit: $\underset{\lambda\to \min \phi}{\lim}K(\lambda)=0$ by using the monotone convergence theorem. For $j(t)=t\ln(t)$, we have $\underset{\lambda\to -\infty}{\lim}K(\lambda)=0$. For both functions, we have $\underset{\lambda\to +\infty}{\lim}K(\lambda)=+\infty$ by using Fatou's lemma. Hence, there exists a unique $\lambda$ such that $\Vert F^{\phi}\Vert_{\dL^1}=M$. 
\end{proof}

We introduce a second problem of minimization, we set $M>0$. Let $j(t)=t\ln(t)$ or $j$ given by a function satisfying (H1) and (H2).
\begin{equation}\label{J1}
\cj_0=\underset{\int_0^{2\pi}\phi=0}{\inf}\cj(\phi) \text{ where } \cj(\phi)=\iint\left(\frac{v^2}{2}+\phi(\theta)\right)F^{\phi}(\theta,v)\mathrm{d}\theta\mathrm{d}v+\frac{1}{2}\int_0^{2\pi}\phi'(\theta)^2\mathrm{d}\theta+\iint j(F^{\phi}),
\end{equation}
where $F^{\phi}$ is defined by Lemma \ref{F^phi1}.

\begin{lem}\label{inequality_relou1}
We have the following inequalities:
\begin{enumerate}
\item For all $\phi \!\in\! H^2([0,2\pi])$ such that $\phi(0)=\phi(2\pi)$ and $\int_0^{2\pi}\!\phi\!=\!0$, we have $J(F^{\phi})\!\leq \!\cj(\phi)$.
\item For all $f\in E_j$ with $\Vert f\Vert_{\dL^1}=M_1$, we have 
\vspace{-0.2cm}
 \[\ci(M)\leq J(F^{\phi_f})\leq \cj(\phi_f)\!\leq\! J(f).\] 
Besides $\ci(M)=\cj_0$.
\end{enumerate}
\end{lem}

\begin{proof}
First we will show item (1) of this lemma. Let $\phi \in H^2([0,2\pi])$ such that $\phi(0)=\phi(2\pi)$ and $\int_0^{2\pi}\phi=0$, we have
\begin{align*}
\cj(\phi)&=J(F^{\phi})-\frac{1}{2}\Vert \phi'_{F^{\phi}}\Vert_{\dL^2}^2+\frac{1}{2}\Vert \phi'\Vert_{\dL^2}^2+\iint (\phi(\theta)-\phi_{F^{\phi}}(\theta))F^{\phi}(\theta,v)\dd \theta\dd v\\
&=J(F^{\phi})-\frac{1}{2}\Vert \phi'_{F^{\phi}}\Vert_{\dL^2}^2+\frac{1}{2}\Vert \phi'\Vert_{\dL^2}^2+\int_0^{2\pi}(\phi-\phi_{F^{\phi}})(\phi''_{F^{\phi}}+\frac{\Vert F^{\phi}\Vert_{\dL^1}}{2\pi})\dd\theta,
\end{align*}
since $\phi_{F^{\phi}}$ satisfies the Poisson equation (\ref{Poisson}). Then, after integrating by parts and gathering the terms, we get
\vspace{-0.4cm}
\begin{equation}\label{norme2}
\cj(\phi)=J(F^{\phi})+\frac{1}{2}\Vert \phi'_{F^{\phi}}-\phi'\Vert_{\dL^2}^2.
\end{equation}
Hence $\cj(\phi)\geq J(F^{\phi})$. Then, let us show the right inequality of item (2). Let $f\in E_j$ such that $\Vert f\Vert_{\dL^1}=M$. Using $\Vert F^{\phi}\Vert_{\dL^1}=M$, using the equality (\ref{Hamiltonien}), the functional can be written as
\begin{align*}
J(f)&=\cj(\phi_f)+\iint\left(\frac{v^2}{2}+\phi_f(\theta)\right)(f(\theta,v)-F^{\phi_f}(\theta,v))\dd\theta\dd v+\iint j(f)-\iint j(F^{\phi})\\
&=\cj(\phi_f)+\iint(\lambda -j'(F^{\phi_f}))(f(\theta,v)-F^{\phi_f}(\theta,v))\dd\theta\dd v+\iint j(f)-\iint j(F^{\phi}).
\end{align*}
\vspace{-0.3cm}
We get
\vspace{-0.1cm}
\begin{equation}\label{convexity1}
J(f)=\cj(\phi_f)+\iint (j(f)-j(F^{\phi_f})-j'(F^{\phi_f})(f-F^{\phi_f}))\dd\theta\dd v.
\end{equation}
The convexity of $j$ gives us the desired inequality. The others inequalities are straightforward.
\end{proof}

We are now ready to get Euler-Lagrange equations. According to Lemma \ref{inequality_relou1}, if $\bar{f}$ is a minimizer of $\ci(M)$, $\bar{\phi}:=\phi_{\bar{f}}$ is a minimizer of $\cj_0$ and $J(\bar{f})=\cj(\bar{\phi})$. Using (\ref{convexity1}), we get 
\[\iint (j(\bar{f})-j(F^{\bar{\phi}})-j'(F^{\bar{\phi}})(\bar{f}-F^{\bar{\phi}}))\dd\theta\dd v=0.\]
Then writting the Taylor's formula for the function $j(\bar{f})$ and integrating over $[0,2\pi]\times\br$, we get
\vspace{-0.1cm}
\[\iint (\bar{f}-F^{\bar{\phi}})^2\int_0^1(1-u)j''(u(\bar{f}-F^{\bar{\phi}})+F^{\bar{\phi}})\mathrm{d}u=\iint j(\bar{f})-\iint j(F^{\bar{\phi}})-\iint (\bar{f}-F^{\bar{\phi}})j'(F^{\bar{\phi}}).\]
Thus $\iint (\bar{f}-F^{\bar{\phi}})^2\int_0^1(1-u)j''(u(\bar{f}-F^{\bar{\phi}})+F^{\bar{\phi}})\mathrm{d}u\dd\theta\dd v=0$. As $j''>0$, we deduce that $\bar{f}=F^{\bar{\phi}}$. Hence, in the case where $j$ satisfies (H1) and (H2), the minimizer $\bar{f}$ has the following expression
\vspace{-0.2cm}
\[\bar{f}(\theta,v)=(j')^{-1}\left(\bar{\lambda}-\frac{v^2}{2}-\phi_{\bar{f}}(\theta)\right)_+ \text{ where } \bar{\lambda}\in\br.\] 
In the case where $j(t)=t\ln(t)$, we have 
\[\bar{f}(\theta,v)=\exp\left(\bar{\lambda}-\frac{v^2}{2}-\phi_{\bar{f}}(\theta)\right),\quad \text{ where } \bar{\lambda}\in\br.\]
Notice that in the case of $j$ satisfying (H1) and (H2), the minimizer is continuous, piecewise $\cc^1$ and compactly supported in $v$. In the case of $j(t)=t\ln(t)$, $\bar{f}$ is a function of class $\cc^{\infty}$. We have shown that any minimizer of (\ref{I(M)}) takes the above form and is at least piecewise $\cc^1$ thus clearly any minimizer is a steady state of (\ref{HMF}). The proof of Theorem \ref{Existence_minimizers1} is complete.

\subsection{Orbital stability of the ground states}
To prove the orbital stability result stated in Theorem \ref{Orbital_stability1}, we first need to prove the uniqueness of the minimizers under equimeasurability condition.
\subsubsection{Uniqueness of the minimizers under equimeasurability condition}\label{Unicite 1}

This section is devoted to the proof of Lemma \ref{Unicite1}. Let $f_1$ and $f_2$ be two equimeasurable minimizers of $\mathcal{I}(M)$. In the case where $j$ satisfies (H1) and (H2), they have the following expressions
\[f_1(\theta,v)=(j')^{-1}\left(\lambda_1-\frac{v^2}{2}-\phi_{f_1}(\theta)\right)_+,\quad
f_2(\theta,v)=(j')^{-1}\left(\lambda_2-\frac{v^2}{2}-\phi_{f_2}(\theta)\right)_+.\]
In the case where $j(t)=t\ln(t)$, they have the following expressions
\vspace{-0.1cm}
\[f_1(\theta,v)=\exp\left(\lambda_1-\frac{v^2}{2}-\phi_{f_1}(\theta)\right),\quad f_2(\theta,v)=\exp\left(\lambda_2-\frac{v^2}{2}-\phi_{f_2}(\theta)\right).\]
They can be written in the form
\begin{equation}\label{psi}
f_1(\theta,v)=G\left(\frac{v^2}{2}+\psi_1(\theta)\right), \quad f_2(\theta,v)=G\left(\frac{v^2}{2}+\psi_2(\theta)\right);
\end{equation}
where $G(t)=(j')^{-1}((-t)_+)$ or $G(t)=\exp(-t)$ with $\psi_i(\theta)=\phi_{f_i}(\theta)-\lambda_i$. In both cases, $G$ is a continuous, strictly decreasing and piecewise $\cc^1$ function. The functions $f_1$ and $f_2$ are equimeasurable so $\Vert f_1\Vert_{\dL^{\infty}}=\Vert f_2\Vert_{\dL^{\infty}}$. Since $G$ is a decreasing function, this means that $G(\min \psi_1)=G(\min \psi_2)$. Besides, $G$ being strictly decreasing and continuous on $\br$, it is one-to-one from $\br$ to $\br_+$ then $\min \psi_1=\min \psi_2=\alpha$. Thus, there exist $\theta_1$ and $\theta_2$ such that 
\vspace{-0.1cm}
\[\psi_1(\theta_1)=\psi_2(\theta_2)=\alpha,\quad \psi'_1(\theta_1)=\psi'_2(\theta_2)=0.\]
Therefore, $\psi_i$ satisfies 
\vspace{-0.3cm}
\begin{align*}
\begin{cases}
&\Psi''(\theta)=\mathcal{G}(\Psi(\theta)),\\
&\Psi'(\theta_i)=0,\\
&\Psi(\theta_i)=\psi_1(\theta_1)=\psi_2(\theta_2)=\alpha,
\end{cases}
\end{align*}
for $i=1$ or $2$ and where $\mathcal{G}(e)=\int_{\mathbb{R}}G(\frac{v^2}{2}+e)\mathrm{d}v-\frac{M}{2\pi}$. In both cases, $\mathcal{G}$ is locally Lipschitz thus according to Cauchy-Lipschitz theorem, $\psi_1=\psi_2$ up to the translation shift $\theta_2-\theta_1$. From (\ref{psi}), we get $f_1=f_2$ up to a translation shift in $\theta$.

\subsubsection{Proof of Theorem \ref{Orbital_stability1}}\label{OS1}
We will prove the orbital stability of steady states of (\ref{HMF}) which are minimizers of (\ref{I(M)}) in two steps. First, we will assume that all minimizing sequences of $\ci(M)$ are compact and deduce that all minimizer is orbitally stable. Then, we will show the compactness of all minimizing sequence.\\

\textbf{Step 1} Proof of the orbital stability\\

\noindent Assume that all minimizing sequences are compact. Let us argue by contradiction. Let $f_0$ be a minimizer and assume that $f_0$ is orbitally unstable. Then there exist $\varepsilon_0>0$, a sequence $(f_{init}^n)_n\in E_j^{\mathbb{N}}$ and a sequence $(t_n)_n\in \mathbb{R}^{+}_*$ such that $\underset{n\to+\infty}{\lim}\Vert (1+v^2)( f_{init}^n-f_0)\Vert_{\dL^1}=0$ and $\underset{n\to+\infty}{\lim}\left\vert \iint j(f_{init}^n)-\iint j(f_0)\right\vert =0$ and for all $n$, for all $\theta_0\in [0,2\pi]$ 
\vspace{-0.1cm}
\begin{align}\label{eq2}
\begin{cases}
&\Vert f^n(t_n,\theta+\theta_0,v)-f_0(\theta,v)\Vert_{\dL^1}>\varepsilon_0,\\
&\text{or }\Vert v^2(f^n(t_n,\theta+\theta_0,v)-f_0(\theta,v))\Vert_{\dL^1}>\varepsilon_0.
\end{cases}
\end{align}
where $f^n(t_n,\theta,v)$ is a solution to (\ref{HMF}) with initial data $f_{init}^n$. Let $g_n(\theta,v)=f^n(t_n,\theta,v)$, we have $J(g_n)-J(f_0)\leq J(f_{init}^n)-J(f_0) \underset{n\to +\infty}{\longrightarrow} 0$ since the system (\ref{HMF}) preservs the Casimir functionals and $\ch(f^n(t_n))\leq \ch(f_{init}^n)$. Introduce $\tilde{g_n}(\theta,v)=g_n(\theta,\frac{v}{\lambda_n})$ with $\lambda_n=\frac{M}{\Vert g_n\Vert_{\dL^1}}$. This function $\tilde{g_n}$ satisfies $\Vert \tilde{g_n}\Vert_{\dL^1}=M$, thus $0\leq J(\tilde{g_n})-J(f_0)$. Notice that
\vspace{-0.1cm}
\[J(f_0)\leq J(\tilde{g_n})\leq \lambda_n[(\lambda_n^2-1)\iint\frac{v^2}{2}g_n(\theta,v)\mathrm{d}\theta\mathrm{d}v-\frac{\lambda_n-1}{2}\int_0^{2\pi}\phi_{g_n}'^2(\theta)\mathrm{d}\theta+J(f_{init}^n)].\]
It is clear that  $\lambda_n\underset{n\to+\infty}{\longrightarrow}1$. Moreover using inequality (\ref{Norme_infini_phi'}), we show that $\left(\int_0^{2\pi}\phi_{g_n}'^2(\theta)\mathrm{d}\theta\right)_n$ is a bounded sequence. Then, arguing as in the proof of item (2) of Lemma \ref{Inf_fini_compacite1}, we get $\left(\Vert v^2g_n\Vert_{\dL^1}\right)_n$  is bounded sequence. Thus, $J(f_0)\leq \underset{n\to +\infty}{\lim} J(\tilde{g_n})\leq J(f_0)$. Hence $(\tilde{g_n})_n$ is a minimizing sequence of $\mathcal{I}(M)$. According to our assumption, it is a compact sequence in $E_j$: there exists $\tilde{g}\in E_j$ such that, up to an extraction of a subsequence, we have 
\begin{equation}\label{eq1}
\Vert g_n-\tilde{g}\Vert_{\dL^1}\underset{n\to +\infty}{\longrightarrow}0, \quad \Vert v^2(g_n-\tilde{g})\Vert_{\dL^1} \underset{n\to +\infty}{\longrightarrow}0,\quad \left\vert\iint j(g_n)-\iint j(\tilde{g})\right\vert\underset{n\to +\infty}{\longrightarrow}0.
\end{equation}
According to the conservation properties of HMF Poisson system, we have 
\[\vert\{(\theta,v)\in[0,2\pi]\times\mathbb{R},g_n(\theta,v)>t\}\vert=\vert\{(\theta,v)\in[0,2\pi]\times\mathbb{R}, f_{init}^n(\theta,v)>t\}\vert.\]
Let $\varepsilon>0$, we notice that $\forall\, 0<t<\varepsilon$ 
\vspace{-0.2cm}
\begin{align*}
\begin{cases}
&\{g_n>t\}\subset\left\{\{\vert g_n-\tilde{g}\vert<\varepsilon\}\cap\{\tilde{g}>t-\varepsilon\}\right\}\cup\{\vert g_n-\tilde{g}\vert\geq \varepsilon\},\\
&\{g_n>t\}\supset\{\vert g_n-\tilde{g}\vert <\varepsilon\}\cap\{\tilde{g}>t+\varepsilon\}.
\end{cases}
\end{align*}
Passing to the limit, we get 
\[\underset{n\to+\infty}{\limsup}\vert\{g_n>t\}\vert \leq \vert\{\tilde{g}>t-\varepsilon\}\vert,\qquad \underset{n\to+\infty}{\liminf}\vert\{g_n>t\}\vert\geq \vert\{ \tilde{g}>t+\varepsilon\}\vert.\]
Then we pass to the limit as $\varepsilon\to 0$ and we get up to an extraction of a subsequence; 
\[\underset{n\to+\infty}{\lim}\vert\{g_n>t\}\vert=\vert\{\tilde{g}>t\}\vert\quad \text{ for almost all } t>0.\] 
In the same way, we obtain up to an extraction of a subsequence
\[\underset{n\to+\infty}{\lim}\vert\{f_{init}^n>t\}\vert=\vert\{f_0>t\}\vert\quad \text{ for almost all } t>0.\] 
Noticing that the functions $t\to\vert\{f_0>t\}\vert$ and $t\to\vert\{\tilde{g}>t\}\vert$ are right-continuous, we get
\vspace{-0.1cm}
\[\vert\{f_0>t\}\vert=\vert\{\tilde{g}>t\}\vert,\quad \forall t\geq 0.\]
Thus $f_0$ and $g$ are two equimeasurable minimizers of $\mathcal{I}(M)$ but according to the previous uniqueness result stated in Lemma \ref{Unicite1}, $f_0=\tilde{g}$ up to a translation shift. To conclude, (\ref{eq1}) contradicts (\ref{eq2}) and we have proved that $f_0$ is orbitally stable.\\

\textbf{Step 2} Compactness of the minimizing sequences\\

\noindent Let $j$ satisfying (H1) and (H2) or $j(t)=t\ln(t)$. Let $(f_n)_n$ be a minimizing sequence of $\mathcal{I}(M)$. Let us show that $(f_n)_n$ is compact in $E_j$ i.e. that there exists $f_0 \in E_j$ such that $\underset{n\to+\infty}{\lim}\Vert (1+v^2)(f_n-f_0)\Vert_{\dL^1}=0$ and $\underset{n\to+\infty}{\lim}\left\vert \iint j(f_{init}^n)-\iint j(f_0)\right\vert =0$ up to an extraction of a subsequence. Arguing as before in Section \ref{Proof1}, there exists $f_0\in E_j$ such that $\Vert f_0\Vert_{\dL^1}=M$, $f_n \underset{n\to+\infty}{\rightharpoonup}f_0$ in $\dL^1_w$ up to an extraction of a subsequence and $J(f_0)=\mathcal{I}(M)$. From this last equality and the strong convergence in $\dL^2$ of the potential established in item (b) of Lemma \ref{Inf_fini_compacite1}, we deduce that
\vspace{-0.1cm}
\begin{equation}\label{somme}
\underset{n\to+\infty}{\lim}\left(\iint \frac{v^2}{2}f_n(\theta,v)\mathrm{d}\theta\mathrm{d}v+\iint j(f_n)\right)=\iint \frac{v^2}{2}f_0(\theta,v)\mathrm{d}\theta\mathrm{d}v+\!\iint j(f_0).
\end{equation}
From equality (\ref{somme}), from lower semicontinuity properties of nonnegative convex functions (see \cite{Kavian}) and from Lemma \ref{sci_bis}, we get
\vspace{-0.2cm}
\begin{equation}\label{lim1}
\iint j(f_n)\underset{n\to+\infty}{\longrightarrow}\iint j(f_0),\quad \text{ and } \quad \iint \frac{v^2}{2}f_n(\theta,v)\dd\theta\dd v\underset{n\to+\infty}{\longrightarrow}\iint \frac{v^2}{2}f_0(\theta,v)\dd\theta\dd v.
\end{equation}
There remains to show that $\Vert v^2(f_n-f_0)\Vert_{\dL^1}\underset{n\to+\infty}{\longrightarrow}0$ and $\Vert f_n-f_0\Vert_{\dL^1}\underset{n\to+\infty}{\longrightarrow}0$.\\

In the case of $j(t)=t\ln(t)$, the Csiszar-Kullback's inequality, see \cite{Csiszar_Kullback}, gives us the strong convergence in $\dL^1([0,2\pi]\times\br)$.  In our case, this Csiszar-Kullback's inequality writes
\vspace{-0.1cm}
\begin{equation}\label{CK}
\Vert f_n-f_0\Vert_{\dL^1}^2\leq 2M\iint f_n\ln\left(\frac{f_n}{f_0}\right).
\end{equation}
Hence, to prove the strong convergence in $\dL^1([0,2\pi]\times\br)$, it is sufficient to prove that 
\vspace{-0.1cm}
\[\iint f_n\ln\left(\frac{f_n}{f_0}\right)\dd\theta\dd v\underset{n\to+\infty}{\longrightarrow}0.\]

Since $f_0(\theta,v)=\exp\left(\lambda_0-\frac{v^2}{2}-\phi_{f_0}(\theta)\right)$, we have
\vspace{-0.1cm}
\begin{equation}\label{CKK}
\iint\!\! f_n\ln\left(\frac{f_n}{f_0}\right)\dd\theta\dd v\!=\!J(f_n)-J(f_0)+\frac{1}{2}(\Vert \phi_{f_n}'\Vert_{\dL^2}^2-\Vert \phi_{f_0}'\Vert_{\dL^2}^2)+\!\!\iint \!\!\phi_{f_0}(f_n-f_0).
\end{equation}
Note that
\begin{enumerate}
\item $J(f_n)-J(f_0)\underset{n\to+\infty}{\longrightarrow}0$ since $(f_n)_n$ is a minimizing sequence of $\ci(M)$,
\item $\Vert \phi_{f_n}'\Vert_{\dL^2}^2-\Vert \phi_{f_0}'\Vert_{\dL^2}^2\underset{n\to+\infty}{\longrightarrow}0$ since of the strong convergence in $\dL^2([0,2\pi]\times\br)$ of the potential established in item (b) of Lemma \ref{Inf_fini_compacite1},
\item $\iint \phi_{f_0}(\theta)(f_n(\theta,v)-f_0(\theta,v))\dd\theta\dd v\underset{n\to+\infty}{\longrightarrow}0$ since of the weak convergence of $f_n$ to $f_0$ in $\dL^1([0,2\pi]\times\br)$.
\end{enumerate}
Hence with (\ref{CK}) and (\ref{CKK}), we get $\Vert f_n-f_0\Vert_{\dL^1}\underset{n\to+\infty}{\longrightarrow}0$. From this strong convergence in $\dL^1([0,2\pi]\times\br)$, we deduce the a.e. convergence of $f_n$ and with Brezis-Lieb's lemma, and the second limit in (\ref{lim1}), we get the strong convergence of $v^2f_n$ in $\dL^1([0,2\pi]\times\br)$. Hence the sequence $(f_n)_n$ is compact in $E_j$.\\

In the case of $j$ satisfying (H1) and (H2), we again use Brezis-Lieb's lemma, see \cite{Brezis_Lieb}, to get the strong convergence of $f_n$ in $\dL^1$. We already have that $\Vert f_n\Vert_{\dL^1}\underset{n\to+\infty}{\longrightarrow}\Vert f_0\Vert_{\dL^1}$. Hence, with Brezis-Lieb's lemma, it is sufficient to show that $f_n\underset{n\to+\infty}{\longrightarrow}f_0$ a.e. Writing  the Taylor formula for the function $j(f_n)$ and integrating over $[0,2\pi]\times\mathbb{R}$, we get
\vspace{-0.1cm}
\begin{equation}\label{TL2}
\iint \!(f_n-f_0)^2\!\!\int_0^1\!(1\!-\!u)j''(u(f_n\!-\!f_0)\!+\!f_0)\mathrm{d}u=\iint\! j(f_n)-\!\!\iint \!j(f_0)-\!\!\iint\! (f_n-f_0)j'(f_0).
\end{equation}
Note also that
\begin{enumerate}
\item $\iint j(f_n)\underset{n\to+\infty}{\longrightarrow}\iint j(f_0)$,
\item $\iint j'(f_0)(f_n-f_0)\!\!\underset{n\to+\infty}{\longrightarrow}\!\!0$ since $f_n\!\!\underset{n\to+\infty}{\rightharpoonup}\!\!f_0$ $\dL^1_w$. Note that  $j'(f_0)\in \dL^{\infty}$ since $f_0 \in \dL^{\infty}$.
\end{enumerate}
Hence with Fubini-Tonelli 's theorem, we get
\vspace{-0.1cm}
\[\iint (f_n-f_0)^2j''((f_n-f_0)u+f_0)\underset{n\to+\infty}{\longrightarrow}0 \text{ for almost all } u\in[0,1].\] 
Let $u_0\in [0,1]$ such that $\iint (f_n-f_0)^2j''((f_n-f_0)u_0+f_0)\underset{n\to+\infty}{\longrightarrow}0$. Up to an extraction of a subsequence, we have
\vspace{-0.1cm}
\[(f_n-f_0)^2j''((f_n-f_0)u_0+f_0)\underset{n\to+\infty}{\longrightarrow}0 \text{ for almost all } (\theta,v)\in [0,2\pi]\times\mathbb{R}.\] 
This means there exists $\Omega_{u_0}$ such that $\vert \Omega_{u_0}\vert =0$ and $\forall (\theta,v)\in [0,2\pi]\times\mathbb{R}\setminus\Omega_{u_0}$, 
\begin{equation}\label{valadh}
(f_n(\theta,v)-f_0(\theta,v))^2j''(u_0(f_n(\theta,v)-f_0(\theta,v))+f_0(\theta,v))\underset{n\to+\infty}{\longrightarrow}0.
\end{equation}
Let us show that, up to a subsequence, $f_n(\theta,v)\underset{n\to+\infty}{\longrightarrow}f_0(\theta,v)$ for $(\theta,v)\in [0,2\pi]\times\mathbb{R}\setminus\Omega_{u_0}$. If $u_0=0$, we directly have the wanted convergence. Then let $u_0\in ]0,1]$ and let $l(\theta,v)$ be a limit point of $(f_n(\theta,v))_n$. Assume that $l(\theta,v)\neq f_0(\theta,v)$.
\begin{itemize}
\item[$\bullet$] First case: $l(\theta,v)<+\infty$. As $j''$ is continous and $j''>0$, we have
\begin{align*}
(f_n(\theta,v)-f_0(\theta,v))^2&j''(u_0(f_n(\theta,v)-f_0(\theta,v))+f_0(\theta,v))\\
&\underset{n\to+\infty}{\longrightarrow}(l(\theta,v)-f_0(\theta,v))^2j''(u_0(l(\theta,v)-f_0(\theta,v))+f_0(\theta,v))>0.
\end{align*}
This contradicts (\ref{valadh}).
\item[$\bullet$] Second case: $l(\theta,v)=+\infty$. Thus:
\begin{equation}\label{secondcas}
(f_n(\theta,v)-f_0(\theta,v))^2\underset{n\to+\infty}{\longrightarrow}+\infty\,\text{ and } \,
u_0(f_n(\theta,v)-f_0(\theta,v))+f_0(\theta,v)\underset{n\to+\infty}{\longrightarrow}+\infty.
\end{equation}
However the hypothesis (H2) implies that $t^2j''(t)$ does not converge to $0$ when $t$ goes to infinity. Indeed, arguing by contradiction, integrating twice over $[x_0,x]$ and taking the limit for $x\to +\infty$, we get
\vspace{-0.4cm}
\[\forall \varepsilon>0,\, \exists M>0, \text{ such that } \forall x>M,\, 0\leq \frac{j(x)}{x}\leq\frac{\varepsilon}{x_0}+j'(x_0).\]
This inequality contradicts (H2) then $t^2j''(t)$ does not converge to $0$ when $t$ goes to infinity and (\ref{secondcas}) contradicts (\ref{valadh}).
\end{itemize}
Hence $f_n\underset{n\to+\infty}{\longrightarrow}f_0$ a.e and we conclude using the Brezis-Lieb's lemma. The minimizing sequence is compact in $E_j$.\\

\section{Problem with two constraints} \label{2contraintes}

\subsection{Toolbox for the two constraints problem}
In this section, we define a new function denoted by $F^{\phi}$. Note that the function $F^{\phi}$ of (\ref{Fphi_expression}) differs from the one of Section \ref{Proof1}. However it can be seen as an equivalent of (\ref{Fphi_expression1}) in the sense that both functions $F^{\phi}$ satisfy the constraints of the one and two constraints problem respectively. There will be no possible confusion since the function $F^{\phi}$ of Section \ref{Proof1} will no longer be used. First, thank to this new function, the existence of minimizers is shown. Indeed the sequence $(F^{\phi_{f_n}})_n$ has better compactness properties than the sequence $(f_n)_n$. Then, we get the compactness of the sequence $(f_n)_n$ via the sequence $(F^{\phi_{f_n}})_n$ thanks to monotonicity properties of $\ch$ with respect to the transformation $F^{\phi}$. These properties will be detailed in Lemma \ref{inequality_relou}. More precisely, we have the following lemma:
\begin{lem}\label{F^phi}
Let $j$ be a function verifying (H1), (H2) and (H3) and let $M_1, M_j >0$. For all $\phi:[0,2\pi]\longrightarrow\br$  continuous function, there exists a unique pair $(\lambda,\mu)\in \br\times\br_-^*$ such that the function $F^{\phi}:[0,2\pi]\times\br\longrightarrow\br_+$ defined by
\begin{equation}\label{Fphi_expression}
F^{\phi}(\theta,v)=(j')^{-1}\left(\frac{\frac{v^2}{2}+\phi(\theta)-\lambda}{\mu}\right)_+ \text{satisfies } \Vert F^{\phi}\Vert_{\dL^1}=M_1, \,  \Vert j(F^{\phi})\Vert_{\dL^1}=M_j.
\end{equation}
\end{lem}
\begin{proof}
Let $(\lambda,\mu)\in\br\times\br_-^*$, we define
\vspace{-0.1cm}
\[K(\lambda,\mu)=\int_0^{2\pi}\int_{\br}(j')^{-1}\left(\frac{\frac{v^2}{2}+\phi(\theta)-\lambda}{\mu}\right)_+\dd\theta\dd v.\] 
We set $\mu\in\br_-^*$, since $j$ is strict convex and $\left\vert\left\{\frac{v^2}{2}+\phi(\theta)<\lambda\right\}\right\vert$ is strictly increasing in $\lambda$, the map $\lambda\to K(\lambda,\mu)$ is strictly increasing on $[\min \phi,+\infty[$. Note that $K(\lambda,\mu)=0$ for $\lambda\leq \min \phi$. We also have the following limits:$\underset{\lambda\to \min \phi}{\lim}K(\lambda,\mu)\!=\!0$ using the monotone convergence theorem and $\underset{\lambda\to +\infty}{\lim}K(\lambda,\mu)=+\infty$ using Fatou's lemma. Therefore, there exists a unique $\lambda=\lambda(\mu)\in ]\min \phi,+\infty[$ such that $\Vert F^{\phi}\Vert_{\dL^1}=M_1$.  We now define the map:
\vspace{-0.2cm}
\begin{align*}
G: \begin{cases}&\br_-^*\longrightarrow \br_+\\
&\mu\to \int_0^{2\pi}\int_{\br}j\circ(j')^{-1}\left(\frac{\frac{v^2}{2}+\phi(\theta)-\lambda(\mu)}{\mu}\right)_+\dd\theta\dd v.
\end{cases}
\end{align*}
Our purpose is to show that $G$ is continuous, strictly increasing on $\br_-^*$ and that $\underset{\mu\to-\infty}{\lim} G(\mu)=0$ and $\underset{\mu\to 0}{\lim} \, G(\mu)=+\infty$. This claim would imply that there exists a unique $\mu\in\br_-^*$ such that $G(\mu)=M_j$ and the proof of the lemma will be ended.

To get the monotony of $G$ and the continuity of $\lambda$ on $\br_-^*$, we first have to show the decrease of $\lambda$. Since $K(\lambda(\mu),\mu)\!\!=\!\!M_1$, using that both functions $\lambda\! \mapsto\! \!K(\lambda, \mu)$ and $\mu\!\mapsto \!\!K(\lambda, \mu)$ are increasing, we get that the map $\lambda$ is nonincreasing on $\br_-^*$.
According to the definition of $G$, it is sufficient to show that $\mu\to\lambda(\mu)$ is continuous on $\br_-^*$ to get the continuity of $G$ on $\br_-^*$. To prove the continuity of $\lambda$, we argue by contradiction. Assume that $\mu\to\lambda(\mu)$ is discontinous at $\mu_0<0$. Assume on the one hand that $\lambda$ is left-discontinous, ie there exist $\varepsilon_0>0$ and an increasing sequence $(\mu_n)_n\in (\br_-^*)^{\bn}$ converging to $\mu_0$ such that $\vert \lambda(\mu_n)-\lambda(\mu_0)\vert>\varepsilon_0$. $\lambda$ being nonincreasing and $j$ being convex, we get 
\[M_1\geq K(\lambda(\mu_0)+\varepsilon_0, \mu_n).\] 
Applying Fatou's lemma, we have \[K(\lambda(\mu_0)+\varepsilon_0, \mu_n)\geq K(\lambda(\mu_0)+\varepsilon_0,\mu_0).\] 
Since $K(\lambda(\mu_0)+\varepsilon_0,\mu_0)>M_1$, we get a contradiction and $\lambda$ is left-continuous. On the other hand, assume that $\lambda$ is right-discontinuous at $\mu_0<0$, ie there exist $\varepsilon_0>0$ and a decreasing sequence $(\mu_n)_n\in (\br_-^*)^{\bn}$ converging to $\mu_0$ such that $\vert \lambda(\mu_n)-\lambda(\mu_0)\vert>\varepsilon_0$. $\lambda$ being nonincreasing and $j$ being convex, we get 
\[M_1 \leq K(\lambda(\mu_0)-\varepsilon_0, \mu_n).\] 
Using a generalization of the Beppo Levi's theorem for the decreasing functions, we get 
\[K(\lambda(\mu_0)-\varepsilon_0, \mu_n)\leq K(\lambda(\mu_0)-\varepsilon_0,\mu_0).\] 
Since $K(\lambda(\mu_0)-\varepsilon_0,\mu_0)<M_1$, we get a contradiction and $\lambda$ is right-continuous. We conclude that the map $\lambda$ is continuous on $\br_-^*$. Let us show the increase of $G$. Before that, notice that $K(\lambda, \mu)$ can be written as 
\begin{equation}\label{new_expression_K}
K(\lambda, \mu)=2\sqrt{2}\int_0^{2\pi}\int_0^{+\infty}\frac{1}{j''\circ (j')^{-1}(t)}\sqrt{(\mu t+\lambda-\phi(\theta))_+}\dd t\dd\theta,
\end{equation}
by performing a change of variables: $t=\frac{\frac{v^2}{2}+\phi(\theta)-\lambda}{\mu}$ and an integration by parts. By doing the exact same thing for $G$, we can also write
\begin{equation}\label{new_expression_G}
G(\mu)=2\sqrt{2}\int_0^{2\pi}\int_0^{+\infty}\frac{t}{j''\circ (j')^{-1}(t)}\sqrt{(\mu t+\lambda(\mu)-\phi(\theta))_+}\dd t\dd\theta.
\end{equation}
Let $\mu_1, \mu_2\in \br_-^*$ be such that $\mu_1\neq\mu_2$. Thanks to the previous step, there exists for $i=1,2$, $\lambda_i:=\lambda(\mu_i)\in]\min\phi,+\infty[$ such that $K(\lambda_i,\mu_i)=M_1$. Hence, by using the equality (\ref{new_expression_K}) and by setting for $i=1,2$, $A_{\mu_i}:=\mu_i t +\lambda_i-\phi(\theta)$, we get
\begin{equation}\label{new_expression_dif_K}
K(\lambda_1,\mu_1)-K(\lambda_2,\mu_2)=2\sqrt{2}\int_0^{2\pi}\int_0^{+\infty}\frac{1}{j''\circ (j')^{-1}(t)}[(A_{\mu_1})_+^{\frac{1}{2}}-(A_{\mu_2})_+^{\frac{1}{2}}]\dd t\dd\theta=0.
\end{equation}
Then, by using (\ref{new_expression_G}) and (\ref{new_expression_dif_K}), we have for all $C\in\br$
\[G(\mu_1)-G(\mu_2)=2\sqrt{2}\int_0^{2\pi}\int_0^{+\infty}\frac{t+C}{j''\circ (j')^{-1}(t)}[(A_{\mu_1})_+^{\frac{1}{2}}-(A_{\mu_2})_+^{\frac{1}{2}}]\dd t\dd\theta.\]
We set $C_0:=\frac{\lambda_1-\lambda_2}{\mu_1-\mu_2}$ and we get
\begin{equation}
(\mu_1-\mu_2)(G(\mu_1)-G(\mu_2))=2\sqrt{2}\int_0^{2\pi}\int_0^{+\infty}\frac{(A_{\mu_1}-A_{\mu_2})}{j''\circ (j')^{-1}(t)}[(A_{\mu_1})_+^{\frac{1}{2}}-(A_{\mu_2})_+^{\frac{1}{2}}]\dd t\dd\theta.
\end{equation}
Since the function $t\mapsto (t)_+^{\frac{1}{2}}$ is nondecreasing, we have $(A_{\mu_1}-A_{\mu_2})[(A_{\mu_1})_+^{\frac{1}{2}}-(A_{\mu_2})_+^{\frac{1}{2}}]\geq 0$. Hence $G$ is a  nondecreasing function. We now notice that $(A_{\mu_1}-A_{\mu_2})[(A_{\mu_1})_+^{\frac{1}{2}}-(A_{\mu_2})_+^{\frac{1}{2}}]>0$ for $\theta \in \{\phi< \lambda_1\}$ and $t\in \,]0,\frac{\phi(\theta)-\lambda_1}{\mu_1}[$. Besides the measure of the set $\{\phi < \lambda_1\}$ is strictly positive because $\lambda_1>\min\phi$. Thus, the function $G$ is  strictly increasing on $\br_-^*$.

It remains to compute the limits of G. First let us prove that $\underset{\mu\to -\infty}{\lim}\lambda(\mu)=+\infty$. The function $\lambda$ being nonincreasing, $\underset{\mu\to -\infty}{\lim}\lambda(\mu)$ exists and we denote it by $\lambda_{\infty}$. Assume that $\lambda_{\infty}<\infty$.  We have
\vspace{-0.2cm}
\[M_1=K(\lambda(\mu),\mu)\leq K(\lambda_{\infty},\mu)\underset{\mu\to -\infty}{\longrightarrow}0.\]
This is a contradiction then $\underset{\mu\to -\infty}{\lim}\lambda(\mu)=+\infty$. Then let us prove that $\underset{\mu\to 0^-}{\lim}\lambda(\mu)=\min\phi$. $\lambda$ being nonincreasing, $\underset{\mu\to 0^-}{\lim}\lambda(\mu)$ exists and we denote it by $\lambda_0$. We have to deal with three cases. First, notice that (H2) and (H3) imply $\underset{t\to+\infty}{\lim}(j')^{-1}(t)=+\infty$, then we get
\begin{align*}
\begin{cases}
&\text{if }\lambda_0>\min \phi\,: \, M_1=K(\lambda(\mu),\mu)>K(\lambda_0,\mu)\underset{\mu\to 0^-}{\longrightarrow}+\infty, \quad \text{ applying Fatou's lemma, } \\
&\text{if }\lambda_0<\min\phi\, :\, M_1=K(\lambda(\mu),\mu)< K(\frac{\min\phi+\lambda_0}{2},\mu)=0 \quad \text{ since } \frac{\min\phi+\lambda_0}{2}<\min\phi.
\end{cases}
\end{align*}
Hence only the third case can occur ie $\underset{\mu\to 0^-}{\lim}\lambda(\mu)=\min\phi$.\\

Let us continue with the computation of $\underset{\mu\to 0^-}{\lim}G(\mu)$. Performing the change of variables: $u=\frac{v}{\sqrt{2(\lambda(\mu)-\phi(\theta))_+}}$, we get
\[G(\mu)=2\sqrt{2}\int_0^{2\pi}\int_0^1\sqrt{(\lambda(\mu)-\phi(\theta))_+}j\circ(j')^{-1}\left(\frac{(\lambda(\mu)-\phi(\theta))_+}{\vert\mu\vert}(1-u^2)\right)\dd\theta\dd u\]
and
\vspace{-0.3cm}
\begin{equation}\label{M1}
M_1=2\sqrt{2}\int_0^{2\pi}\int_0^1\sqrt{(\lambda(\mu)-\phi(\theta))_+}(j')^{-1}\left(\frac{(\lambda(\mu)-\phi(\theta))_+}{\vert\mu\vert}(1-u^2)\right)\dd\theta\dd u.
\end{equation}
Then applying Jensen's inequality to the convex function $j$, we obtain
\[j\left(\frac{M_1}{\int_0^{2\pi}2\sqrt{2}\sqrt{(\lambda(\mu)-\phi(\theta))_+}\dd\theta}\right)_+\leq \frac{G(\mu)}{\int_0^{2\pi}2\sqrt{2}\sqrt{(\lambda(\mu)-\phi(\theta))_+}\dd \theta}.\]
Hence 
\vspace{-0.2cm}
\begin{equation}\label{lambda(mu)}
G(\mu)\geq \dfrac{j\left(\frac{M_1}{\alpha(\mu)}\right)}{\frac{M_1}{\alpha(\mu)}}M_1 \text{ with } \alpha(\mu)=2\sqrt{2}\int_0^{2\pi}\sqrt{(\lambda(\mu)-\phi(\theta))_+}\dd \theta.
\end{equation}
Using the dominated convergence theorem, we show that $\alpha(\mu)\underset{\mu\to 0^-}{\longrightarrow}0$. But $j$ satisfies (H2) therefore
\vspace{-0.4cm}
\[\frac{j\left(\frac{M_1}{\alpha(\mu)}\right)}{\frac{M_1}{\alpha(\mu)}}\underset{\mu\to 0^-}{\longrightarrow}+\infty \text{ and } \underset{\mu\to 0^-}{\lim}G(\mu)=+\infty.\] 
Let us continue with the computation of $\underset{\mu\to -\infty}{\lim}G(\mu)$. The hypothesis (H3) implies the following inequality:
\vspace{-0.3cm}
\begin{equation}\label{inegalite}
\frac{t(j')^{-1}(t)}{q}\leq j\circ (j')^{-1}(t)\leq \frac{t(j')^{-1}(t)}{p}.
\end{equation}
Thanks to (\ref{inegalite}), we can estimate
\vspace{-0.2cm}
\begin{equation}\label{truc2}
0\leq G(\mu)\leq\frac{M_1}{p}\frac{(\lambda(\mu)-\min\phi)_+}{\vert \mu\vert}
\end{equation}
Let us show that $\frac{M_1}{p}\frac{(\lambda(\mu)-\min\phi)_+}{\vert \mu\vert}\underset{\mu\to -\infty}{\longrightarrow}0$. Using the expression of $M_1$ given by (\ref{M1}), we get
\vspace{-0.2cm}
\[M_1\geq\sqrt{(\lambda(\mu)-\max\phi)_+}4\pi\sqrt{2}\int_0^1(j')^{-1}\left(\frac{(\lambda(\mu)-\max\phi)_+}{\vert \mu\vert}(1-u^2)\right)\dd u\geq 0.\]
For $\vert\mu\vert$ sufficiently large, we have $(\lambda(\mu)-\max\phi)_+>0$. Therefore, we have
\vspace{-0.1cm}
\[\frac{M_1}{\sqrt{(\lambda(\mu)-\max\phi)_+}}\frac{1}{4\pi\sqrt{2}}\geq\int_0^1(j')^{-1}\left(\frac{(\lambda(\mu)-\max\phi)_+}{\vert \mu\vert}(1-u^2)\right)\dd u\geq 0,\]
the term on the left side converges to 0. Hence using Fatou's lemma, we get\[\int_0^1\underset{\mu\to-\infty}{\liminf}(j')^{-1}\left(\frac{(\lambda(\mu)-\max\phi)_+}{\vert \mu\vert}(1-u^2)\right)\dd u=0.\] 
We deduce that $\frac{(\lambda(\mu)-\max\phi)_+}{\vert \mu\vert}\underset{\mu\to-\infty}{\longrightarrow}0$ and we conclude with (\ref{truc2}) that $\underset{\mu\to -\infty}{\lim}G(\mu)=0$. The proof is complete.
\end{proof}
As mentionned before the sequence $(F^{\phi_{f_n}})_n$ will be used to show the existence of minimizers of (\ref{I(M1,Mj)}) and the compactness of minimizing sequences. To do that, we need to link $\ch(f_n)$ and $\ch(F^{\phi_{f_n}})$. For this purpose, we introduce a second problem of minimization and we set $M_1,M_j>0$.
\begin{equation}\label{J}
\cj_0=\underset{\int_0^{2\pi}\phi=0}{\inf}\cj(\phi) \text{ where } \cj(\phi)=\iint\left(\frac{v^2}{2}+\phi(\theta)\right)F^{\phi}(\theta,v)\mathrm{d}\theta\mathrm{d}v+\frac{1}{2}\int_0^{2\pi}\phi'(\theta)^2\mathrm{d}\theta,
\end{equation}
where $F^{\phi}$ is defined by Lemma \ref{F^phi}.
\begin{lem}\label{inequality_relou}
We have the following inequalities:
\begin{enumerate}
\item For all $\phi \!\in\! H^2([0,2\pi])$ such that $\phi(0)=\phi(2\pi)$ and $\int_0^{2\pi}\!\phi\!=\!0$, we have $\mathcal{H}(F^{\phi})\!\leq \!\cj(\phi)$.
\item For all $f\in E_j$ with $\Vert f\Vert_{\dL^1}=M_1$ and $\Vert j(f)\Vert_{\dL^1}=M_j$, we have 
 \[\ci(M_1,M_j)\leq \mathcal{H}(F^{\phi_f})\leq \cj(\phi_f)\!\leq\! \mathcal{H}(f).\] 
Besides $\ci(M_1,M_j)=\cj_0$.
\end{enumerate}
\end{lem}
\begin{proof}
First, let us show item (1) of this lemma. Let $\phi \in H^2([0,2\pi])$ such that $\phi(0)=\phi(2\pi)$ and $\int_0^{2\pi}\phi=0$, we have
\begin{align*}
\cj(\phi)&=\ch(F^{\phi})-\frac{1}{2}\Vert \phi'_{F^{\phi}}\Vert_{\dL^2}^2+\frac{1}{2}\Vert \phi'\Vert_{\dL^2}^2+\iint (\phi(\theta)-\phi_{F^{\phi}}(\theta))F^{\phi}(\theta,v)\dd \theta\dd v\\
&=\ch(F^{\phi})-\frac{1}{2}\Vert \phi'_{F^{\phi}}\Vert_{\dL^2}^2+\frac{1}{2}\Vert \phi'\Vert_{\dL^2}^2+\int_0^{2\pi}(\phi-\phi_{F^{\phi}})(\phi''_{F^{\phi}}+\frac{\Vert F^{\phi}\Vert_{\dL^1}}{2\pi})\dd\theta,
\end{align*}
since $\phi_{F^{\phi}}$ satisfies the Poisson equation (\ref{Poisson}). Then, after integrating by parts and gathering the terms, we get
\vspace{-0.4cm}
\begin{equation}\label{norme2}
\cj(\phi)=\ch(F^{\phi})+\frac{1}{2}\Vert \phi'_{F^{\phi}}-\phi'\Vert_{\dL^2}^2.
\end{equation}
Hence $\cj(\phi)\geq \ch(F^{\phi})$.
Then, let us show the right inequality of item (2). Let $f\in E_j$ such that $\Vert f\Vert_{\dL^1}=M_1$ and $\Vert j(f)\Vert_{\dL^1}=M_j$. Using $\Vert F^{\phi}\Vert_{\dL^1}=M_1$ and $\Vert j(F^{\phi})\Vert_{\dL^1}=M_j$, using equality (\ref{Hamiltonien}), the Hamiltonian can be written in the form
\begin{align*}
\ch(f)&=\cj(\phi_f)+\iint\left(\frac{v^2}{2}+\phi_f(\theta)\right)(f(\theta,v)-F^{\phi_f}(\theta,v))\dd\theta\dd v\\
&=\cj(\phi_f)+\iint(\mu j'(F^{\phi_f})+\lambda)(f(\theta,v)-F^{\phi_f}(\theta,v))\dd\theta\dd v.
\end{align*}
\vspace{-0.3cm}
We get
\vspace{-0.1cm}
\begin{equation}\label{convexity}
\ch(f)=\cj(\phi_f)-\mu\iint (j(f)-j(F^{\phi_f})-j'(F^{\phi_f})(f-F^{\phi_f}))\dd\theta\dd v.
\end{equation}
The convexity of $j$ gives us the desired inequality. The other inequalities are straightforward.
\end{proof}

\subsection{Existence of ground states} This section is devoted to the proof of Theorem \ref{Existence2}.
\subsubsection{Properties of the infimum}
\begin{lem}\label{Inf_fini_compacite2}
The variational problem (\ref{I(M1,Mj)}) satisfies the following statements.
\begin{enumerate}
\item The infimum (\ref{I(M1,Mj)}) exists i.e. $\ci(M_1,M_j) >-\infty$ for $M_1, M_j>0$.
\item For any minimizing sequence $(f_n)_n$ of the variational problem (\ref{I(M1,Mj)}), we have the following properties:
\begin{enumerate}
\item The minimizing sequence $(f_n)_n$ is weakly compact in $\dL^1([0,2\pi]\times\br)$ i.e. there exists $\bar{f}\in\dL^1([0, 2\pi]\times \br)$ such that $f_n \underset{n\to+\infty}{\rightharpoonup}\bar{f}$ weakly in $\dL^1$.
\item We have $\Vert\phi_{f_n}-\phi_{\bar{f}}\Vert_{H^1}\underset{n\to+\infty}{\longrightarrow}0.$
\end{enumerate}
\end{enumerate}
\end{lem}
The proof of Lemma \ref{Inf_fini_compacite2} is similar to the one of Lemma \ref{Inf_fini_compacite1}.

\begin{lem}\label{lambda_et_mu}
Let $(f_n)_n$ be a minimizing sequence of the variational problem (\ref{I(M1,Mj)}) and let $\phi_n:=\phi_{f_n}$  be the associated potential. Using Lemma \ref{F^phi}, there exists a unique pair $(\lambda_n,\mu_n)\in\br\times\br_-^*$ such that $F^{\phi_n}(\theta,v)\!=\!(j')^{-1}\left(\frac{\frac{v^2}{2}+\phi_n(\theta)-\lambda_n}{\mu_n}\right)_+$ verifies $\Vert F^{\phi_n}\Vert_{\dL^1}=M_1$ and $\Vert j(F^{\phi_n})\Vert_{\dL^1}=M_j$. The sequences $(\lambda_n)_n$ and $(\mu_n)_n$ are bounded.
\end{lem}
\begin{proof}
Let us first prove that the sequence $(\lambda_n)_n$ is bounded. We argue by contradiction. Hence up to an extraction of a subsequence, $\lambda_n\underset{n\to+\infty}{\longrightarrow}+\infty$. According to the expression (\ref{phi}) of the potential $\phi_n$, we have $\Vert \phi_n\Vert_{\dL^{\infty}}\leq 2\pi\Vert W\Vert_{\dL^{\infty}}M_1:=C$. Using the expression of $M_1$ given by (\ref{M1}), we get
\vspace{-0.2cm}
\[M_1\geq\sqrt{(\lambda_n-C)_+}4\pi\sqrt{2}\int_0^1(j')^{-1}\left(\frac{(\lambda_n-C)_+}{\vert \mu_n\vert}(1-u^2)\right)\dd u\geq 0.\]
Then, we argue as at the end of the proof of Lemma \ref{F^phi} and we deduce that $\frac{(\lambda_n-C)_+}{\vert \mu_n\vert}\!\!\underset{n\to+\infty}{\longrightarrow}\!\!0$. With the hypothesis (H3) and $\Vert \phi_n\Vert_{\dL^{\infty}}\leq C$, we can estimate $M_j$ as follows:
\vspace{-0.2cm}
\[0\leq M_j\leq\frac{M_1}{p}\frac{(\lambda_n+C)_+}{\vert \mu_n\vert}.\]
The term of the right side converges to 0 then we get a contradiction. The sequence $(\lambda_n)_n$ is hence bounded.
Now, we shall prove that the sequence $(\mu_n)_n$ is bounded. Using the expression (\ref{M1}) of $M_1$ and the fact that $\lambda_n$ is bounded, we have
\[\frac{M_1}{4\pi\sqrt{2}\tilde{C}}\leq(j')^{-1}\left(\frac{\tilde{C}}{\vert\mu_n\vert}\right) \quad \text{ where }\tilde{C} \text{ is a constant}.\]
Therefore we obtain
\vspace{-0.4cm}
\[0\leq\vert \mu_n\vert\leq \frac{\tilde{C}}{j'\left(\frac{M_1}{4\pi\sqrt{2}\tilde{C}}\right)}\]
and we deduce that the sequence $(\mu_n)_n$ is bounded. This achieves the proof of this lemma.
\end{proof}
\subsubsection{Proof of Theorem \ref{Existence2}}\label{Proof2}
We are now ready to prove Theorem \ref{Existence2}.\\

\textbf{Step 1} Existence of a minimizer.\\

Let $M_1, M_j>0$. From Lemma \ref{Inf_fini_compacite2}, we know that $\ci(M_1,M_j)$ is finite. Let us show that there exists a function of $E_j$ which minimizes the variational problem (\ref{I(M1,Mj)}). Let $(f_n)_n\in E_j^{\bn}$ be a minimizing sequence of $\ci(M_1,M_j)$. Thus $\ch(f_n)\underset{n\to+\infty}{\longrightarrow}\ci(M_1,M_j)$, $\Vert f_n\Vert_{\dL^1}=M_1$ and $\Vert j(f_n)\Vert_{\dL^1}=M_j$. From item (2) of Lemma \ref{Inf_fini_compacite2}, there exists $\bar{f}\in \dL^1([0,2\pi]\times\br)$ such that $f_n\underset{n\to+\infty}{\rightharpoonup}\bar{f}$ weakly in $\dL^1$. In what follows, we will denote by $\phi_n$ the potential $\phi_{f_n}$ defined by (\ref{phi}). Thanks to the weak convergence in $\dL^1$, we only get that $\Vert\bar{f}\Vert_{\dL^1}=M_1$ and $\Vert j(\bar{f})\Vert_{\dL^1}\leq M_j$. The idea is to introduce a new sequence which is a minimizing sequence of (\ref{I(M1,Mj)}) and which has better compactness properties. For this purpose, we define
\vspace{-0.1cm}
\begin{equation}\label{F^phibis}
F^{\phi_n}(\theta,v)=(j')^{-1}\left(\frac{\frac{v^2}{2}+\phi_n(\theta)-\lambda_n}{\mu_n}\right)_+
\end{equation}
where $(\lambda_n,\mu_n)$ is the unique pair  of $\br\times\br_-^*$ such that $\Vert F^{\phi_n}\Vert_{\dL^1}=M_1$ and $\Vert j(F^{\phi_n})\Vert_{\dL^1}=M_j$. According to Lemma \ref{F^phi}, $F^{\phi_n}$ is well-defined and notice that the pair $(\lambda_n, \mu_n)$ depends on $\phi_n$ this is why we will denote by $\lambda_n=\lambda(\phi_n)$ and $\mu_n=\mu(\phi_n)$. Besides, using Lemma \ref{inequality_relou}, we see that $(F^{\phi_n})_n$ is a minimizing sequence of (\ref{I(M1,Mj)}). According to item (b) of Lemma \ref{Inf_fini_compacite2}, $\phi_n$ converges to $\bar{\phi}:=\phi_{\bar{f}}$ strongly in $\dL^2([0,2\pi]\times\br)$. Thus, up to an extraction of a subsequence, $\phi_n$ converges to $\bar{\phi}$ a.e. Let us prove that the sequences $(\lambda_n)_n$ and $(\mu_n)_n$ converge. Using Lemma \ref{lambda_et_mu}, we get that the sequences $(\lambda_n)_n$ and $(\mu_n)_n$ are bounded. Therefore, there exists $\lambda_0$ and $\mu_0$ such that, up to an extraction of a subsequence, $\lambda_n\underset{n\to+\infty}{\longrightarrow}\lambda_0$ and $\mu_n\underset{n\to+\infty}{\longrightarrow}\mu_0$.
Let us prove that $\mu_0<0$. Assume that $\mu_n\underset{n\to +\infty}{\longrightarrow} 0$. First assume that $\lambda_n \underset{n\to +\infty}{\longrightarrow}\lambda_0\neq \min\bar{\phi}$. From assumptions on $j$, this implies
\vspace{-0.2cm}
\[(j')^{-1}\left(\frac{\lambda_n -\frac{v^2}{2}-\phi_n(\theta)}{\vert \mu_n\vert}\right)_+\underset{n\to+\infty}{\longrightarrow}+\infty \text{ for almost all } (\theta,v)\in[0,2\pi]\times\br.\]
And using Fatou's lemma, we get a contradiction. Then assume that $\lambda_n\underset{n\to+\infty}{\longrightarrow}\min\bar{\phi}$, using inequality (\ref{lambda(mu)}), we get 
\vspace{-0.2cm}
\begin{equation}\label{Mj}
M_j \geq \frac{j\left(\frac{M_1}{\alpha_n}\right)}{\frac{M_1}{\alpha_n}} \quad \text{ with }\quad \alpha_n=2\sqrt{2}\int_0^{2\pi}\sqrt{(\lambda_n-\phi_n(\theta))_+}\dd\theta.
\end{equation}
Using the dominated convergence theorem, we show that $\alpha_n\underset{n\to+\infty}{\longrightarrow} 0$. But $j$ satisfies (H2) thus $\frac{j\left(\frac{M_1}{\alpha_n}\right)}{\frac{M_1}{\alpha_n}}\underset{n\to+\infty}{\longrightarrow}+\infty$ and we get a contradiction with (\ref{Mj}). Besides $\lambda_0\neq \min\bar{\phi}$ since otherwise $F^{\phi_n}$ converges to $0$ and we get a contradiction with $\Vert F^{\phi_n}\Vert_{\dL^1}=M_1$. Hence we have proved that $F^{\phi_n}$ converges to $(j')^{-1}\left(\frac{\frac{v^2}{2}+\bar{\phi}(\theta)-\lambda_0}{\mu_0}\right)_+$ a.e. Now let us show that $\lambda_0=\lambda(\bar{\phi})$ and $\mu_0=\mu(\bar{\phi})$ to get that $(j')^{-1}\left(\frac{\frac{v^2}{2}+\bar{\phi}(\theta)-\lambda_0}{\mu_0}\right)_+$ satisfies the two constraints. For this purpose, we first prove by the dominated convergence theorem, $\Vert \phi_n\Vert_{\dL^{\infty}}$ being bounded, that
\begin{align}\label{a}
\begin{cases}
&\Vert F^{\phi_n}\Vert_{\dL^1}\underset{n\to+\infty}{\longrightarrow}\int_0^{2\pi}\int_{\br}(j')^{-1}\left(\frac{\frac{v^2}{2}+\bar{\phi}(\theta)-\lambda_0}{\mu_0}\right)_+\dd\theta\dd v,\\
&\Vert j(F^{\phi_n})\Vert_{\dL^1}\underset{n\to+\infty}{\longrightarrow}\int_0^{2\pi}\int_{\br}j\circ(j')^{-1}\left(\frac{\frac{v^2}{2}+\bar{\phi}(\theta)-\lambda_0}{\mu_0}\right)_+\dd\theta\dd v.
\end{cases}
\end{align}
But $(\Vert F^{\phi_n}\Vert_{\dL^1},\Vert j(F^{\phi_n})\Vert)=(M_1,M_j)$ then
\[
M_1\!=\!\int_0^{2\pi}\!\!\!\int_{\br}(j')^{-1}\left(\frac{\frac{v^2}{2}+\bar{\phi}(\theta)-\lambda_0}{\mu_0}\right)_{\!\!+}\!\!\!\dd\theta\dd v,\ \ M_j\!=\!\int_0^{2\pi}\!\!\!\int_{\br}j\circ(j')^{-1}\left(\frac{\frac{v^2}{2}+\bar{\phi}(\theta)-\lambda_0}{\mu_0}\right)_{\!\!+}\!\!\!\dd\theta\dd v.
\]
According to Lemma \ref{F^phi}, the couple ($\lambda(\bar{\phi}),\mu(\bar{\phi}))$ is unique, so $\lambda_0=\lambda(\bar{\phi})$ and $\mu_0=\mu(\bar{\phi})$. Hence $F^{\phi_n}$ converges to $F^{\bar{\phi}}$ a.e. . But $\Vert F^{\bar{\phi}}\Vert_{\dL^1}=\Vert F^{\bar{\phi_n}}\Vert_{\dL^1}=M_1$ then according to Brezis-Lieb's lemma, $F^{\phi_n}\underset{n\to+\infty}{\longrightarrow}F^{\bar{\phi}}$ strongly in $\dL^1([0,2\pi]\times\br)$. We already know that $F^{\bar{\phi}}$ satisfies the two constraints, there remains to show that $\ch(F^{\bar{\phi}})=\ci(M_1,M_j)$. The strong convergence in $\dL^1([0,2\pi]\times\br)$ of $F^{\phi_n}$ to $F^{\bar{\phi}}$ implies that $\phi'_{F^{\phi_n}}\underset{n\to+\infty}{\longrightarrow}\phi'_{F^{\bar{\phi}}}$ strongly in $\dL^{2}$. Therefore using classical inequalities about the lower semicontinuity properties of convex nonnegative functions see \cite{Kavian} and the convergence in $\dL^2([0,2\pi])$ of $\phi'_{F^{\phi_n}}$, we get
\vspace{-0.1cm}
\[\mathcal{I}(M_1,M_j)\geq\iint \frac{v^2}{2}F^{\bar{\phi}}(\theta,v)\mathrm{d}\theta\mathrm{d}v-\frac{1}{2}\int_0^{2\pi}\phi'_{F^{\bar{\phi}}}(\theta)^2\mathrm{d}\theta.\]
Thus $\mathcal{I}(M_1,M_j)\geq \ch(F^{\bar{\phi}})$.  As $F^{\bar{\phi}}$ satisfies the two constraints and belongs to $E_j$, we have $\mathcal{I}(M_1,M_j)\leq \ch(F^{\bar{\phi}})$. Therefore we get the equality and we have shown the existence of a minimizer.\\

\textbf{Step 2}: The minimizer is a steady state of (\ref{HMF}).\\

To prove that the minimizer $F^{\bar{\phi}}$ is a stationary state of the system (\ref{HMF}), it is sufficient to show that $\bar{\phi}=\phi_{F^{\bar{\phi}}}$. First, $(F^{\phi_n})_n$ being a minimizing sequence of (\ref{I(M1,Mj)}), we have $\ch(F^{\phi_n})\underset{n\to+\infty}{\longrightarrow}\ci(M_1,M_j)$. Then, using Lemma \ref{inequality_relou}, we know that $\cj_0=\ci(M_1,M_j)$ and that $\ci(M_1,M_j)\leq \cj(\phi_n)\leq\ch(f_n)$. Hence $(\phi_n)_n$ is a minimizing sequence of $\cj_0$: we have $\cj(\phi_n)\underset{n\to+\infty}{\longrightarrow}\ci(M_1,M_j)=\cj_0$. Hence using the equality (\ref{norme2}), we get
\vspace{-0.1cm}
\[\Vert \phi_{F^{\phi_n}}'-\phi_n'\Vert_{\dL^2}^2\underset{n\to+\infty}{\longrightarrow}0.\] 
Passing to the limit $n\to+\infty$ and knowing that $\bar{\phi}$ has a zero average, we deduce that $\bar{\phi}=\phi_{F^{\bar{\phi}}}$ a.e.\\ 

\textbf{Step 3}: Euler-Lagrange equation for minimizers.\\

There remains to prove part (2) of Therorem \ref{Existence2}. We obtain Euler-Lagrange equation for the minimizer in the same way as in the proof of Theorem \ref{Existence_minimizers1} in Section \ref{Proof1}. Indeed, according to Lemma \ref{inequality_relou}, if $\bar{f}$ is a minimizer of $\ci(M_1,M_j)$, $\bar{\phi}:=\phi_{\bar{f}}$ is a minimizer of $\cj_0$ and $\ch(\bar{f})=\cj(\bar{\phi})$. Using (\ref{convexity}), we get 
\[\iint (j(\bar{f})-j(F^{\bar{\phi}})-j'(F^{\bar{\phi}})(\bar{f}-F^{\bar{\phi}}))\dd\theta\dd v=0.\]
Then writting the Taylor's formula for $j$ and using $j''>0$, we can deduce as in Section \ref{Proof1} that $\bar{f}=F^{\bar{\phi}}$.\\

\textbf{Step 4}: Regularity of the potential $\phi_{f}$.\\

First, we will show that $\phi_f\in \cc^1([0,2\pi])$. Thanks to the Sobolev embedding 
\vspace{-0.1cm}
\[W^{2,3}([0,2\pi])\hookrightarrow \cc^{1,\frac{2}{3}}([0,2\pi]),\] 
it is sufficient to show that $\phi_f\in W^{2,3}([0,2\pi])$. We know that $f\in \dL^1([0,2\pi]\times\br)$, then with expression (\ref{phi}), we get $\phi_f\in\dL^{\infty}([0,2\pi])\subset L^3([0,2\pi])$. In the same way, $\phi_f'\in\dL^3([0,2\pi])$. Besides $\phi_f$ satisfies (\ref{Poisson}), then let us show that $\rho_f\in\dL^3([0,2\pi])$. According to the previous step, $f$ is compactly supported and since $\phi_f\in\dL^{\infty}$, we get $f\in\dL^{\infty}([0,2\pi]\times\br)$. We also have $v^2f\in\dL^1([0,2\pi]\times\br)$. Therefore with a classical argument, we show $\rho_f\in\dL^3([0,2\pi])$ and we get $\phi_f\in \cc^1([0,2\pi])$. Then, according to its expression (\ref{phi}), $\rho_f$ is continuous. Hence $\phi_f''\in \cc^0([0,2\pi])$ and $\phi_f'\in W^{1,3}([0,2\pi])\cap\cc^0([0,2\pi])$, then we can write for $x$, $y\in[0,2\pi]$
\vspace{-0.2cm}
\begin{equation}\label{phi_c2}
\phi_f'(y)-\phi_f'(x)=\int_x^y\phi_f''(t)\dd t.
\end{equation}
We deduce from (\ref{phi_c2}) that $\phi_f'\in\cc^1([0,2\pi])$ then $\phi_f\in\cc^2([0,2\pi])$.

\subsection{Orbital stability of the ground states}
To prove the orbital stability result stated in Theorem \ref{Orbital_stability2}, we first need to prove the local uniqueness of the minimizers under equimeasurability condition.
\subsubsection{Local uniqueness of the minimizers under equimeasurability condition}\label{Uniqueness2}

In this section, we prove Lemma \ref{Unicite2}. To this purpose, we first need to prove some preliminary lemmas.
\begin{lem}\label{mu_isole}
Let $f_1$, $f_2$ be two equimeasurable steady states of (\ref{HMF}) which minimizes (\ref{I(M1,Mj)}), they can be written in the form (\ref{forme}) with $(\lambda_1, \mu_1), (\lambda_2,\mu_2)\in \br\times\br_-^*$, we have for all $e\geq 0$
\begin{equation}\label{egalite_utile}
\vert\mu_1\vert^{\frac{1}{2}}\!\!\int_0^{2\pi}\!\!\!(\psi_1(\theta)-e)_+^{\frac{1}{2}}\dd\theta=\vert\mu_2\vert^{\frac{1}{2}}\!\!\int_0^{2\pi}\!\!\!(\psi_2(\theta)-e)_+^{\frac{1}{2}}\dd\theta \quad { where } \quad \psi_i=\frac{\phi_{f_i}-\lambda_i}{\mu_i},\, i=1,2.
\end{equation}
 Besides, if $f_1$ and $f_2$ are inhomogeneous then there exist $p_1=p_1(\phi_{f_1})\in\bn^*$ and $p_2=p_2(\phi_{f_2})\in \bn^*$ such that
\vspace{-0.1cm}
\begin{equation}\label{egalite_mu}
\frac{p_1\vert\mu_1\vert^{\frac{3}{4}}}{\sqrt{\left\vert a(e_0)-\frac{1}{\vert\mu_1\vert^{\frac{1}{2}}}c_0\right\vert}}=\frac{p_2\vert\mu_2\vert^{\frac{3}{4}}}{\sqrt{\left\vert a(e_0)-\frac{1}{\vert\mu_2\vert^{\frac{1}{2}}}c_0\right\vert}},
\end{equation}
where 
\vspace{-0.1cm}
\[
\begin{cases}
& a(e_0)=\int_{\br}(j')^{-1}\left(e_0-\frac{v^2}{2}\right)_+\quad \text{ with } e_0=max\left(\frac{\phi_{f_i}-\lambda_i}{\mu_i}\right), \quad i=1,2;\\
&c_0=\frac{M_1}{2\pi}.
\end{cases}\]
\end{lem}

\begin{lem}\label{DLcas2}
Let $\psi\in\cc^2([0,2\pi])$ such that there exists a finite number $p$ of values $\xi\in[0,2\pi]$ satisfying $\psi(\xi)=\max(\psi):=e_0$. We will denote them by $\xi_i$ for $i\in \{1,..,p\}$. Besides we assume that for all $i\in \{1,..,p\}$, we have $\psi''(\xi_i)\neq 0$ thus we have
\[\int_0^{2\pi} (\psi(\theta)-e)_+^{\frac{1}{2}}\dd\theta=\varepsilon\sum_{i=1}^p\frac{\sqrt{2}}{\sqrt{\vert \psi''(\xi_i)\vert }}\int_0^1s^{-\frac{1}{2}}(1-s)^{\frac{1}{2}}\dd s+ o(\varepsilon) \quad \text{ with } \varepsilon=e_0-e.\]
\end{lem}
We first show Lemma \ref{Unicite2} using Lemmas \ref{mu_isole} and \ref{DLcas2} then Lemmas \ref{mu_isole} and \ref{DLcas2} will be proved.
\begin{proof}[Proof of Lemma \ref{Unicite2}]
Let $f_0$ be a homogeneous steady state of (\ref{HMF}) and a minimizer of (\ref{I(M1,Mj)}). It can be written in the form (\ref{forme}) with $(\lambda_0, \mu_0)\in \br\times\br_-^*$. First, let $f$ be a homogeneous steady state of (\ref{HMF}) and a minimizer of (\ref{I(M1,Mj)}) equimeasurable to $f_0$. It can be written in the form (\ref{forme}) with $(\lambda, \mu)\in\br\times\br_-^*$. We can also write
\begin{align}\label{expression_f_psi}
\begin{cases}
f_0(\theta,v)=(j')^{-1}\left(\frac{-v^2}{2\vert\mu_0\vert}+\psi_0(\theta)\right)_+ \quad &\text{ with } \quad \psi_0(\theta)=\frac{\phi_{f_0}(\theta)-\lambda_0}{\vert\mu_0\vert},\\
f(\theta,v)=(j')^{-1}\left(\frac{-v^2}{2\vert\mu\vert}+\psi(\theta)\right)_+ \quad &\text{ with }\quad \psi(\theta)=\frac{\phi_f(\theta)-\lambda}{\vert\mu\vert}.
\end{cases}
\end{align} 
The homogeneity and equimeasurability of $f_0$ and $f$ implies $\frac{\lambda_0}{\vert \mu_0\vert}=\frac{\lambda}{\vert \mu\vert}$. Besides replacing in equality (\ref{egalite_utile}) of Lemma \ref{mu_isole}, we get $\mu_0=\mu$ and then $\lambda_0=\lambda$. Thus $f_0=f$. Then let $f$ be an inhomogeneous steady state (\ref{HMF}) and a minimizer of (\ref{I(M1,Mj)}) equimeasurable to $f_0$. The minimizer $f$ can be written in the form (\ref{expression_f_psi}). The equimeasurability of $f_0$ and $f$ implies $\max(\psi_0)=\max(\psi)$. We note this value $e_0$ and we notice that $\psi_0(\theta)=e_0$ for all $\theta\in[0,2\pi]$. Replacing in equality (\ref{egalite_utile}) of Lemma \ref{mu_isole}, we get
\[2\pi\vert \mu_1\vert^{\frac{1}{2}}(e_0-e)_+^{\frac{1}{2}}=\vert\mu_2\vert^{\frac{1}{2}}\int_0^{2\pi}(\psi_2(\theta)-e)_+^{\frac{1}{2}}\dd\theta.\]
To estimate the right term of this equality, we will apply Lemma \ref{DLcas2} and we get
\[2\pi\vert\mu_1\vert^{\frac{1}{2}}\sqrt{\varepsilon}=\left(\vert\mu_2\vert^{\frac{1}{2}}\sum_{j=1}^{p_2}\frac{\sqrt{2}}{\sqrt{\vert\psi_2''(\xi_j)\vert}}\int_0^1s^{-\frac{1}{2}}(1-s)^{\frac{1}{2}}\dd s\right)\varepsilon+o(\varepsilon).\]
This last equality show us that this case cannot occur. Thus $f_0$ is the only homogeneous steady states of (\ref{HMF}) and minimizer of (\ref{I(M1,Mj)}) under equimeasurability condition.\\

Let $f_0$ be an inhomogeneous steady state of (\ref{HMF}) and a minimizer of (\ref{I(M1,Mj)}), it can be written in the form (\ref{forme}) with $(\lambda_0, \mu_0)\in \br\times\br_-^*$. Let $f$ be an inhomogeneous steady state of (\ref{HMF}) and a minimizer of (\ref{I(M1,Mj)}) equimeasurable to $f_0$. It can be written in the form (\ref{forme}) with $(\lambda, \mu)\in\br\times\br_-^*$. Let assume that $\mu_0=\mu$ then we can write our two minimizers like that
\vspace{-0.2cm}
\[f_0(\theta,v)=G\left(\frac{v^2}{2}+\psi_0(\theta)\right), \quad f(\theta,v)=G\left(\frac{v^2}{2}+\psi(\theta)\right),\]
 with $G(t)=(j')^{-1}\left(\left(\frac{t}{\mu}\right)_+\right)$ and $\psi_i(\theta)-\lambda_i$. Arguing as the one constraint case, we get $f_0=f$ up to a translation shift in $\theta$.
Let assume that $\mu_0\neq \mu$ and let us show that $\mu_0$ is isolated. Since $f_0$ and $f$ are inhomogeneous, they verify (\ref{egalite_mu}) according to Lemma \ref{mu_isole}. Define for $x>0$, $F(x)=\frac{x^{\frac{3}{4}}}{\sqrt{\vert a(e_0)-x^{-\frac{1}{2}}c_0\vert}}$ and  introduce the set 
\[E=\underset{p\in\bn}{\cup}\{\mu \text{ s.t. } pF(\vert\mu\vert)=A_0\}.\] 
If $E$ is finite, the result is trivial. Otherwise $E$ is countable, it can be written in the form $E=(\mu_n)_n$ with $\mu_n$ injective and satisfying for all $n\in\bn$, there exists $p_n$ such that $p_nF(\vert\mu_n\vert)=A_0$. Let $\mu_1$ a limit point of the sequence $(\mu_n)_n$, it verifies $F(\vert \mu_1\vert)=0$. Indeed, the sequence $(p_n)_n$ cannot take an infinity of times the same value since in equality (\ref{egalite_mu}), for $p$ fixed, there are at the most 4 $\mu$. Therefore $p_n\underset{n\to+\infty}{\longrightarrow}+\infty$. Thus $\mu_1=0$. As $\mu_0<0$, it is isolated. Thus there exists $\delta_0>0$ such that for all $f\neq f_0$ inhomogeneous steady state of (\ref{HMF}) and minimizer of (\ref{I(M1,Mj)}), we have $\vert\vert \mu\vert-\vert\mu_0\vert\vert>\delta_0$.
\end{proof}
Now, let us prove Lemma \ref{DLcas2}.
\begin{proof}[Proof of Lemma \ref{DLcas2}]
Let $\psi\in\cc^2([0,2\pi])$ satisfying the assumptions noted above, we have
\vspace{-0.1cm}
\begin{align*}
\int_0^{2\pi} (\psi(\theta)-e)_+^{\frac{1}{2}}\dd\theta&=\frac{1}{2}\int_0^{2\pi}\int_0^{(\psi(\theta)-e)_+}s^{-\frac{1}{2}}\dd s\dd\theta=\frac{1}{2}\int_0^{e_0-e}s^{-\frac{1}{2}}\left\vert\left\{e+s\leq \psi\leq e_0\right\}\right\vert\dd s,\\
&=\frac{\varepsilon^{\frac{1}{2}}}{2}\int_0^1s^{-\frac{1}{2}}\left\vert\left\{e_0-\varepsilon(1-s)\leq \psi\leq e_0\right\}\right\vert\dd s,\\
\end{align*}
using Fubini's theorem, putting $\varepsilon=e_0-e$ and performing a change of variables $\tilde{s}=\frac{s}{\varepsilon}$.

\noindent We define $E_{\varepsilon}=\left\{\theta \in [0,2\pi], e_0-\varepsilon(1-s)\leq \psi\leq e_0\right\}$. We can write $[0,2\pi]=\cup_{i=1}^{p}E_i$ with
\begin{align*}
\begin{cases}
&E_1=[0,\frac{\xi_1+\xi_2}{2}]\\
&E_i=[\frac{\xi_{i-1}+\xi_{i}}{2},\frac{\xi_{i}+\xi_{i+1}}{2}] \text{ for } i\in\{2,..,p-1\}\\
&E_{p}=[\frac{\xi_{p-1}+\xi_p}{2},2\pi].\\
\end{cases}
\end{align*}
Thus $E_{\varepsilon}=\cup_{i=1}^{p}E_{\varepsilon}^i$ with $E_{\varepsilon}^i=\left\{\theta\in E_i, -\varepsilon(1-s)\leq \psi(\theta)-e_0\leq 0\right\}$ and we get
\[\int_0^{2\pi}(\psi(\theta)-e)_+^{\frac{1}{2}}\dd\theta=\sum_{i=1}^{p}\frac{\varepsilon^{\frac{1}{2}}}{2}\int_0^1s^{-\frac{1}{2}}\vert E_{\varepsilon}^i\vert\dd s.\]
The next step is to compute for $i\in \{1...p\}$ the limit of $\vert E_{\varepsilon}^i\vert$ when $\varepsilon$ goes to 0. Notice that there is a unique $\xi_i$ in each interval $E_i$ for $i\in \{1...p\}$, and use the Taylor formula for $\psi$, to get
\vspace{-0.3cm}
\[E_{\varepsilon}^i=\left\{\theta\in E_i, -\varepsilon(1-s)\leq (\theta-\xi_i)^2\int_0^1(1-u)\psi''(u(\theta-\xi_i)+\xi_i)\dd u\leq 0 \right\}.\]
Let $A(\theta,\xi)=\int_0^1(1-u)\psi''(u(\theta-\xi)+\xi)\dd u$, we can write
\[E_{\varepsilon}^i=\left\{\theta\in E_i,  \frac{\vert\theta-\xi_i\vert}{\sqrt{\varepsilon}}\sqrt{\vert A(\theta,\xi_i)\vert}\leq \sqrt{1-s}\right\}.\]
Then we have
\[\vert E_{\varepsilon}^i\vert=2\sqrt{\varepsilon}\left\vert\left\{\theta\in E_i,  \theta\sqrt{\vert B(\theta,\xi_i)\vert}\leq \sqrt{1-s}\}\right\}\right\vert \text{ where } B(\theta,\xi_i)\!=\!\!\!\int_0^1\!\!(1-u)\psi''(u\sqrt{\varepsilon}\theta+\xi_i)\dd u.\]
Recall that $\psi''(\xi_i)\neq 0$ hence by continuity of $\psi''$, we have $\psi''\neq 0$ on a neighborhood of $\xi_i$. Thus for $e$ close to $e_0$ i.e. for $\varepsilon$ sufficiently small, we have $B(\theta, \xi_i)\neq 0$. Thus we can write
\[\frac{\vert E_{\varepsilon}^i\vert}{\sqrt{\varepsilon}}=2\int_0^{2\pi}\mathds{1}_{\left\{0\leq \theta\leq \frac{\sqrt{1-s}}{\sqrt{\vert B(\theta,\xi_i)\vert}}\right\}}\dd\theta.\]
Applying the dominated convergence theorem, we get for $i\in \{1...p\}$
\[\underset{\varepsilon\to 0}{\lim}\frac{\vert E_{\varepsilon}^i\vert}{\sqrt{\varepsilon}}=2\frac{\sqrt{2(1-s)}}{\sqrt{\vert\psi''(\xi_i)\vert}}.\]
This ends the proof of Lemma \ref{DLcas2}.
\end{proof}
To prove Lemma \ref{mu_isole}, we need a last technical lemma.
\begin{lem}\label{nb_fini_max}
Let $f$ be an inhomogenenous minimizer of the variational problem (\ref{I(M1,Mj)}) given by (\ref{forme}) with $(\lambda,\mu)\in\br\times\br_-^*$. We denote by $e_0:=\max\psi$ where $\psi(\theta)=\frac{\phi_f(\theta)-\lambda}{\mu}$. Then there is only a finite number of values $\xi$ satisfying $\psi(\xi)=e_0$.
\end{lem}
\begin{proof}
Let us argue by contradiction. Assume there is an infinite number of different values $\xi$ satisfying $\psi(\xi)=e_0$. We define a strictly increasing sequence $(\xi_n)_n$ such that for all $n$,  $\psi(\xi_n)=e_0$. In particular we have $\psi'(\xi_n)=0$. Then we apply Rolle's theorem on each interval $[\xi_n,\xi_{n+1}]$ and we build a new sequence $(\tilde{\xi}_n)_n$ such that $\psi''(\tilde{\xi}_n)=0$. We have $(\tilde{\xi}_n)_n\in[0,2\pi]^{\bn}$ thus there exists $\tilde{\xi}$ such that $\tilde{\xi}_n\underset{n\to+\infty}{\longrightarrow}\tilde{\xi}$ up to an extraction of a subsequence. With the continuity of $\psi''$ and Theorem \ref{Existence2}, we get $\psi''(\tilde{\xi})=0$. By construction, we have for all $n$, $\tilde{\xi}_{n-1}<\xi_n<\tilde{\xi}_n$. Thus up to an extraction of a subsequence $\xi_n\underset{n\to+\infty}{\longrightarrow}\tilde{\xi}$ and the limit satisfies $\psi'(\tilde{\xi})=0$ and $\psi(\tilde{\xi})=e_0$. Besides we know that
\vspace{-0.2cm}
\[\psi''=\frac{\phi''_f}{\mu}=\frac{\rho_f-\frac{M_1}{2\pi}}{\mu},\]
then $\rho_f(\tilde{\xi})=\frac{M_1}{2\pi}$. Using the expression of $\rho_f$, we get for all $\theta\in[0,2\pi]$, $\rho_f(\theta)\leq \rho_f(\tilde{\xi})$ and $\max(\rho_f)=\rho_f(\tilde{\xi})=\frac{M_1}{2\pi}$. Since $\int_{\br}\rho_f=M_1$, we deduce that for all $\theta\in[0,2\pi]$, $\rho_f(\theta)=\frac{M_1}{2\pi}$. Thus for all $\theta$, $\phi''_f(\theta)=0$. Since $\phi_f$ has a zera average and $\phi_f(0)=\phi_f(2\pi)$, we get $\phi_f=0$. Contradiction.
\end{proof}
We are now ready to prove Lemma \ref{mu_isole}.
\begin{proof} Let $f_1$ and $f_2$ be two steady states of (\ref{HMF}) and two minimizers of (\ref{I(M1,Mj)}) equimeasurable. They can be written in the form (\ref{forme}) and we can write
\vspace{-0.2cm}
\[f_1(\theta,v)=(j')^{-1}\left(\frac{v^2}{2\mu_1}+\psi_1(\theta)\right)_+, \qquad f_2(\theta,v)=(j')^{-1}\left(\frac{v^2}{2\mu_2}+\psi_2(\theta)\right)_+\]
where $\psi_i(\theta)=\frac{\phi_{f_i}(\theta)-\lambda_i}{\mu_i}$ for $i=1$ or $2$. Since $f_1$ and $f_2$ are equimeasurable, we know that for all $t\geq 0$
\[\left\vert\left\{(j')^{-1}\left(\frac{-v^2}{2\vert\mu_1\vert}+\psi_1(\theta)\right)_+>t\right\}\right\vert=\left\vert\left\{(j')^{-1}\left(\frac{-v^2}{2\vert\mu_2\vert}+\psi_2(\theta)\right)_+>t\right\}\right\vert.\]
We have for $i=1$ or $2$,
\vspace{-0.1cm}
\begin{align*}
\left\vert\left\{(j')^{-1}\left(\frac{-v^2}{2\vert\mu_i\vert}+\psi_i(\theta)\!\!\right)_{+}>t\right\}\right\vert1&=\left\vert\left\{\frac{v^2}{2}-\vert\mu_i\vert\psi_i(\theta)<-\vert\mu_i\vert j'(t)\right\}\right\vert\\
&=2\sqrt{2}\vert \mu_i\vert^{\frac{1}{2}}\int_0^{2\pi}(\psi_i(\theta)-j'(t))_+^{\frac{1}{2}}\dd\theta.
\end{align*}
Thus for all $e\geq 0$, we have equality (\ref{egalite_utile}). Then let assume that $\phi_{f_1}\neq 0$ and $\phi_{f_2}\neq 0$. According to the third point of Theorem \ref{Existence2}, $\psi_1$, $\psi_2\in\cc^2([0,2\pi])$. Besides according to Lemma \ref{nb_fini_max}, there exists for $i=1$ or $2$, $p_i=p_i(\phi_{f_i})$ such that $\psi_i$ has $p_i$ values $\xi$ satisfying $\psi_i(\xi)=e_0$. We note them $\{\xi_{i,1},..,\xi_{i,p_i}\}$. In order to apply Lemma \ref{DLcas2}, let us show that $\psi_i''(\xi_{i,j})\neq 0$ for $j\in\{1,..,p_i\}$ and $i=1$ or $2$. If $\psi_i''(\xi_{i,j})=0$, since $\xi_{i,j}$ is a maximum of $\psi$ too, we are in the same case as the end of the proof of Lemma \ref{nb_fini_max} and we get a contradiction. Hence we are allowed to use Lemma \ref{DLcas2} and get 
\[\vert\mu_1\vert^{\frac{1}{2}}\sum_{j=1}^{p_1}\frac{1}{\sqrt{\vert \psi_1''(\xi_{1,j})\vert}}=\vert\mu_2\vert^{\frac{1}{2}}\sum_{j=1}^{p_1}\frac{1}{\sqrt{\vert \psi_2''(\xi_{1,j})\vert}}.\]
Notice that we have for $i=1$ or $2$
\begin{align*}
\psi_i''(\theta)&=\phi_{f_i}''(\theta)=\frac{\rho_{f_i}(\theta)-\frac{M_1}{2\pi}}{\mu_i}=\frac{1}{\mu_i}\left(\int_{\br}(j')^{-1}\left(\frac{-v^2}{2\vert\mu_i\vert}+\psi_i(\theta)\right)_+\dd v-\frac{M_1}{2\pi}\right),\\
&=-\vert\mu_i\vert^{-\frac{1}{2}}\left(\int_{\br}(j')^{-1}\left(\frac{-v^2}{2}+\psi_i(\theta)\right)_+\dd v-\frac{1}{\vert \mu_i\vert^{\frac{1}{2}}}\frac{M_1}{2\pi}\right).
\end{align*}
 Thus we have
\vspace{-0.3cm}
\[\psi_i''(\xi_{i,j})=-\vert\mu_i\vert^{-\frac{1}{2}}\left(a(e_0)-\frac{1}{\vert \mu_i\vert}\frac{M_1}{2\pi}\right)\]
with $a(e_0)=\int_{\br}(j')^{-1}\left(e_0-\frac{v^2}{2}\right)_+\dd v$, and therefore equality (\ref{egalite_mu}) is proved.
\end{proof}

\subsubsection{Proof of Theorem \ref{Orbital_stability2}}\label{implication et compacite}
We will prove the orbital stability of steady states of (\ref{HMF}) which are minimizers of (\ref{I(M1,Mj)}) in two steps. First we will show that any minimizing sequence is compact.\\

\textbf{Step 1} Compactness of the minimizing sequences\\

Let $(f_n)_n$ be a minimizing sequence of $\mathcal{I}(M_1,M_j)$. Let us show that $(f_n)_n$ is compact in $E_j$ i.e. there exists $f_0 \in E_j$ such that $ f_n\xrightarrow{E_j}f_0$ up to an extraction of a subsequence. Using item (2) of Lemma \ref{Inf_fini_compacite2}, there exists $f_0\in\dL^1([0,2\pi]\times\br)$ such that $f_n \underset{n\to+\infty}{\rightharpoonup}f_0$ weakly in $\dL^1([0,2\pi]\times\br)$ and we denote by $\phi_0:=\phi_{f_0}$. In the same way as the proof of Theorem \ref{Existence2} in Section \ref{Proof2}, we introduce the function $F^{\phi_n}$ defined by (\ref{F^phibis}). According to Step 1 of the proof of Theorem \ref{Existence2} in Section \ref{Proof2}, it is a minimizing sequence of (\ref{I(M1,Mj)}), $F^{\phi_n}$ converges to $F^{\phi_0}$ strongly in $\dL^1([0,2\pi]\times\br)$ and $F^{\phi_0}$ is a minimizer of $\ci(M_1,M_j)$. Our goal is to prove that $f_0=F^{\phi_0}$ and $f_n\xrightarrow{E_j}f_0$.

In order to do that, let us start with the proof of the strong convergence in $\dL^1([0,2\pi]\times\br)$ of $f_n$ to $F^{\phi_0}$. First, we notice that $\Vert f_n\Vert_{\dL^1}=\Vert F^{\phi_0}\Vert_{\dL^1}=M_1$, then thanks to Brezis-Lieb's lemma, it is sufficient to show that $f_n$ converges to $F^{\phi_0}$ a.e. in order to get the strong convergence in $\dL^1([0,2\pi]\times\br)$. To this purpose, let us write
\vspace{-0.1cm}
\[f_n-F^{\phi_0}=f_n-F^{\phi_n}+F^{\phi_n}-F^{\phi_0}.\]
As the a.e.  convergence of $F^{\phi_n}$ to $F^{\phi_0}$ is already known, the next step is to show that $f_n-F^{\phi_n}$ converges to 0 a.e. For this purpose, we wil argue as in the proof of Theorem \ref{Orbital_stability1} in Section \ref{OS1}. We notice that we have
\begin{equation}\label{pp}
\iint (j(f_n)-j(F^{\phi_n})-j'(F^{\phi_n})(f_n-F^{\phi_n}))\dd\theta\dd v \underset{n\to+\infty}{\longrightarrow}0.
\end{equation}
Indeed, using equality (\ref{convexity}), we get
\vspace{-0.1cm}
\[\iint (j(f_n)-j(F^{\phi_n})-j'(F^{\phi_n})(f_n-F^{\phi_n}))\dd\theta\dd v =\frac{\cj(\phi_n)-\ch(f_n)}{\mu}.\]
 There remains to argue as in Step 2 of the proof of Theorem \ref{Existence2} in Section \ref{Proof2} to get the desired limit. Then writting the Taylor's formula for the function $j(f_n)$ and integrating over $[0,2\pi]\times\br$, we get
\vspace{-0.1cm}
\[\iint (f_n-F^{\phi_n})^2\!\int_0^1(1-u)j''(u(f_n-F^{\phi_n})+F^{\phi_n})\mathrm{d}u=\!\!\iint\! j(f_n)-\iint \!j(F^{\phi_n})-\iint \!(f_n-F^{\phi_n})j'(F^{\phi_n}).\]
Thus $\iint (f_n-F^{\phi_n})^2\int_0^1(1-u)j''(u(f_n-F^{\phi_n})+F^{\phi_n})\mathrm{d}u\underset{n\to+\infty}{\longrightarrow}0$. Arguing in the same way as the proof of Theorem \ref{Orbital_stability1} in Section \ref{OS1}., we get $f_n-F^{\phi_n}\underset{n\to+\infty}{\longrightarrow}0$ a.e. To recap, we have obtained that $\Vert f_n- F^{\phi_0}\Vert_{\dL^1}\underset{n\to+\infty}{\longrightarrow}0$. But $f_n\underset{n\to+\infty}{\longrightarrow}f_0$ weakly in $\dL^1([0,2\pi]\times\br)$ then by uniqueness of the limit, we have $F^{\phi_0}=f_0$. Therefore $\Vert f_n- f_0\Vert_{\dL^1}\underset{n\to+\infty}{\longrightarrow}0$. To show the convergence in $E_j$, there remains to show that
\vspace{-0.1cm}
\[\Vert v^2(f_n-f_0)\Vert_{\dL^1}\underset{n\to+\infty}{\longrightarrow}0,\quad \text{and} \quad
\Vert j(f_n)\Vert_{\dL^1}\underset{n\to+\infty}{\longrightarrow}\Vert j(f_0)\Vert_{\dL^1}.\]
The second limit clearly comes from the fact that $f_0=F^{\phi_0}$ satisfies the constraints. For the first limit, we write
\[\iint v^2(f_n(\theta,v)-f_0(\theta,v))\dd\theta\dd v = 2(\ch(f_n)-\ch(f_0))+\Vert \phi'_n\Vert_{\dL^2}^2-\Vert \phi'_0\Vert_{\dL^2}^2.\]
Then $\Vert v^2 f_n\Vert_{\dL^1}\underset{n\to+\infty}{\longrightarrow}\Vert v^2f_0\Vert_{\dL^1}$. Besides the strong convergence in $\dL^1([0,2\pi]\times\br)$ of $f_n$ to $f_0$ implies that $v^2f_n\underset{n\to+\infty}{\longrightarrow}v^2f_0$ a.e. up to an extraction of a subsequence. We conclude with Brezis-Lieb's lemma. Hence the minimizing sequence is compact in $E_j$.\\

\textbf{Step 2 } Proof of the orbital stability\\

Before starting the proof of Theorem \ref{Orbital_stability2}, notice the following fact. As mentioned in Section \ref{Proof2}, it is possible to obtain Euler-Lagrange equations for the minimizers in the same way as in the proof of Theorem \ref{Existence_minimizers1}. This method provides the expressions of $\lambda$ and $\mu$. In particular, we have
\vspace{-0.3cm}
\begin{equation}\label{expression_mu}
\mu=-\frac{\Vert v^2f \Vert_{\dL^1}}{C_f} \text{ with } C_f= \iint fj'(f)\dd\theta\dd v-M_j.
\end{equation}
If $f_1$ and $f_2$ are equimeasurable, then $C_{f_1}=C_{f_2}$. Hence, we can rewrite the first point of Lemma \ref{Unicite2} as follows.
\begin{lem}\label{nouveau}
Let $f_0$ be an inhomogeneous steady state of (\ref{HMF}) which is a minimizer of (\ref{I(M1,Mj)}). Let $(\lambda,\mu)\in\br\times\br_-^*$ be the Lagrange multipliers associated with $f_0$ according to (\ref{forme}). There exists $\delta_0>0$ such that for all $f\in E_j $ inhomogeneous steady state of (\ref{HMF}) which is minimizer of (\ref{I(M1,Mj)}) and which is equimeasurable to $f_0$ with $\mu_0\neq \mu$, where $\mu$ is the Lagrange constant associated with $f$ in the expression (\ref{forme}), we have
\vspace{-0.1cm}
\begin{equation}\label{eq_nouveau}
\left\vert \Vert v^2f_0\Vert_{\dL^1}-\Vert v^2f\Vert_{\dL^1}\right\vert>\delta_0.
\end{equation}
\end{lem}
This characterization will be used in the proof of the orbital stability of steady states.\\

Before proving the orbital stability of minimizers, we need to prove a preliminary lemma. 
\begin{lem}\label{lem_pre}Let $f_0$ be an inhomogeneous steady state of (\ref{HMF}) which minimizes (\ref{I(M1,Mj)}). We denote by $\delta_0$ the constant associated with $f_0$ as defined in Lemma \ref{Unicite2}. We have: $\forall\varepsilon>0$, $\exists \eta>0$ such that $\forall f_{init}\in E_j$
\vspace{-0.1cm}
\begin{align*}
\Vert (1+v^2)&(f_{init}-f_0)\Vert_{\dL^1}\leq\eta \text{ and } \left\vert \iint j(f_{init})-\iint j(f_0)\right\vert\leq\eta \\
&\Rightarrow \left[\forall t>0, \left[\left\vert\Vert v^2f(t)\Vert_{\dL^1}-\Vert v^2f_0\Vert_{\dL^1}\right\vert\leq \frac{\delta_0}{2}\Rightarrow\left\vert\Vert v^2f(t)\Vert_{\dL^1}-\Vert v^2f_0\Vert_{\dL^1}\right\vert\leq \varepsilon\right]\right],
\end{align*}
where $f(t)$ is a solution to (\ref{HMF}) with initial data $f_{init}$.
\end{lem}
With this lemma, we are able to prove Theorem \ref{Orbital_stability2}. We will prove Lemma \ref{lem_pre} after the proof of Theorem \ref{Orbital_stability2}.
\begin{proof}[Proof of Theorem \ref{Orbital_stability2}]
Let us argue by contradiction, let $f_0$ be an inhomogeneous minimizer of (\ref{I(M1,Mj)}). Assume that $f_0$ is orbitally unstable. Then there exist $\varepsilon_0>0$, a sequence $(f_{init}^n)_n\in E_j^{\mathbb{N}}$ and a sequence $(t_n)_n\in (\mathbb{R}^{+}_*)^{\bn}$ such that $f_{init}^n\xrightarrow{E_j} f_0$ and for all $n$, for all $\theta_0\in [0,2\pi]$ 
\vspace{-0.1cm}
\begin{align}\label{bb}
\begin{cases}
&\Vert f^n(t_n,\theta+\theta_0,v)-f_0(\theta,v)\Vert_{\dL^1}>\varepsilon_0,\\
&\text{or }\Vert v^2(f^n(t_n,\theta+\theta_0,v)-f_0(\theta,v))\Vert_{\dL^1}>\varepsilon_0,
\end{cases}
\end{align}
where $f^n(t_n,\theta,v)$ is a solution to (\ref{HMF}) with initial data $f_{init}^n$. Let $g_n(\theta,v)=f^n(t_n,\theta,v)$, we have $\ch(g_n)\leq \ch(f_{init}^n)$ from the conservation property of the flow (\ref{HMF}). Introduce $\bar{g_n}(\theta,v)=\gamma_n g_n\left(\theta,\frac{\gamma_n}{\lambda_n}v\right)$ where $(\gamma_n,\lambda_n)$ is the unique pair such that $\Vert\bar{g_n}\Vert_{\dL^1}=M_1$ and $\Vert j(\bar{g_n})\Vert_{\dL^1}=M_j$. Besides $\gamma_n$ and $\lambda_n$ satisfy
\vspace{-0.1cm}
\begin{equation}\label{lambda et gamma}
\lambda_n=\frac{M_1}{\Vert g_n\Vert_{\dL^1}} \text{ and } \gamma_n \text{ is such that } \frac{\Vert j(\gamma_n g_n)\Vert_{\dL^1}}{\gamma_n}=\frac{M_j\Vert g_n\Vert_{\dL^1}}{M_1}.
\end{equation}
The existence and uniqueness of such $(\gamma_n,\lambda_n)$ can be proved exactly the same way as Lemma A.1 in \cite{Manev}. As $\bar{g_n}$ satisfies the two constraints of the minimization problem (\ref{I(M1,Mj)}), we have $\ch(f_0)\leq\ch(\bar{g_n})$. Besides we have
\vspace{-0.1cm}
\begin{equation}\label{gnbar}
\ch(f_0)\leq\ch(\bar{g_n})\leq \lambda_n^2\left(\left(\frac{\lambda_n}{\gamma_n^2}-1\right)\Vert \frac{v^2}{2}g_n\Vert_{\dL^1}+\ch(f_{init}^n)\right).
\end{equation}
Notice that
\vspace{-0.2cm}
\begin{equation}\label{gn}
\left\Vert g_n\right\Vert_{\dL^1}=\Vert f_{init}^n\Vert_{\dL^1}\underset{n\to +\infty}{\longrightarrow}M_1 \text{ since } \Vert f_{init}^n-f_0\Vert_{\dL^1}\underset{n\to +\infty}{\longrightarrow}0 \text{ and } \Vert f_0\Vert_{\dL^1}=M_1.
\end{equation}
Hence the sequence $(g_n)_n$ is bounded in $\dL^1$. We also have
\vspace{-0.1cm}
\[\left\Vert\frac{v^2}{2}g_n\right\Vert_{\dL^1}=\ch(g_n)+\frac{1}{2}\int_0^{2\pi}\phi_{g_n}'^2(\theta)\dd\theta\leq C+\pi\Vert W'\Vert_{\dL^{\infty}}^2\Vert g_n\Vert_{\dL^1}^2 \text{ where } C \text{ is a constant},\]
and therefore the sequence $(\Vert \frac{v^2}{2}g_n\Vert_{\dL^1})_n$ is bounded too. Let us then show that $\lambda_n$ and $\gamma_n$ converge to 1. With (\ref{lambda et gamma}), we get $\lambda_n\underset{n\to+\infty}{\longrightarrow}1$. To deal with the case of $\gamma_n$, we will use the fact that the hypothesis (H3) is equivalent to the hypothesis (H3bis)
\vspace{-0.1cm}
\[\text{ (H3bis) }:\, b^pj(t)\leq j(bt)\leq b^q j(t), \, \forall b\geq 1, \, t\geq 0 \text{ and } b^qj(t)\leq j(bt)\leq b^pj(t), \, \forall b\leq 1, \, t\geq 0.\]
Therefore using (H3bis), we get
\vspace{-0.2cm}
\[\min(C_n^{\frac{1}{p-1}},C_n^{\frac{1}{q-1}})\leq \gamma_n\leq \max(C_n^{\frac{1}{p-1}},C_n^{\frac{1}{q-1}}), \text{ where } C_n=\left(\frac{M_j}{M_1}\frac{\Vert g_n\Vert_{\dL^1}}{\Vert j(g_n)\Vert_{\dL^1}}\right)^{\frac{1}{p-1}}.\]
But $\Vert j(g_n)\Vert_{\dL^1}=\Vert j(f_{init}^n)\Vert_{\dL^1}\underset{n\to+\infty}{\longrightarrow}\Vert j(f_0)\Vert_{\dL^1}$ and therefore $C_n\underset{n\to+\infty}{\longrightarrow}1$. Thus $\gamma_n\underset{n\to+\infty}{\longrightarrow}1$. We deduce with (\ref{gnbar}) that $\underset{n\to+\infty}{\lim}\ch(\bar{g_n})=\ch(f_0)$ and thus $(\bar{g_n})_n$ is a minimizing sequence of (\ref{I(M1,Mj)}). According to the previous step, this sequence is compact, hence, up to an extraction of a subsequence, there exists $\bar{g}\in E_j$ such that $\bar{g_n}\underset{n\to+\infty}{\longrightarrow}\bar{g}$ in $E_j$. It is easy to show with Brezis-Lieb's lemma that $g_n\underset{n\to+\infty}{\longrightarrow}\bar{g}$ in $E_j$ up to an extraction of a subsequence. This implies that
\vspace{-0.1cm}
\begin{equation}\label{b}
\Vert g_n-\bar{g}\Vert_{\dL^1}\underset{n\to+\infty}{\longrightarrow}0, \,\Vert v^2(g_n-\bar{g})\Vert_{\dL^1}\underset{n\to+\infty}{\longrightarrow}0 \text{ and } \vert\iint j(g_n)-\iint j(\bar{g})\vert\underset{n\to+\infty}{\longrightarrow}0.
\end{equation}
Then we deduce of this convergence that $\ch(g_n)\underset{n\to+\infty}{\longrightarrow}\ch(\bar{g})$, but $\ch(g_n)\underset{n\to+\infty}{\longrightarrow}\ci(M_1,M_j)$ and $\ci(M_1,M_j)=\ch(\bar{g})$. Besides $\bar{g}$ satisfies the two constraints therefore $\bar{g}$ is a minimizer of (\ref{I(M1,Mj)}). Furthermore in the same way as the proof of Theorem \ref{Orbital_stability1} in Section \ref{OS1}, we prove that $\bar{g}$ and $f_0$ are equimeasurable. In summary, $f_0$ and $\bar{g}$ are equimeasurable minimizers of $\ci(M_1,M_j)$. According to Lemma \ref{Unicite2}, $g$ cannot be a homogeneous steady state. Thus $g$ is an inhomogeneous minimizer and has the form (\ref{forme}) with $(\lambda_{\bar{g}}, \mu_{\bar{g}})\in\br\times\br_-^*$. The inhomogeneous minimizer $f_0$ also has the form (\ref{forme}) with $(\lambda_0,\mu_0)\in\br\times\br_-^*$. If \mbox{$\mu_{\bar{g}}= \mu_0$},  according to Lemma \ref{Unicite2}, $f_0=\bar{g}$ up to a translation in $\theta$. Then (\ref{b}) contradicts (\ref{bb}) and we have proved that $f_0$ is an orbitally stable steady state. Otherwise, $\mu_{\bar{g}}\neq\mu_0$ and according to Lemma \ref{nouveau}, there exists $\delta_0$ such that (\ref{eq_nouveau}) holds. Now, let us show that $\vert\Vert v^2\bar{g}\Vert_{\dL^1}-\Vert v^2f_0\Vert_{\dL^1}\vert\leq\delta_0$. In order to do that, let us prove that for all $n$,
\vspace{-0.2cm}
\begin{equation}\label{gn}
\vert\Vert v^2g_n\Vert_{\dL^1}-\Vert v^2 f_0\Vert_{\dL^1}\vert\leq\frac{\delta_0}{2}.
\end{equation}
We will show that $\forall t\geq 0$, $\vert\Vert v^2f^n(t)\Vert_{\dL^1}-\Vert v^2f_0\Vert_{\dL^1}\vert\leq\frac{\delta_0}{2}$. Let us argue by contradiction and assume there exists $t\geq 0$ such that $\vert\Vert v^2f^n(t)\Vert_{\dL^1}-\Vert v^2f_0\Vert_{\dL^1}\vert>\frac{\delta_0}{2}$. As $\Vert (1+v^2)(f_{init}^n-f_0)\Vert_{\dL^1}\underset{n\to+\infty}{\longrightarrow}0$, we can assume $\forall n$, $\Vert (1+v^2)(f_{init}^n-f_0)\Vert_{\dL^1}\leq\frac{\delta_0}{4}$. This implies $\forall n$, $\vert\Vert v^2f_{init}^n\Vert_{\dL^1}-\Vert v^2f_0\Vert_{\dL^1}\vert\leq\frac{\delta_0}{4}$. Thus we have 
\vspace{-0.2cm}
\[\vert\Vert v^2f^n(0)\Vert_{\dL^1}-\Vert v^2f_0\Vert_{\dL^1}\vert\leq\frac{\delta_0}{4} \text{ and } \exists t>0 \text{ s.t. } \vert\Vert v^2f^n(t)\Vert_{\dL^1}-\Vert v^2f_0\Vert_{\dL^1}\vert>\frac{\delta_0}{2}.\]
By continuity of the map $t\mapsto\Vert v^2f^n(t)\Vert_{\dL^1}$, there exists $t_0>$ such that 
\[\vert\Vert v^2f^n(t_0)\Vert_{\dL^1}-\Vert v^2f_0\Vert_{\dL^1}\vert=\frac{\delta_0}{3}<\frac{\delta_0}{2},\]
therefore according to Lemma \ref{lem_pre}, for all $\varepsilon>0$, we have \mbox{$\vert\Vert v^2f^n(t_0)\Vert_{\dL^1}-\Vert v^2f_0\Vert_{\dL^1}\vert\leq\varepsilon$.} For instance with $\varepsilon=\frac{\delta_0}{5}$, we get a contradiction. Hence: $\forall t\geq 0$, \mbox{$\vert\Vert v^2f^n(t)\Vert_{\dL^1}-\Vert v^2f_0\Vert_{\dL^1}\vert\leq\frac{\delta_0}{2}$} and we deduce (\ref{gn}). Recall that we have $\Vert v^2(g_n-\bar{g})\Vert\underset{n\to+\infty}{\longrightarrow}0$, hence with (\ref{gn}), we deduce that $\vert\Vert v^2f_0\Vert_{\dL^1}-\Vert v^2\bar{g}\Vert_{\dL^1}\vert\leq\delta_0.$
We get a contradiction with (\ref{eq_nouveau}) and $\mu_0=\mu_{\bar{g}}$ then $f_0=\bar{g}$ up to a translation shift in $\theta$. Then (\ref{b}) contradicts (\ref{bb}) and we have proved that $f_0$ is an orbitally stable steady state.\\

If $f_0$ is a homogeneous minimizer of (\ref{I(M1,Mj)}). We follow the same reasoning by contradiction and we build an other equimeasurable minimizer $\bar{g}$. Two cases arise: first, $\bar{g}$ is inhomogeneous and in fact, this case cannot occur according to the third point of Lemma \ref{mu_isole}. Hence we get a contradiction. Secondly, $\bar{g}$ is homogeneous and we have $f_0=\bar{g}$ according to the first point of Lemma \ref{Unicite2}. We get the same kind of contradiction as in the case of $f_0$ inhomogeneous. Hence, we have proved that $f_0$ is an orbitally stable steady state.
\end{proof}
To end this section, let us prove the preliminary lemma \ref{lem_pre}.
\begin{proof}[Proof of Lemma \ref{lem_pre}]
Let us argue contradiction. Then there exist $\varepsilon_0>0$, a sequence $(f_{init}^n)_n\in E_j^{\mathbb{N}}$ and a sequence $(t_n)_n\in \mathbb{R}^{+}_*$ such that $f_{init}^n \xrightarrow{E_j} f_0$ and for all $n$, 
\vspace{-0.1cm}
\begin{equation}\label{bb_pre}
\vert\Vert v^2f^n(t_n)\Vert_{\dL^1}-\Vert v^2 f_0\Vert_{\dL^1}\vert\leq \frac{\delta_0}{2} \quad \text{ and } \quad\vert\Vert v^2f^n(t_n)\Vert_{\dL^1}-\Vert v^2f_0\Vert_{\dL^1}\vert>\varepsilon_0,
\end{equation}
where $f^n(t_n)$ is a solution to (\ref{HMF}) with initial data $f_{init}^n$. Let $g_n(\theta,v)=f^n(t_n,\theta,v)$, exactly like in the proof of Theorem \ref{Orbital_stability2}, we introduce $\bar{g_n}(\theta,v)=\gamma_n g_n\left(\theta,\frac{\gamma_n}{\lambda_n}v\right)$ where $(\gamma_n,\lambda_n)$ is the unique pair such that $\Vert\bar{g_n}\Vert_{\dL^1}=M_1$ and $\Vert j(\bar{g_n})\Vert_{\dL^1}=M_j$. In the same way as the proof of Theorem \ref{Orbital_stability2} in Section \ref{implication et compacite}, we prove that $\bar{g}$ is a minimizer of (\ref{I(M1,Mj)}) and as in the proof of Theorem \ref{Orbital_stability1} in Section \ref{OS1}, we show that $\bar{g}$ and $f_0$ are equimeasurable. Using the first inequality of (\ref{bb_pre}) and the convergence of $\Vert v^2g_n\Vert_{\dL^1}$ to $\Vert v^2 \bar{g}\Vert_{\dL^1}$, we get
\begin{equation}\label{proche}
\vert\Vert v^2 f_0\Vert_{\dL^1}-\Vert v^2\bar{g}\Vert_{\dL^1}\vert\leq \delta_0
\end{equation}
Therefore according to Lemma \ref{Unicite2}, we deduce that $f_0=\bar{g}$ up to a translation in $\theta$ and we get a contradiction with the second inequality of (\ref{bb_pre}) and the convergence in $E_j$ of $g_n$ to $\bar{g}$.
\end{proof}
%

\section{Problem with an infinite number of constraints}\label{infini_contraintes}
\subsection{Generalized rearrangement with respect to the microscopic energy}

In the same way as in the two-constraints problem, we introduce a new function denoted by $f^{*\phi}$. The sequence $(f^{*\phi_n})_n$ has better compactness properties than the sequence $(f_n)_n$. We get the compactness of $(f_n)_n$ via the compactness of $(f^{*\phi_n})_n$ thanks to monotonicity properties of $\ch$ with respect to the transformation $f^{*\phi}$ which will be detailed in Lemma \ref{Inequalities1-2-3}. To define this new function, we use the generalization of symmetric rearrangement with respect to the microscopic energy $e=\frac{v^2}{2}+\phi(\theta)$ introduced in \cite{Stability_HMFcos}. For more generalized results, see also \cite{Lemou_seul}. We first recall the usual notion of rearrangement which is adapted here to functions defined on the domain $\bt\times\br$. For more details  on this subject see \cite{Kavian} and \cite{Loss}. For any nonnegative function $f\in\dL^1(\bt\times\br)$, we define its distribution function with (\ref{mu_f}). Let $f^{\#}$ be the pseudo-inverse of the function $\mu_f$ defined by (\ref{mu_f}):
\begin{equation}
f^{\#}(s)=\inf\{t\geq 0, \mu_f(t)\leq s\}=\sup \{t\geq 0, \mu_f(t)>s\},\, \text{ for all } s\geq 0.
\end{equation}
We notice that $f^{\#}(0)=\Vert f\Vert_{\dL^{\infty}}\in\br\cup\{+\infty\}$ and $f^{\#}(+\infty)=0$. It is well known that $\mu_f$ is right-continuous and that for all $s\geq 0$, $t\geq 0$,
\begin{equation}
f^{\#}(s)>t\quad\Longleftrightarrow\quad\mu_f(t)>s.
\end{equation}
Next, we define the rearrangement $f^*$ of $f$ by
\begin{equation}
f^*(\theta,v)=f^{\#}\left(\left\vert B(0,\sqrt{\theta^2+v^2})\cap\bt\times\br\right\vert\right),
\end{equation}
where $B(0,R)$ denotes the open ball in $\br^2$ centered at 0 with radius $R$. Then in order to generalize the rearrangements, we introduce for $\phi\in\cc^2(\bt)$ the quantity
\begin{equation}\label{a_phi}
a_{\phi}(e)=\left\vert \left\{(\theta,v)\in[0,2\pi]\times\br:\frac{v^2}{2}+\phi(\theta)<e\right\}\right\vert.
\end{equation}
From this quantity, we can adapt the proofs in Section 2.1 of \cite{Stability_HMFcos} to the case of $\phi\in\cc^2$ and we are able to define the generalized rearrangement with respect to the microscopic energy. We get the following properties gathered in Lemma \ref{rearrangement}. The last item of this lemma is proved in the Step 2 of the proof of Proposition 2.3 in \cite{Stability_HMFcos}.
\begin{lem}[Properties of $a_{\phi}$] We have the following statements.\label{rearrangement}
\begin{enumerate}
\item The function $a_{\phi}$ is continuous on $\br$, vanishes on $]-\infty,\min\phi]$ and is strictly increasing from $[\min\phi,+\infty[$ to $[0,+\infty[$.
\item  The function $a_{\phi}$ is invertible from $[\min\phi,+\infty[$ to $[0,+\infty[$, we denote its inverse by $a_{\phi}^{-1}$. This inverse satisfies
\vspace{-0.1cm}
\begin{equation}\label{inegalite_aphi}
\frac{s^2}{32\pi^2}+\min\phi\leq a_{\phi}^{-1}(s)\leq\frac{s^2}{32\pi^2}+\max\phi, \quad \forall s\in\br_+.
\end{equation}
\item Let $\phi\in\cc^{2}([0,2\pi])$ and let $a_{\phi}$ be the function defined by (\ref{a_phi}). Let $f$ be a nonnegative function in $\dL^1([0,2\pi]\times\br)$. Then the function
\[f^{*\phi}(\theta,v)=f^{\#}\left(a_{\phi}\left(\frac{v^2}{2}+\phi(\theta)\right)\right), \quad (\theta,v)\in[0,2\pi]\times\br\]
is equimeasurable to $f$, that is $\mu_{f^{*\phi}}=\mu_f$ where $\mu_f$ is defined by (\ref{mu_f}). The function $f^{*\phi}$ is called the decreasing rearrangement with respect to the microscopic energy $\frac{v^2}{2}+\phi(\theta)$.
\item Let $f\in\dL^1([0,2\pi]\times\br)$ and $\phi_f$ is the potential associated to $f$ defined by (\ref{phi}), we have 
\begin{equation}\label{utile1}
\iint\left(\frac{v^2}{2}+\phi_f(\theta)\right)(f(\theta,v)-f^{*\phi_f}(\theta,v))\mathrm{d}\theta\mathrm{d}v\geq 0,
\end{equation}
\end{enumerate}
\end{lem}
The next lemma, proved in Section 3.1 of \cite{Lemou_seul}, is a technical lemma about rearrangements which will be used in Lemma \ref{Bphi}.
\begin{lem}\label{classic_egalite}
Let $\phi\in\cc^{2}([0,2\pi])$ and $f\in\dL^1([0,2\pi]\times\br)$, we have the following identity
\[\int_0^{2\pi}\int_{\br}\left(\frac{v^2}{2}+\phi(\theta)\right)f^{*\phi}(\theta,v)\dd\theta\dd v=\int_0^{+\infty}a_{\phi}^{-1}(s)f^{\#}(s)\dd s.\]
\end{lem}

In the rest of this Section, we adopt the following definition of minimizing sequences.
\begin{defi}[Minimizing sequence] \label{minimizing_sequence}
We shall say that $(f_n)_n$ is a minimizing sequence of (\ref{H0}) if $(f_n)_n$ is uniformly bounded and 
\[\mathcal{H}(f_n)\underset{n\to+\infty}{\longrightarrow}H_0 \quad \text{ and } \quad\Vert f_n^*-f_0^*\Vert_{\dL^1}\underset{n\to+\infty}{\longrightarrow}0 .\]
\end{defi}

As mentionned at the beginning of this section, we need to link $\ch(f_n)$ and $\ch(f^{*\phi_n})$ to get compactness for $f_n$. Hence, we introduce a second problem of minimization 
\begin{equation}\label{J0}
\cj_{f^*}^0=\!\!\underset{\int_0^{2\pi}\phi=0}{\inf}\cj_{f^*}(\phi) \text{ where } \cj_{f^*}(\phi)\!\!=\!\!\iint\left(\frac{v^2}{2}+\phi(\theta)\right)f^{*\phi}(\theta,v)\mathrm{d}\theta\mathrm{d}v+\frac{1}{2}\int_0^{2\pi}\phi'(\theta)^2\mathrm{d}\theta.
\end{equation}
\begin{lem}[Monotonicity properties of $\ch$ with respect to the transformation $f^{*\phi}$]\label{Inequalities1-2-3}
We have the following inequalities:
\begin{enumerate} 
\item Let $f\in \mathcal{E}$, for all $\phi \in H^2([0,2\pi])$ such that $\phi(0)=\phi(2\pi)$ and $\int_0^{2\pi}\phi=0$, we have $\mathcal{H}(f^{*\phi})\leq \cj_{f^*}(\phi)$.
\item For all $f\in\mathcal{E}$, $H_0\leq \mathcal{H}(f^{*\phi_f})\leq \cj_{f^*}(\phi_f)\leq \mathcal{H}(f)$ where $H_0$ is defined by (\ref{H0}). Besides $H_0=\cj_{f^*}^0$.
\end{enumerate}
\end{lem}
\begin{proof} 
The first item of this lemma is proved exactly like item (2) of Lemma \ref{inequality_relou}. Hence we have
\begin{equation}\label{norme2ter}
\cj_{f^*}(\phi)=\ch(f^{*\phi})+\frac{1}{2}\Vert \phi'_{f^{*\phi}}-\phi'\Vert_{\dL^2}^2
\end{equation}
Then, let us prove the right inequality of item (2). Let $f\in\ce$, the hamiltonian can be written as
\begin{align*}
\mathcal{H}(f)&=\int_0^{2\pi}\int_{\mathbb{R}}\left(\frac{v^2}{2}+\phi_f(\theta)\right)f^{*\phi_f}(\theta,v)\mathrm{d}\theta\mathrm{d}v+\frac{1}{2}\int_0^{2\pi}\phi'_f(\theta)^2\mathrm{d}\theta\\
&+\int_0^{2\pi}\int_{\mathbb{R}}\left(\frac{v^2}{2}+\phi_f(\theta)\right)(f(\theta,v)-f^{*\phi_f}(\theta,v))\mathrm{d}\theta\mathrm{d}v\\
&=\cj_{f^*}(\phi_f)+\int_0^{2\pi}\int_{\mathbb{R}}\left(\frac{v^2}{2}+\phi_f(\theta)\right)(f(\theta,v)-f^{*\phi_f}(\theta,v))\mathrm{d}\theta\mathrm{d}v.
\end{align*}
Using (\ref{utile1}), we get that $\mathcal{H}(f^{*\phi})\leq \cj_{f^*}(\phi)$. Thanks to the two above inequalities, we easily deduce $H_0=\cj_{f^*}^0$.
\end{proof}

\subsection{Existence of ground states}
This section is devoted to the proof of Theorem \ref{Existence3}. 
\subsubsection{Properties of the infimum}

\begin{lem}\label{Inf_fini_compacite3}The variational problem (\ref{H0}) satisfies the following statements.
\begin{enumerate}
\item The infimum (\ref{H0}) exists i.e. $H_0 >-\infty$.
\item For any minimizing sequence $(f_n)_n$ of the variational problem (\ref{H0}), we have the following properties:
\begin{enumerate}
\item There exists $\bar{f}\in\dL^1([0,2\pi]\times\br)$ such that $f_n\underset{n\to+\infty}{\longrightarrow}\bar{f}$ weakly in $\dL^1$.
\item We have $\Vert\phi_{f_n}-\phi_{\bar{f}}\Vert_{H^1}\underset{n\to+\infty}{\longrightarrow}0$. 
\end{enumerate}
\end{enumerate}
\end{lem}
The proof of item (1) from Lemma \ref{Inf_fini_compacite3} is similar to the one of Lemma \ref{Inf_fini_compacite1}.  In the spirit of Lemma \ref{Inf_fini_compacite1}, noticing that $\Vert f_n\Vert_{\dL^1}=\Vert f_n^*\Vert_{\dL^1}$ is bounded and using Dunford-Pettis's theorem, we get the weak convergence of $(f_n)_n$ in $\dL^1([0,2\pi]\times\br)$. The proof of item (b) is similar to the one of item (2) in Lemma \ref{Inf_fini_compacite1}.


\subsubsection{Proof of Theorem \ref{Existence3}} \label{existence_limite_faible}
We are now ready to prove Theorem \ref{Existence3}.\\

\textbf{Step 1}: Existence of a minimizer.\\

From item (1) of Lemma \ref{Inf_fini_compacite3}, we know that $H_0$ is finite. Let us show that there exists a function which minimizes the variational problem (\ref{H0}). Let  $(f_n)_n\in\mathcal{E} ^{\mathbb{N}}$ be a minimizing sequence of (\ref{H0}). From item (a) of Lemma \ref{Inf_fini_compacite3}, there exists $\bar{f}\in\dL^1([0,2\pi]\times\br)$ such that $f_n \underset{n\to +\infty}{\rightharpoonup}\bar{f}$ in $\dL^1_w$. From item (b) of Lemma \ref{Inf_fini_compacite3}, $\phi_{f_n}$ strongly converges to $\phi_{\bar{f}}$ in $\dL^2([0,2\pi]\times\br)$ and $\phi'_{f_n}$ strongly converges to $\phi'_{\bar{f}}$ in $\dL^2([0,2\pi]\times\br)$.

In the following paragraphs, we will note $\phi_n:=\phi_{f_n}$ and $\bar{\phi}:=\phi_{\bar{f}}$.  Notice using item (2) of Lemma \ref{Inequalities1-2-3} that $(\phi_n)_n$ is a minimizing sequence of (\ref{J0}). As in the proof of Theorem \ref{Existence2}, we introduce a new minimizing sequence which has better compactness properties than $(f_n)_n$. The sequence $(f_0^{*\phi_n})_n$ is well-defined according to Lemma \ref{rearrangement}. Since $(\phi_n)_n$ is a minimizing sequence of (\ref{J0}) and using the second item of Lemma \ref{Inequalities1-2-3}, we directly get $\mathcal{H}(f_0^{*\phi_n})\underset{n\to +\infty}{\longrightarrow} H_0$. The next step is to prove that $\mathcal{H}(f_0^{*\phi_n})\underset{n\to +\infty}{\longrightarrow}\mathcal{H}(f_0^{*\bar{\phi}})$. In order to do that, let us show that $f_0^{*\phi_n}\underset{n\to +\infty}{\longrightarrow}f_0^{*\bar{\phi}}$ strongly in $\dL^1([0,2\pi]\times\br)$. From general properties of rearrangements, see \cite{Kavian} and \cite{Loss}, we have $\Vert f_0^{*\phi_n}\Vert_{\dL^1}=\Vert f_0\Vert_{\dL^1}$ and $\Vert f_0^{*\bar{\phi}}\Vert_{\dL^1}=\Vert f_0\Vert_{\dL^1}$ and therefore using Brezis-Lieb, see \cite{Brezis_Lieb}, it is sufficient to show that $f_0^{*\phi_n}\underset{n\to +\infty}{\longrightarrow}f_0^{*\bar{\phi}}$ a.e. to get the strong convergence in $\dL^1([0,2\pi]\times\br)$. Using the dominated convergence theorem, we easily get that
\[a_{\phi_n}\left(\frac{v^2}{2}+\phi_n(\theta)\right)\underset{n\to +\infty}{\longrightarrow}a_{\bar{\phi}}\left(\frac{v^2}{2}+\bar{\phi}(\theta)\right) \text{ a.e. up to a subsequence.}\] 
As by hypothesis, $f_0\in \ce\cap\mathcal{C}^0([0,2\pi]\times\mathbb{R})$, $f_0^{\#}$ is continuous then $f_0^{*\phi_n}\underset{n\to +\infty}{\longrightarrow}f_0^{*\bar{\phi}}$ a.e. up to an extraction of a subsequence. Thus, we get $\Vert f_0^{*\phi_n}-f_0^{*\bar{\phi}}\Vert_{\dL^1}\underset{n\to +\infty}{\longrightarrow}0$. Then, from classical inequality about lower semicontinuous functions (see \cite{Kavian}) and the convergence in $\dL^2([0,2\pi]\times\br)$ of $\phi_n$, we deduce that 
\vspace{-0.1cm}
\begin{equation}\label{z}
H_0\geq \iint \frac{v^2}{2}f_0^{*\bar{\phi}}(\theta,v)\mathrm{d}\theta\mathrm{d}v-\frac{1}{2}\int_0^{2\pi}\phi_{f_0^{*\bar{\phi}}}'(\theta)^2\mathrm{d}\theta=\ch(f_0^{*\bar{\phi}})
\end{equation}
Since $f_0^{*\bar{\phi}}\in \ce$ and is equimeasurable to $f_0$, we get $H_0\leq \ch(f_0^{*\bar{\phi}})$. Hence with the inequality (\ref{z}), we deduce  $H_0=\ch(f_0^{*\bar{\phi}})$ and $f_0^{*\bar{\phi}}$ is a minimizer of (\ref{H0}).\\

\textbf{Step 2}: The minimizer is a steady state of (\ref{HMF}).\\

The minimizer $f_0^{*\bar{\phi}}$ is a stationary state of the system (\ref{HMF}) and to prove that it is sufficient to show that $\bar{\phi}=\phi_{f_0^{*\bar{\phi}}}$. The proof is similar to the one of two-constraints case in Section \ref{Proof2}, we use Lemma \ref{Inequalities1-2-3} and equality (\ref{norme2ter}) to get the result.

\subsection{Orbital stability of the ground states}
\subsubsection{Proof of Theorem \ref{Stability3}}\label{OS3}

This section is devoted to the proof of Theorem \ref{Stability3}. As we do not have the uniqueness of the minimizers under constraint of equimeasurability, we can only get the orbital stability of the set of minimizers and not the orbital stability of each minimizer.\\

First, we need to the following  lemma which is at the heart of the proof of the compactness of minimizing sequences. This lemma will be proved at the end of the proof of Theorem \ref{Stability3}.
\begin{lem}\label{Bphi}
Let $f_0\in\ce\cap\mathcal{C}^0([0,2\pi]\times\mathbb{R})$ and let $(f_n)_n$ be a minimizing sequence of (\ref{H0}). Then $(f_n)_n$ has a weak limit $\bar{f}$ in $\dL^1([0,2\pi]\times\br)$. Denoting $\bar{\phi}:=\phi_{\bar{f}}$, we have
\[\int_0^{\Vert f_0\Vert_{\dL^{\infty}}}B_{\bar{\phi}}(\mu_{f_0}(t)+\beta_{f_n,f_0^{*\bar{\phi}}}(t))+B_{\bar{\phi}}(\mu_{f_0}(t)-\beta_{f_n,f_0^{*\bar{\phi}}}(t))-2B_{\bar{\phi}}(\mu_{f_0}(t))\mathrm{d}t\underset{n\to+\infty}{\longrightarrow}0\]
where 
\vspace{-0.2cm}
\begin{align}\label{beta}
\begin{cases}
&\beta_{f,g}(t)=\vert\{(\theta,v)\in [0,2\pi]\times\mathbb{R}: f(\theta,v)\leq t< g(\theta,v)\}\vert,\\
&B_{\bar{\phi}}(\mu)=\iint_{\{a_{\bar{\phi}}(\frac{v^2}{2}+\bar{\phi}(\theta))<\mu\}}\frac{v^2}{2}+\bar{\phi}(\theta)\mathrm{d}\theta\mathrm{d}v.
\end{cases}
\end{align}
\end{lem}

\textbf{Step 1}: Compactness of the minimizing sequences\\ 

Let $(f_n)_n$ be a minimizing sequence of (\ref{H0}), let us show that $(f_n)_n$ is compact in $\ce$. Using Lemma \ref{Inf_fini_compacite3}, there exists $\bar{f}\in\dL^1$ such that $f_n\underset{n\to+ \infty}{\rightharpoonup}\bar{f}$ weakly in  $\dL^1([0,2\pi]\times\br)$ and $\phi_n\underset{n\to+\infty}{\longrightarrow}\bar{\phi}$ strongly in $\dL^2([0,2\pi]\times\br)$ where $\bar{\phi}:=\phi_{\bar{f}}$. Arguing as in the proof of Theorem \ref{Existence3} in Section \ref{existence_limite_faible}, we also get  $f_0^{*\phi_n}\underset{n\to+\infty}{\longrightarrow}f_0^{*\bar{\phi}}$ strongly in $\dL^1([0,2\pi]\times\br)$. Our aim is now to show that $\Vert f_n -f_0^{*\bar{\phi}}\Vert_{\dL^1}\underset{n\to+\infty}{\longrightarrow}0$. In order to do that, we will use some techniques about rearrangements introduced in \cite{Lemou_seul}. In particular, we will use the following equality established in the proof of Theorem 1 in Section 2.3 in \cite{Lemou_seul}
\vspace{-0.1cm}
\begin{equation}\label{egalite_compacite}
\Vert f_n-f_0^{*\bar{\phi}}\Vert_{\dL^1}=2\int_0^{+\infty}\beta_{f_n,f_0^{*\bar{\phi}}}(t)\dd t+\Vert f_n\Vert_{\dL^1}-\Vert f_0\Vert_{\dL^1}
\end{equation}
where $\beta_{f,g}$ is defined in (\ref{beta}). The second term of (\ref{egalite_compacite}): $\Vert f_n\Vert_{\dL^1}-\Vert f_0\Vert_{\dL^1}$ goes to 0 when $n$ goes to infinity. Indeed, according to Definition \ref{minimizing_sequence} of a minimizing sequence, we have: $\Vert f_n^*-f_0^*\Vert_{\dL^1}\underset{n\to+\infty}{\longrightarrow}0$ then $\Vert f_n^*\Vert_{\dL^1}=\Vert f_n \Vert_{\dL^1}\underset{n\to+\infty}{\longrightarrow}\Vert f_0^*\Vert_{\dL^1}=\Vert f_0\Vert_{\dL^1}$ using rearrangements properties, see \cite{Kavian}. Hence to prove that: $\Vert f_n-f_0^{*\bar{\phi}}\Vert_{\dL^1}\underset{n\to+\infty}{\longrightarrow}0$, we need to prove that $\int_0^{+\infty}\beta_{f_n,f_0^{*\bar{\phi}}}(t)\dd t\underset{n\to+\infty}{\longrightarrow}0$. For this purpose, it is sufficient to show that $\beta_{f_n,f_0^{*\bar{\phi}}}(t)\underset{n\to+\infty}{\longrightarrow}0$. Indeed, this a direct application of the dominated convergence theorem
\begin{itemize}
\item $\beta_{f_n,f_0^{*\bar{\phi}}}(t)\underset{n\to+\infty}{\longrightarrow}0$,
\item $0\leq \beta_{f_n,f_0^{*\bar{\phi}}}(t)\leq \mu_{f_0}(t)$ and $\int_0^{+\infty}\mu_{f_0}(t)\dd t=\Vert f_0\Vert_{\dL^1}$ using Fubini's theorem.
\end{itemize}
To get the a.e. convergence to 0 of $\beta_{f_n,f_0^{*\bar{\phi}}}(t)$, we will use Lemma \ref{Bphi}. By convexity of $B_{\bar{\phi}}$ given by Theorem 1 in \cite{Lemou_seul}, 
\vspace{-0.1cm}
\[B_{\bar{\phi}}(\mu_{f_0}(t)+\beta_{f_n,f_0^{*\bar{\phi}}}(t))+B_{\bar{\phi}}(\mu_{f_0}(t)-\beta_{f_n,f_0^{*\bar{\phi}}}(t))-2B_{\bar{\phi}}(\mu_{f_0}(t))\geq 0\] 
therefore Lemma \ref{Bphi} implies that 
\vspace{-0.1cm}
\[B_{\bar{\phi}}(\mu_{f_0}(t)+\beta_{f_n,f_0^{*\bar{\phi}}}(t))+B_{\bar{\phi}}(\mu_{f_0}(t)-\beta_{f_n,f_0^{*\bar{\phi}}}(t))-2B_{\bar{\phi}}(\mu_{f_0}(t)) \underset{n\to+\infty}{\longrightarrow}0 \text{ for almost } t\geq 0.\]
Notice that $\beta_{f_n,f_0^{*\bar{\phi}}}(0)=0$ and for all $t>0$,
\vspace{-0.2cm}
\[0<\beta_{f_n,f_0^{*\bar{\phi}}}(t)\leq \frac{1}{t}\Vert f\Vert_{\dL^1}.\] 
Thus the sequence $(\beta_{f_n,f_0^{*\bar{\phi}}}(t))_n$ is bounded and has a convergent subsequence. Let us suppose that $\beta_{f_n,f_0^{*\bar{\phi}}}(t)\underset{n\to+\infty}{\longrightarrow}l\neq 0$, then by strict convexity of $B_{\bar{\phi}}$,
\begin{align*}
B_{\bar{\phi}}(\mu_{f_0}(t)&+\beta_{f_n,f_0^{*\bar{\phi}}}(t))+B_{\bar{\phi}}(\mu_{f_0}(t)-\beta_{f_n,f_0^{*\bar{\phi}}}(t))-2B_{\bar{\phi}}(\mu_{f_0}(t)) \\
&\underset{n\to+\infty}{\longrightarrow}B_{\bar{\phi}}(\mu_{f_0}(t)+l)+B_{\bar{\phi}}(\mu_{f_0}(t)-l)-2B_{\bar{\phi}}(\mu_{f_0}(t))>0.
\end{align*} 
Absurd then $\beta_{f_n,f_0^{*\bar{\phi}}}(t)\underset{n\to+\infty}{\longrightarrow}0$ for almost $t\geq 0$. 
Hence $\Vert f_n-f_0^{*\bar{\phi}}\Vert_{\dL^1}\underset{n\to+\infty}{\longrightarrow}0$. Besides we have proved that $f_n\underset{n\to+\infty}{\rightharpoonup}\bar{f}$ weakly in $\dL^1([0,2\pi]\times\br)$, hence by uniqueness of the limit, we get $f_0^{*\bar{\phi}}=\bar{f}$. Since by definition, a minimizing sequence is uniformly bounded, to prove the compactness of the sequence $(f_n)_n$ in the energy space $\ce$, there remains show that
\[\Vert v^2(f_n-\bar{f})\Vert_{\dL^1}\underset{n\to+\infty}{\longrightarrow}0 .\]  
Notice that
\vspace{-0.2cm}
\[\iint v^2(f_n(\theta,v)-\bar{f}(\theta,v))\mathrm{d}\theta\mathrm{d}v=2(\mathcal{H}(f_n)-\mathcal{H}(\bar{f}))+\Vert\phi'_n\Vert_{\dL^2}^2-\Vert \bar{\phi}'\Vert_{\dL^2}^2,\]
thus $\Vert v^2 f_n\Vert_{\dL^1}\underset{n\to+\infty}{\longrightarrow}\Vert v^2\bar{f}\Vert_{\dL^1}$ since $(f_n)_n$ is a minimizing sequence and $\bar{f}$ is a minimizer. Moreover $v^2 f_n \underset{n\to+\infty}{\longrightarrow}v^2\bar{f}$ up to an extraction of a subsequence since $f_n \underset{n\to+\infty}{\longrightarrow}\bar{f}$ strongly in $\dL^1$. Thanks to Brezis Lieb's lemma (see \cite{Brezis_Lieb}), we deduce that $\Vert v^2(f_n-\bar{f})\Vert_{\dL^1}\underset{n\to+\infty}{\longrightarrow}0$. To conclude, we have proved that the sequence $(f_n)_n$ is compact in $\mathcal{E}$.\\

\textbf{Step 2}:  Proof of the orbital stability \\

Let us argue by contradiction, let $f_{i_0}$ be a steady state of (\ref{HMF}) which minimizes (\ref{H0}). Assume that $f_{i_0}$ is orbitally unstable. Then there exist $\varepsilon_0>0$, a sequence $(f_{init}^n)_n\in\ce^{\bn}$ and a sequence $(t_n)_n\in(\mathbb{R}_*^+)^{\bn}$ such that $ f_{init}^n\xrightarrow{\ce}f_{i_0}$ and for all $n$, for all $\theta_0\in [0,2\pi]$, for all $f_i$ minimizer of (\ref{H0}), 
\begin{align}\label{etoile}
\begin{cases}
&\Vert f^n(t_n,\theta+\theta_0,v)-f_{i}(\theta,v)\Vert_{\dL^1}>\varepsilon_0,\\
&\text{or }\Vert v^2(f^n(t_n,\theta+\theta_0,v)-f_{i}(\theta,v))\Vert_{\dL^1}>\varepsilon_0,\\
\end{cases}
\end{align}
where $f^n(t_n,\theta,v)$ is a solution to (\ref{HMF}) with initial data $f_{init}^n$. Let $g_n(\theta,v)=f^n(t_n,\theta,v)$. Notice that 
\begin{align*}
\Vert (f_{init}^n)^*-f_0^*\Vert_{\dL^1}&=\Vert (f_{init}^n)^*-f_{i_0}^*\Vert_{\dL^1}\quad\text { since }f_{i_0}\in Eq(f_0),\\
&\leq \Vert f_{init}^n-f_{i_0}\Vert_{\dL^1}\quad  \text{ by contractivity of rearrangement (see \cite{Kavian}), }
\end{align*} 
but from conservation properties of the flow (\ref{HMF}), we have $g_n ^*\!=\!(f_{init}^n)^*$ together with \mbox{$\Vert g_n \Vert_{\dL^{\infty}}\!=\!\Vert f_{init}^n\Vert_{\dL^{\infty}}$.} Therefore $g_n^* \underset{n\to+\infty}{\longrightarrow}f_0^*$ strongly in $\dL^1$ and $(g_n)_n$ is uniformly bounded. Finally, from item (2) of Lemma \ref{Inequalities1-2-3} and from the conservation property of the flow (\ref{HMF}), we have
\[H_0\leq \mathcal{H}(f_0^{*\phi_{g_n}})\leq\mathcal{H}(g_n)\leq\mathcal{H}(f_{init}^n)\underset{n\to+\infty}{\longrightarrow}H_0.\]
Thus $\mathcal{H}(g_n)\underset{n\to+\infty}{\longrightarrow}H_0$ and the sequence $(g_n)_n$ is a minimizing sequence of (\ref{H0}). According to the previous step, this sequence is compact, hence, up to an extraction of a subsequence, there exists $f_{I}\in\ce$ such that $ g_n\xrightarrow{\ce}f_{I}$. This implies that 
\begin{equation}\label{converge}
\Vert g_n-f_{I}\Vert_{\dL^1}\underset{n\to+\infty}{\longrightarrow}0\quad \text{ and } \quad \Vert v^2(g_n -f_I)\Vert_{\dL^1}\underset{n\to+\infty}{\longrightarrow}0.
\end{equation}
Arguing as in the proof of Theorem \ref{Orbital_stability1} in Section \ref{OS1}, we prove that $\ch(f_I)= H_0$ and that $f_I$ is equimeasurable to $f_{i_0}$. We deduce that $f_I$ is equimeasurable to $f_0$ and hence this is a minimizer of (\ref{H0}). We get a contradiction with (\ref{converge}) and (\ref{etoile}). There remains to show Lemma \ref{Bphi}.

\begin{proof}[Proof of Lemma \ref{Bphi}]
The existence of the weak limit $\bar{f}$ is given by item (3) of Lemma \ref{Inf_fini_compacite3}. Many techniques in this proof have been introduced in \cite{Lemou_seul}.
By convexity of $B_{\bar{\phi}}$, see Theorem 1 in \cite{Lemou_seul}, we have 
\[\int_0^{\Vert f_0\Vert_{\dL^{\infty}}}B_{\bar{\phi}}(\mu_{f_0}(t)+\beta_{f_n,f_0^{*\bar{\phi}}}(t))+B_{\bar{\phi}}(\mu_{f_0}(t)-\beta_{f_n,f_0^{*\bar{\phi}}}(t))-2B_{\bar{\phi}}(\mu_{f_0}(t))\mathrm{d}t\geq 0.\]
Using the remark following Theorem 1 in \cite{Lemou_seul} , we have 
\[\int_0^{\Vert f_0\Vert_{\dL^{\infty}}}B_{\bar{\phi}}(\mu_{f_0}(t)+\beta_{f_n,f_0^{*\bar{\phi}}}(t))+B_{\bar{\phi}}(\mu_{f_0}(t)-\beta_{f_n,f_0^{*\bar{\phi}}}(t))-2B_{\bar{\phi}}(\mu_{f_0}(t))\mathrm{d}t
\leq A_n + B_n\]
where 
\begin{align*}
\begin{cases}
&A_n=\int_0^{2\pi}\int_{\mathbb{R}}\left(\frac{v^2}{2}+\bar{\phi}(\theta)\right)(f_n(\theta,v)-f_0^{*\bar{\phi}}(\theta,v))\mathrm{d}\theta\mathrm{d}v,\\
&B_n=\int_0^{+\infty}[a_{\bar{\phi}}^{-1}(2\mu_{f_0}(s))\beta_{f_n^*,f_0^*}(s)-a_{\bar{\phi}}^{-1}(\mu_{f_0}(s))\beta_{f_0^*,f_n^*}(s)]\mathrm{d}s.
\end{cases}
\end{align*}
Then let us show that $A_n \underset{n\to+\infty}{\longrightarrow}0$. After integrating by parts, we get
\[A_n=\int_0^{2\pi}\int_{\mathbb{R}}\left(\frac{v^2}{2}+\bar{\phi}(\theta)\right)(f_n(\theta,v)-f_0^{*\bar{\phi}}(\theta,v))\mathrm{d}\theta\mathrm{d}v=\mathcal{H}(f_n)-\mathcal{H}(f_0^{*\bar{\phi}})+\frac{1}{2}\Vert \phi_n'-\bar{\phi}'\Vert_{\dL^2}^2.\]
We have seen in Step 1 of the proof of Theorem \ref{Existence3} in Section \ref{existence_limite_faible} that $\ch(f_n)-\ch(f_0^{*\bar{\phi}})$ converges to 0 and $\Vert \phi_n'-\bar{\phi}'\Vert_{\dL^2}^2\underset{n\to+\infty}{\longrightarrow}0$; therefore $A_n\underset{n\to+\infty}{\longrightarrow}0$. Finally let us show that $B_n\underset{n\to+\infty}{\longrightarrow}0$. We have the following inequality using inequality (\ref{inegalite_aphi})
\vspace{-0.2cm}
\begin{align*}
B_n &=\int_0^{+\infty}[a_{\bar{\phi}}^{-1}(2\mu_{f_0}(s))\beta_{f_n^*,f_0^*}(s)-a_{\bar{\phi}}^{-1}(\mu_{f_0}(s))\beta_{f_0^*,f_n^*}(s)]\mathrm{d}s\\
&\leq \int_0^{+\infty}\left(\frac{4\mu_{f_0}(s)^2}{32\pi^2}+\max \bar{\phi}\right)\beta_{f_n^*,f_0^*}(s)-\left(\frac{\mu_{f_0}(s)^2}{32\pi^2}+\min \bar{\phi}\right)\beta_{f_0^*,f_n^*}(s)\mathrm{d}s.
\end{align*}
Using the following identity, see the proof of Proposition 4.1 in \cite{Stability_HMFcos},
\[\int_0^{+\infty}\beta_{f_0^*,f_n^*}(s)\mathrm{d}s+\int_0^{+\infty}\beta_{f_n^*,f_0^*}(s)\mathrm{d}s=\Vert f_n^*-f_0^*\Vert_{\dL^1},\] 
we get
\vspace{-0.2cm}
\[B_n\leq \frac{1}{8\pi^2}\int_0^{+\infty}\mu_{f_0}(s)^2\beta_{f_n^*,f_0^*}(s)\mathrm{d}s+(\max \bar{\phi}+\min \bar{\phi})\int_0^{+\infty}\beta_{f_n^*,f_0^*}(s)\mathrm{d}s-\min \bar{\phi}\Vert f_n^*-f_0^*\Vert_{\dL^1}.\]
Notice that $\min\bar{\phi}\Vert f_n^*-f_0^*\Vert_{\dL^1}\underset{n\to+\infty}{\longrightarrow}0$ since $(f_n)_n$ is a minimizing sequence of (\ref{H0}). Besides 
\vspace{-0.3cm}
\begin{align*}
(\max\bar{\phi}+\min\bar{\phi})\int_0^{+\infty}\beta_{f_n^*,f_0^*}(s)\mathrm{d}s&\leq \max \bar{\phi}\int_0^{+\infty}\beta_{f_n^*,f_0^*}(s)\mathrm{d}s\\
&\leq \max \bar{\phi}\int_0^{+\infty}(f_n^*-f_0^*)_+\mathrm{d}s\\
&\leq \max \bar{\phi} \Vert f_n^*-f_0^*\Vert_{\dL^1}\underset{n\to+\infty}{\longrightarrow}0.
\end{align*}
Finally, let us prove that
\[\frac{1}{8\pi^2}\int_0^{+\infty}\mu_{f_0}(s)^2\beta_{f_n^*,f_0^*}(s)\mathrm{d}s\underset{n\to+\infty}{\longrightarrow}0.\] 
First notice that $\beta_{f_n^*,f_0^*}(s)\underset{n\to+\infty}{\longrightarrow}0$. Indeed we shall apply the dominated convergence theorem to $\beta_{f_n^*,f_0^*}(s)=\iint \mathds{1}_{\{f_n^*(\theta,v)\leq s <f_0^*(\theta,v)\}}\mathrm{d}\theta\mathrm{d}v$ for $s>0$. We first have
\begin{itemize}
\item $\mathds{1}_{\{f_n^*(\theta,v)\leq s <f_0^*(\theta,v)\}}\underset{n\to+\infty}{\longrightarrow}\mathds{1}_{\{f_0^*(\theta,v)\leq s <f_0^*(\theta,v)\}}$ a.e. since $f_n^*\underset{n\to+\infty}{\longrightarrow}f_0^*$ strongly in $\dL^1([0,2\pi]\times\br)$,
\item $\mathds{1}_{\{f_n^*(\theta,v)\leq s <f_0^*(\theta,v)\}}\leq \mathds{1}_{\{s <f_0^*(\theta,v)\}}$. But $\iint  \mathds{1}_{\{s <f_0^*(\theta,v)\}}\mathrm{d}\theta\mathrm{d}v=\mu_{f_0}^*(s)=\mu_{f_0}(s)<\infty$ since $f_0\in \dL^1([0,2\pi]\times\br)$.
\end{itemize}
Hence by the dominated convergence theorem, we get for all $s>0$, $\beta_{f_n^*,f_0^*}(s)\underset{n\to+\infty}{\longrightarrow}0$. For $s=0$, $\beta_{f_n^*,f_0^*}(0)=\vert \varnothing \vert =0$, thus for all $s\geq 0$, $\beta_{f_n^*,f_0^*}(s)\underset{n\to+\infty}{\longrightarrow}0$. There remains to dominate the term $\mu_{f_0}(s)^2\beta_{f_n^*,f_0^*}(s)$. Notice that $\mu_{f_0}(s)^2\beta_{f_n^*,f_0^*}(s)\leq \mu_{f_0}(s)^3$. However we have
\[\int_0^{+\infty}s^2f_0^{\#}(s)\mathrm{d}s=\int_0^{+\infty}\left(\int_{0\leq s<\mu_{f_0}(t)}s^2\mathrm{d}s\right)\mathrm{d}t=\frac{1}{3}\int_0^{+\infty}\mu_{f_0}(t)^3\mathrm{d}t.\]
So to prove the integrability of $s\to \mu_{f_0}(s)^3$, it is sufficient to show that $\int_0^{+\infty}\!\!s^2f_0^{\#}(s)\mathrm{d}s\!<\!\infty$. Using equality (\ref{inegalite_aphi}), identity $\int f_0^{\#}(s)\mathrm{d}s=\Vert f_0\Vert_{\dL^1}$ and Lemma \ref{classic_egalite}, we get 
\begin{align*}
\int_0^{+\infty}s^2 f_0^{\#}(s)\mathrm{d}s &\lesssim \int_0^{+\infty}(a_{\bar{\phi}}^{-1}(s)+1)f_0^{\#}(s)\mathrm{d}s\\
&=\int_0^{+\infty}a_{\bar{\phi}}^{-1}(s)f_0^{\#}(s)\mathrm{d}s+\Vert f_0\Vert_{\dL^1},\\
&=\iint \left(\frac{v^2}{2}+\bar{\phi}(\theta)\right)f_0^{*\bar{\phi}}(\theta,v)\mathrm{d}\theta\mathrm{d}v+\Vert f_0\Vert_{\dL^1}<+\infty
\end{align*}
since $f_0^{*\bar{\phi}}$ satisfies $H_0=\mathcal{H}(f_0^{*\bar{\phi}})$ and $f_0\in \dL^1([0,2\pi]\times\br)$. Hence $\int_0^{+\infty}\mu_{f_0}(t)^3\mathrm{d}t<+\infty$. We conclude by dominated convergence that 
\vspace{-0.1cm}
\[\int_0^{+\infty}\mu_{f_0}(s)^2\beta_{f_n^*,f_0^*}(s)\mathrm{d}s\underset{n\to+\infty}{\longrightarrow}0.\] 
Therefore $B_n\underset{n\to+\infty}{\longrightarrow}0$ and the lemma is proved.
\end{proof}
\subsubsection{Expression of the minimizers}

From the proof of compactness of minimizing sequences in Section \ref{OS3}, we can deduce the expression of the steady states of (\ref{HMF}) which minimizes (\ref{H0}). Indeed, we have proved that any minimizing sequences $(f_n)_n$ converge to a minimizer $\bar{f}$ in $\ce$ which satisfies $\bar{f}=f_0^{*\bar{\phi}}$. Hence any minimizer $\bar{f}$ of (\ref{H0}) has the following expression:

\[\bar{f}=f_0^{\#}\left(a_{\bar{\phi}}\left(\frac{v^2}{2}+\bar{\phi}(\theta)\right)\right).\]

\appendix
\phantomsection
\renewcommand{\thesection}{\Alph{section}}
\renewcommand{\thesubsection}{\Alph{section}.\arabic{subsection}}
\renewcommand{\theequation}{\Alph{section}.\arabic{equation}}
\section{}

\begin{proof}[Proof of Lemma \ref{sci_bis}] 
Let $(f_n)_n$ be a sequence of nonnegative functions converging weakly in $\dL^1([0,2\pi]\times\br)$ to $\bar{f}$ such that $\Vert f_n\Vert_{\dL^1}=M$, $\Vert v^2 f_n\Vert_{\dL^1}\leq C_1$ and $\iint f_n\ln(f_n)\leq C_2$ where $M$, $C_1$ and $C_2$ do not depend on $n$. Let $\lambda\in\br_+$ and $f_1(\theta,v)=e^{-\vert v\vert}$, we have
\vspace{-0.1cm}
\begin{align*}
\iint f_n\ln(f_n) &=\iint f_n\ln\left(\frac{f_n}{\lambda f_1}\right)+\ln(\lambda)\iint f_n+\iint f_n\ln(f_1)\\
&=\iint_{\{0\leq f_n\leq \lambda f_1\}} f_n\ln\left(\frac{f_n}{\lambda f_1}\right)+\iint \left(f_n\ln\left(\frac{f_n}{\lambda f_1}\right)\right)_+ +\ln(\lambda)M+\iint f_n\ln(f_1).\\
\end{align*}
\vspace{-0.2cm}
First by using the lower semicontinuity properties of convex positive functions, we get
\vspace{-0.1cm}
\[\underset{n\to+\infty}{\liminf}\iint \left(f_n\ln\left(\frac{f_n}{\lambda f_1}\right)\right)_+\geq \iint \left(\bar{f}\ln\left(\frac{\bar{f}}{\lambda f_1}\right)\right)_+.\]
At this stage, we have the following identity
\begin{align}\label{A1}
\underset{n\to+\infty}{\liminf}\iint f_n\ln(f_n)\geq &\left[\iint \left(\bar{f}\ln\left(\frac{\bar{f}}{\lambda f_1}\right)\right)_+ +\ln(\lambda)M\right]+\underset{n\to+\infty}{\liminf}\iint f_n\ln(f_1)\\
&+\underset{n\to+\infty}{\liminf}\iint_{\{0\leq f_n\leq \lambda f_1\}} f_n\ln\left(\frac{f_n}{\lambda f_1}\right).\nonumber
\end{align}
Let us then show that 
\vspace{-0.3cm}
\begin{equation}\label{limite}
\lim_{\lambda\to 0}\sup_n\left\vert\iint_{\{0\leq f_n\leq \lambda f_1\}} f_n\ln\left(\frac{f_n}{\lambda f_1}\right)\right\vert=0.
\end{equation}
This term can be written as
\[
\iint_{\{0\leq f_n\leq \lambda f_1\}} f_n\ln\left(\frac{f_n}{\lambda f_1}\right)=\iint_{\{0\leq f_n\leq \lambda f_1\}} f_n\ln\left(\frac{f_n}{ f_1}\right)-\ln(\lambda)\iint_{\{0\leq f_n\leq \lambda f_1\}} f_n=T_1+T_2.
\]
We have $\vert T_2\vert \leq \lambda\vert \ln(\lambda)\vert M_1 \underset{\lambda\to 0}{\longrightarrow}0 \text{ uniformly in } n$ where $M_1=\Vert f_1\Vert_{\dL^1}$. Since for $\lambda$ sufficiently small, the function $x\to x\vert\ln(x)\vert$ is increasing on $[0, \lambda f_1]$, we have for $T_1$
\begin{align*}
\vert T_1\vert &\leq \iint_{\{0\leq f_n\leq \lambda f_1\}}f_n\vert \ln(f_1)\vert+\iint_{\{0\leq f_n\leq \lambda f_1\}}f_n\vert \ln(f_n)\vert\\
&\leq \lambda\iint f_1\vert \ln(f_1)\vert+\lambda\iint f_1\vert\ln(\lambda f_1)\vert \leq 2\lambda\iint f_1\vert \ln(f_1)\vert+\lambda\vert\ln(\lambda)\vert M_1.
\end{align*}
Clearly, we have $\iint f_1\vert\ln(f_1)\vert<+\infty$ so $\vert T_1\vert \underset{\lambda\to 0}{\longrightarrow}0$ uniformly in $n$. So far, we have
\begin{equation}\label{w1}
\underset{n\to+\infty}{\liminf}\iint f_n\ln(f_n)\geq \underset{\lambda\to 0}{\lim}\left[\iint\left(\bar{f}\ln\left(\frac{\bar{f}}{\lambda f_1}\right)\right)_+ + \ln(\lambda)M\right]+\underset{n\to+\infty}{\liminf}\iint f_n\ln(f_1).
\end{equation}
The next step is to show that $\underset{\lambda\to 0}{\lim}\left[\iint\left(\bar{f}\ln\left(\frac{\bar{f}}{\lambda f_1}\right)\right)_+ +\ln(\lambda)M\right]=\iint \bar{f}\ln\left(\frac{\bar{f}}{f_1}\right).$
We have
\vspace{-0.5cm}
\begin{equation}\label{w2}
\left\vert\!\iint\!\!\left(\bar{f}\ln\left(\frac{\bar{f}}{\lambda f_1}\right)\!\!\!\right)_+ \!\!\!\!\!+\ln(\lambda)M-\!\!\iint \bar{f}\ln\left(\frac{\bar{f}}{f_1}\right)\right\vert\leq \left\vert\iint_{\{\bar{f}\geq \lambda f_1\}}\!\!\!\!\bar{f}\ln\left(\frac{\bar{f}}{f_1}\right)\!-\!\iint \bar{f}\ln\left(\frac{\bar{f}}{f_1}\right)\right\vert+\lambda\vert\ln(\lambda)\vert M_1.
\end{equation}
\vspace{-0.1cm}
Let us show, using the dominated convergence theorem, that the first term of (\ref{w2}) converges to 0 when $\lambda$ goes to 0. The term $\bar{f}\ln(\frac{\bar{f}}{f_1})\mathds{1}_{\{\bar{f}\geq \lambda f_1\}}$ clearly converges to $\bar{f}\ln(\frac{\bar{f}}{f_1})$. So it remains to show that $\iint\vert \bar{f}\ln(\frac{\bar{f}}{f_1})\vert\dd\theta\dd v<+\infty$. We have
\vspace{-0.2cm}
\begin{align*}
\iint\left\vert \bar{f}\ln\left(\frac{\bar{f}}{f_1}\right)\right\vert\dd\theta\dd v&\leq \iint\vert\bar{f}\ln(\bar{f})\vert\dd\theta\dd v+\iint \vert \bar{f}\ln(f_1)\vert\dd\theta\dd v\\
&\leq \iint\vert\bar{f}\ln(\bar{f})\vert\dd\theta\dd v+ M+\Vert v^2\bar{f}\Vert_{\dL^1}.
\end{align*}
It is well-knwon, see \cite{Villani_Desvillettes}, that for $\bar{f}\geq 0$, if $\Vert \bar{f}\Vert_{\dL^1}\!\!<\!\!+\infty$, $\Vert v^2\bar{f}\Vert_{\dL^1}\!\!<\!\!+\infty$, $\vert\!\iint\!\! \bar{f}\ln(\bar{f})\dd\theta\dd v\vert\!\!<\!\!+\infty$, we have $\iint\vert\bar{f}\ln(\bar{f})\vert\dd\theta\dd v<+\infty$. We already have that $\Vert \bar{f}\Vert_{\dL^1}<+\infty$, $\Vert v^2\bar{f}\Vert_{\dL^1}<+\infty$, so let us show that $\vert \iint \bar{f}\ln(\bar{f})\dd\theta\dd v\vert<+\infty$. Thanks to Jensen's inequality (\ref{jensen}), we have
\vspace{-0.1cm}
\[\iint \bar{f}\ln(\bar{f})\dd\theta\dd v\geq M(\ln(M)-\ln(M_1))-\iint \vert v\vert\bar{f}>-\infty.\]
By hypothesis, we know that $\underset{n\to + \infty}{\liminf}\iint f_n\ln(f_n)\dd\theta\dd v\leq C_2$ and with inequality (\ref{A1}) and limit (\ref{limite}), we get for all $\lambda \in\br_+$

\[ C_2\geq \iint_{\{\bar{f}\geq \lambda f_1\}}\bar{f}\ln(\bar{f})\dd\theta\dd v+\ln(\lambda)\iint_{\{\bar{f}\leq \lambda f_1\}}\bar{f}\dd\theta\dd v-\iint \vert v\vert \bar{f}.\]
The two last terms are bounded so $\iint_{\{\bar{f}\geq \lambda f_1\}}\bar{f}\ln(\bar{f})\dd\theta\dd v$ is bounded from above et we deduce that $\iint \bar{f}\ln(\bar{f})\dd\theta\dd v$ is bounded from above. So the dominated convergence theorem gives the limit. Then the second term of (\ref{w2}) clearly converges to 0. So
\[\underset{n\to+\infty}{\liminf}\iint f_n\ln(f_n)\geq\iint \bar{f}\ln(\bar{f})+\underset{n\to+\infty}{\liminf}\iint (f_n-\bar{f})\ln(f_1).\]
To conclude, it is sufficient to show that $\iint (f_n-\bar{f})\ln(f_1)\underset{n\to +\infty}{\longrightarrow}0$. Let $\varepsilon>0$ and $R>0$ such that $\frac{2C_1}{R}\leq \varepsilon$, we have
\vspace{-0.1cm}
\begin{align*}
\left\vert \iint (f_n-\bar{f})\ln(f_1)\right\vert&\leq \left\vert \iint_{\{\vert v\vert \leq R\}}(f_n-\bar{f})\vert v\vert\dd\theta\dd v\right\vert+\iint_{\{\vert v\vert>R\}}(f_n+\bar{f})\vert v\vert\dd\theta\dd v\\
&\leq \left\vert \iint_{\{\vert v\vert \leq R\}}(f_n-\bar{f})\vert v\vert\dd\theta\dd v\right\vert+\frac{1}{R}\iint v^2(f_n+\bar{f})\dd\theta\dd v\\
&\leq \left\vert \iint_{\{\vert v\vert \leq R\}}(f_n-\bar{f})\vert v\vert\dd\theta\dd v\right\vert+\frac{2C_1}{R}.
\end{align*}

\noindent The first term converges to 0 when $n$ goes to infinity thanks to the weak convergence in $\dL^1([0,2\pi]\times\br)$ of $f_n$ to $\bar{f}$ and $R$ is chosen such that the second term is smaller than $\varepsilon$.

\end{proof}

\subsection*{Acknowledgements}
The authors acknowledge supports from the ANR project Moonrise ANR-14-CE23-0007-01,  and from the INRIA project IPSO.

\bibliographystyle{plain}
\bibliography{biblio-article}

\begin{thebibliography}{10}

\bibitem{Antoni_Ruffo}
Mickael Antoni and Stefano Ruffo.
\newblock Clustering and relaxation in hamiltonian long-range dynamics.
\newblock {\em Phys. Rev. E}, 52:2361--2374, Sep 1995.

\bibitem{antoniazzi}
A.~{Antoniazzi}, D.~{Fanelli}, S.~{Ruffo}, and Y.~Y. {Yamaguchi}.
\newblock {Nonequilibrium Tricritical Point in a System with Long-Range
  Interactions}.
\newblock {\em Physical Review Letters}, 99(4):040601, July 2007.

\bibitem{barre:hal-00018773}
Julien Barr{\'e}, Freddy Bouchet, Thierry Dauxois, Stefano Ruffo, and
  Yoshiyuki~Y. Yamaguchi.
\newblock {The Vlasov equation and the Hamiltonian mean-field model}.
\newblock {\em {Physica A}}, 365:177--183, 2006.

\bibitem{BarreOlivettiYamaguchiDynamics}
Julien Barr\'e, Alain Olivetti, and Yoshiyuki~Y Yamaguchi.
\newblock Dynamics of perturbations around inhomogeneous backgrounds in the hmf
  model.
\newblock {\em Journal of Statistical Mechanics: Theory and Experiment},
  2010(08):P08002, 2010.

\bibitem{BarreOlivettiYamaguchidamping}
Julien Barr{\'e}, Alain Olivetti, and Yoshiyuki~Y. Yamaguchi.
\newblock {Algebraic damping in the one-dimensional Vlasov equation}.
\newblock {\em {Journal of Physics A: Mathematical and Theoretical}},
  44:405502, 2011.

\bibitem{BarreYamaguchi}
Julien Barr{\'e} and Yoshiyuki Yamaguchi.
\newblock {On the neighborhood of an inhomogeneous stable stationary solution
  of the Vlasov equation -Case of the Hamiltonian mean-field model}.
\newblock {\em {Journal of Mathematical Physics}}, 56(081502), August 2015.

\bibitem{PhysRevE.79.036208}
Julien Barr\'e and Yoshiyuki~Y. Yamaguchi.
\newblock Small traveling clusters in attractive and repulsive hamiltonian
  mean-field models.
\newblock {\em Phys. Rev. E}, 79:036208, Mar 2009.

\bibitem{Brezis_Lieb}
H.~Brezis and E.~Lieb.
\newblock A relation between pointwise convergence of functions and convergence
  of functionals.
\newblock {\em Proceedings of the American Mathematical Society}, 88:486--490,
  1983.

\bibitem{chavanis1}
P.-H. {Chavanis} and A.~{Campa}.
\newblock {Inhomogeneous Tsallis distributions in the HMF model}.
\newblock {\em European Physical Journal B}, 76:581--611, August 2010.

\bibitem{Chavanis-Vatteville-Bouchet}
P.~H. {Chavanis}, J.~{Vatteville}, and F.~{Bouchet}.
\newblock {Dynamics and thermodynamics of a simple model similar to
  self-gravitating systems: the HMF model}.
\newblock {\em European Physical Journal B}, 46:61--99, July 2005.

\bibitem{Chavanis-Delfini}
Pierre-Henri Chavanis and Luca Delfini.
\newblock {Phase transitions in self-gravitating systems and bacterial
  populations with a screened attractive potential}.
\newblock {\em {Physical Review E : Statistical, Nonlinear, and Soft Matter
  Physics}}, 81(5):051103, 2010.

\bibitem{Villani_Desvillettes}
I.~Desvillettes and C.~Villani.
\newblock On the spatially homogeneous landau equation for hard potentials part
  i: Existence, uniqueness and smoothness.
\newblock {\em Comm. P.D.E.}, 25:179--259, 2000.

\bibitem{Dunford-Pettis}
R.~E. Edwards.
\newblock {\em Functional analysis: theory and applications}.
\newblock Holt, Rinehart and Winston, 1965.

\bibitem{faou_rousset}
E.~{Faou} and F.~{Rousset}.
\newblock {Landau Damping in Sobolev Spaces for the Vlasov-HMF Model}.
\newblock {\em Archive for Rational Mechanics and Analysis}, 219:887--902,
  February 2016.

\bibitem{Kavian}
O.~Kavian.
\newblock {\em Introduction \`a la th\'eorie des points critiques et
  application aux probl\`emes elliptiques}.
\newblock Springer, 1993.

\bibitem{Lemou_seul}
M.~{Lemou}.
\newblock {Extended Rearrangement Inequalities and Applications to Some
  Quantitative Stability Results}.
\newblock {\em Communications in Mathematical Physics}, 348:695--727, December
  2016.

\bibitem{Stability_HMFcos}
Mohammed Lemou, Ana~Maria Luz, and Florian M{\'e}hats.
\newblock {Nonlinear stability criteria for the HMF Model}.
\newblock {\em {Archive for Rational Mechanics and Analysis}}, 224(2):353--380,
  2017.

\bibitem{lemou-mehats-raphael}
Mohammed Lemou, Florian M{\'e}hats, and Pierre Raphael.
\newblock {The Orbital Stability of the Ground States and the Singularity
  Formation for the Gravitational Vlasov Poisson System}.
\newblock {\em {Archive for Rational Mechanics and Analysis}}, 189(3):425--468,
  2008.

\bibitem{Stability_VP}
Mohammed Lemou, Florian M{\'e}hats, and Pierre Raphael.
\newblock {Orbital stability of spherical galactic models}.
\newblock {\em {Inventiones Mathematicae}}, 187(1):145--194, 2012.

\bibitem{Manev}
Mohammed Lemou, Florian M{\'e}hats, and Cyril Rigault.
\newblock {Stable ground states and self-similar blow-up solutions for the
  gravitational Vlasov-Manev system}.
\newblock {\em {SIAM Journal on Mathematical Analysis}}, 44(6):3928--3968,
  2012.

\bibitem{Loss}
E.~H. Lieb and M.~Loss.
\newblock {\em Analysis}.
\newblock American Mathematical Society, 1997.

\bibitem{Messer1982}
Joachim Messer and Herbert Spohn.
\newblock Statistical mechanics of the isothermal lane-emden equation.
\newblock {\em Journal of Statistical Physics}, 29(3):561--578, 1982.

\bibitem{Nagai}
Toshitaka Nagai and Takasi Senba.
\newblock Behavior of radially symmetric solutions of a system related to
  chemotaxis.
\newblock {\em Nonlinear Analysis: Theory, Methods and Applications},
  30(6):3837 -- 3842, 1997.
\newblock Proceedings of the Second World Congress of Nonlinear Analysts.

\bibitem{ogawa}
S.~{Ogawa}.
\newblock {Spectral and formal stability criteria of spatially inhomogeneous
  stationary solutions to the Vlasov equation for the Hamiltonian mean-field
  model}.
\newblock {\em Phys. Rev. E}, 87(6):062107, June 2013.

\bibitem{yama2}
S.~{Ogawa} and Y.Y. {Yamaguchi}.
\newblock {Precise determination of the nonequilibrium tricritical point based
  on Lynden-Bell theory in the Hamiltonian mean-field model.}
\newblock {\em Phys. Rev. E.}, 84, 2011.

\bibitem{staniscia}
F.~{Staniscia}, P.H. {Chavanis}, and G.~{De Ninno}.
\newblock {Out-of-equilibrium phase transitions in the HMF model : a closer
  look.}
\newblock {\em Phys. Rev. E.}, 83, 2011.

\bibitem{Csiszar_Kullback}
Andreas Unterreiter, Anton Arnold, Peter Markowich, and Giuseppe Toscani.
\newblock On generalized csisz{\'a}r-kullback inequalitieys.
\newblock {\em Monatshefte f{\"u}r Mathematik}, 131(3):235--253, Dec 2000.

\bibitem{yamaguchi:hal-00008414}
Y.~Y. Yamaguchi, Julien Barr{\'e}, Freddy Bouchet, Thierry Dauxois, and
  S.~Ruffo.
\newblock {Stability criteria of the Vlasov equation and quasi-stationary
  states of the HMF model}.
\newblock {\em {Physica A}}, 337:36--66, 2004.

\bibitem{yama1}
Y.Y. {Yamaguchi}.
\newblock {Construction of traveling clusters in the Hamiltonian mean-field
  model by nonequilibrium statistical mechanics and Bernstein-Greene-Kruskal
  waves.}
\newblock {\em Phys Rev E Stat Nonlin Soft Matter Phys.}, 84, July 2011.

\end{thebibliography}

\end{document}